\RequirePackage{etex}
\documentclass[a4paper, twoside, 11pt]{amsart}

\usepackage[margin=3cm]{geometry}                 
\usepackage{graphicx}
\usepackage{amssymb}
\usepackage{mathrsfs}
\usepackage{enumitem}
\usepackage{array,multirow}
\usepackage{mathtools}
\usepackage[usenames,dvipsnames]{xcolor}
\usepackage{dsfont}
\usepackage{calligra}
\DeclareMathAlphabet{\mathcalligra}{T1}{calligra}{m}{n}
\usepackage{tikz}
\usepackage{tikz-cd}
\usetikzlibrary{shapes.misc}
\usetikzlibrary{decorations.markings}
\usepackage{etoolbox}
\usepackage{float}
\usepackage{multicol}
\usepackage{rotating}
\usepackage{arydshln}
\usetikzlibrary{positioning}
\usepackage{caption}
\usepackage{subcaption}
\usepackage{mathdots}
\usepackage{threeparttable}
\usepackage{dashbox}

\usetikzlibrary{calc,automata,patterns,decorations,decorations.pathmorphing}
\usetikzlibrary{fadings}

\usepackage{hyperref}
\hypersetup{
breaklinks=true,
    colorlinks=true,
    citecolor=blue,
    linkcolor=blue,
    filecolor=magenta,      
    urlcolor=black,
    }
\usepackage{autonum}

\usepackage{soul}

\makeatletter

\newcommand{\Rmnum}[1]{\expandafter\@slowromancap\romannumeral #1@}
\makeatother

\newtheorem{theorem}{Theorem}[section]
\newtheorem{lemma}[theorem]{Lemma}

\newtheorem{proposition}[theorem]{Proposition}
\newtheorem{corollary}[theorem]{Corollary}
\newtheorem{example}[theorem]{Example}
\newtheorem{remark}[theorem]{Remark}
\newtheorem{definition}[theorem]{Definition}

\makeatletter
\def\@tocline#1#2#3#4#5#6#7{\relax
  \ifnum #1>\c@tocdepth 
  \else
    \par \addpenalty\@secpenalty\addvspace{#2}%
    \begingroup \hyphenpenalty\@M
    \@ifempty{#4}{%
      \@tempdima\csname r@tocindent\number#1\endcsname\relax
    }{%
      \@tempdima#4\relax
    }%
    \parindent\z@ \leftskip#3\relax \advance\leftskip\@tempdima\relax
    \rightskip\@pnumwidth plus4em \parfillskip-\@pnumwidth
    #5\leavevmode\hskip-\@tempdima
      \ifcase #1
       \or\or \hskip 1em \or \hskip 2em \else \hskip 3em \fi%
      #6\nobreak\relax
    \dotfill\hbox to\@pnumwidth{\@tocpagenum{#7}}\par
    \nobreak
    \endgroup
  \fi}
\makeatother

\title[]{Models of $2$-nondegenerate CR hypersurface in $\mathbb{C}^N$}

\author{Jan Gregorovi\v{c}}
\address{Jan Gregorovi\v{c},
	 University of Ostrava, 701 03 Ostrava, Czech Republic, and Institute of Discrete Mathematics and Geometry, TU Vienna, Wiedner Hauptstrasse 8-10/104, 1040 Vienna, Austria}\email{ jan.gregorovic@seznam.cz}

\author{Martin Kol\'{a}\v{r}}
\address{Martin Kol\'{a}\v{r},
	Department of Mathematics and Statistics 
	Masaryk University,
	Kotl\'{a}\v{r}sk\'{a}~2,
	611 37 Brno,
	Czech Republic}\email{ mkolar@math.muni.cz}
\urladdr{\url{http://www.math.muni.cz/~mkolar}}

 \author{David Sykes}
\address{David Sykes,
	Department of Mathematics and Statistics, 
	Masaryk University,
	Kotl\'{a}\v{r}sk\'{a} 2,
	611 37 Brno,
	Czech Republic, and Institute of Discrete Mathematics and Geometry, TU Vienna, Wiedner Hauptstrasse 8-10/104, 1040 Vienna}\email{ sykes@math.muni.cz}

\makeatletter
\@namedef{subjclassname@2020}{\textup{2022} Mathematics Subject Classification}
\makeatother
\thanks{The first and the third author were supported by Austrian Science Fund (FWF): P34369. The second author was supported by the GACR grant GC22-15012J}

\begin{document}
\begin{abstract}
We show that every point in a uniformly $2$-nondegenerate CR hypersurface is canonically associated with a model $2$-nondegenerate structure. The $2$-nondegenerate models are basic CR invariants playing the same fundamental role as quadrics do in the Levi nondegenerate case. We characterize all $2$-nondegenerate models and show that the moduli space of such hypersurfaces in $\mathbb{C}^N$ is infinite dimensional for each $N>3$. We derive a normal form for these models' defining equations that is unique up to an action of a finite dimensional Lie group. We generalize recently introduced CR invariants termed modified symbols, and show how to compute these intrinsically defined invariants from a model's defining equation. We show that these models automatically possess infinitesimal symmetries spanning a complement to their Levi kernel and derive explicit formulas for them.
 \end{abstract}
\maketitle
\tableofcontents
\section{Introduction}
In classical studies of Levi nondegenerate CR hypersurfaces, the model quadrics play a fundamental role  \cite{cartanCR,chern1974,tanaka1962pseudo}. General Levi nondegenerate CR hypersurfaces locally appear as deformations of model hypersurfaces and many of their geometric properties (such as the structure and jet determinacy of their symmetries) are related to those of the models  \cite{baouendi1999real}. Beyond the Levi nondegenerate hypersurface case, the theory of models is also developed in \cite{beloshapka2004universal} for higher codimension Levi-nondegenerate CR submanifolds,  in \cite{ebenfelt2001uniformly,kolar2019complete} for Levi degenerate hypersurfaces in $\mathbb{C}^3$, and for finite Catlin multitype hypersurfaces in \cite{kmz2014,kolavr2010catlin}.

In this article, we derive defining equations of models for everywhere 2-nondegenerate hypersurfaces in $\mathbb{C}^{n+1}$ with constant rank Levi forms (shortly \emph{uniformly $2$-nondegenerate CR hypersurfaces}), study their local equivalence problem -- reducing it to finite-dimensional linear algebra -- and further investigate their symmetries and other fundamental features of their local geometry. 

Here and throughout the sequel, \emph{$2$-nondegenerate model} refers to a real analytic CR hypersurface
\begin{align}\label{gen def fun}
\Re(w)=z^TH(\zeta,\overline{\zeta})\overline{z}+\Re(\overline{z}^TS(\zeta,\overline{\zeta})\overline{z})
\end{align}
with constant rank $s$ Levi form that is $2$-nondegenerate everywhere in some neighborhood of $0$, where we use coordinates $(w,z,\zeta)\in\mathbb{C}\oplus\mathbb{C}^s\mathbb\oplus{C}^{n-s}$ and $H(\zeta,\overline{\zeta})$ and $S(\zeta,\overline{\zeta})$ are respectively $s\times s$ Hermitian and symmetric matrix valued functions of $(\zeta,\overline{\zeta})$ with $H(0,0)$ nondegenerate. This is terminology justified by the role of model structures in the classical theory and by the following three-part result.

\begin{theorem}\label{model perturbations proposition}
First, if a real analytic hypersurface $M\subset\mathbb{C}^{n+1}$ has constant rank $s$ Levi form, then it is locally given by a defining equation of the form 
\[
\Re(w)=P(z,\zeta,\overline{z},\overline{\zeta})+Q(z,\zeta,\overline{z},\overline{\zeta},\Im(w)),
\]
where $Q(z,\zeta,\overline{z},\overline{\zeta},\Im(w))$ is a series whose terms all have weighted degree higher than $2$ with respect to the weights
\begin{align}\label{grading convention}
\mathrm{wt}(w)=2,
\quad \mathrm{wt}(z_j)=1,
\quad\mbox{and}\quad
\mathrm{wt}(\zeta_\alpha)=0
\quad\forall\,j\in\{1,\ldots, s\},\,\alpha\in\{1,\ldots, n-s\},
\end{align}
and $\Re(w)=P(z,\zeta,\overline{z},\overline{\zeta})$ is in the form \eqref{gen def fun} and defines a hypersurface $M_0\subset\mathbb{C}^{n+1}$ with constant rank $s$ Levi form. 
Moreover, if $M$ is everywhere $2$-nondegenerate, then $M_0$ is a $2$-nondegenerate model. 

Second, if $\phi: \mathbb{C}^{n+1}\to \mathbb{C}^{n+1}$  is a local biholomorphism preserving the origin with $\phi(M)= M^\prime$, where $M^\prime$ is given by $\Re(w)=P^\prime(z,\zeta,\overline{z},\overline{\zeta})+Q^\prime(z,\zeta,\overline{z},\overline{\zeta},\Im(w))$ with corresponding model $M_0^\prime=\{\Re(w)=P^\prime(z,\zeta,\overline{z},\overline{\zeta})\}$, then $\phi_0(M_0)=M_0^\prime$, where $\phi_0$ is the weighted degree preserving component of $\phi$. 

Third, $\dim(\mathfrak{hol}(M,0))\leq \dim(\mathfrak{hol}(M_0,0))$ holds for the Lie algebras of germs of holomorphic infinitesimal symmetries at $0$ of $M$ and $M_0.$
\end{theorem}

Thus every point in a real analytic uniformly $2$-nondegenerate CR hypersurface is canonically associated with a $2$-nondegenerate model, which highlights the fundamental importance of $2$-nondegenerate models in the general study of $2$-nondegenerate structures.  We prove Theorem \ref{model perturbations proposition} in Section \ref{model perturbations section}. Within the proof, we establish that the constant rank $s$ Levi form  for CR hypersurface in the form \eqref{gen def fun} is equivalent to a solution of system of PDE
\begin{align}
H_{\zeta_\alpha\overline{\zeta_\beta}}(\zeta,\overline{\zeta})&=H_{\zeta_\alpha}(\zeta,\overline{\zeta})H(\zeta,\overline{\zeta})^{-1}H_{\overline{\zeta_\beta}}(\zeta,\overline{\zeta})+\overline{S_{\zeta_\beta}(\zeta,\overline{\zeta})}(H(\zeta,\overline{\zeta})^{-1})^TS_{\zeta_\alpha}(\zeta,\overline{\zeta})\label{rank condition 2},\\
S_{\zeta_{\alpha}\overline{\zeta_{\beta}}}(\zeta,\overline{\zeta})&=H_{\overline{\zeta_\beta}}(\zeta,\overline{\zeta})^T(H(\zeta,\overline{\zeta})^T)^{-1}S_{\zeta_\alpha}(\zeta,\overline{\zeta})+S_{\zeta_\alpha}(\zeta,\overline{\zeta})H(\zeta,\overline{\zeta})^{-1})H_{\overline{\zeta_\beta}}(\zeta,\overline{\zeta}).
\end{align}

As \eqref{rank condition 2} is a complicated system, deriving an explicit characterization of $2$-nondegenerate models is challenging. Our methods, however, provide the following characterization. 

\begin{theorem}\label{general formula for Cn}
Every $2$-nondegenerate model is locally equivalent to a real analytic CR hypersurface in $\mathbb{C}^{n+1}$ given by defining equation \eqref{gen def fun} with
\begin{align}\label{H from S2 formula}
H(\zeta,\overline{\zeta})=\frac12\left(\mathbf{H}(\mathrm{Id}-\overline{\mathbf{S}(\zeta)}\mathbf{H}^T\mathbf{S}(\zeta)\mathbf{H})^{-1}+(\mathrm{Id}-\mathbf{H}\overline{\mathbf{S}(\zeta)}\mathbf{H}^T\mathbf{S}(\zeta))^{-1}\mathbf{H}\right)
\end{align}
and 
\begin{align}\label{S from S2 formula}
S(\zeta,\overline{\zeta})=\mathbf{H}^T(\mathrm{Id}-\mathbf{S}(\zeta)\mathbf{H}\overline{\mathbf{S}(\zeta)}\mathbf{H}^T)^{-1}\mathbf{S}(\zeta)\mathbf{H},
\end{align}
where $\mathbf{H}$ is an $s\times s$ nondegenerate Hermitian matrix and $\mathbf{S}(\zeta)$ is an $s\times s$ symmetric matrix of holomorphic functions of $\zeta_1,\dots,\zeta_{n-s}$ that vanish at $0$, but with linearly independent partial derivatives $\mathbf{S}_{\zeta_1}(0),\dots, \mathbf{S}_{\zeta_{n-s}}(0)$.

Conversely, for any pair $\mathbf{H}$ and $\mathbf{S}(\zeta)$ satisfying the above assumptions, the defining equation \eqref{gen def fun} with \eqref{H from S2 formula} and \eqref{S from S2 formula} defines a $2$-nondegenerate model, and moreover a subset of its holomorphic infinitesimal symmetries acts locally transitively at $0$ on the space of leaves of the Levi kernel.
\end{theorem}

\begin{remark}
   In the case of pseudoconvex $2$-nondegenerate models, we can take $\mathbf{H}$ to be the identity matrix and the pseudoconvex $2$-nondegenerate models are given by
   \[
   \Re(w)=z^T(\mathrm{Id}-\overline{\mathbf{S}(\zeta)}\mathbf{S}(\zeta))^{-1}\overline{z}+\Re\left(\overline{z}^T(\mathrm{Id}-\mathbf{S}(\zeta)\overline{\mathbf{S}(\zeta)})^{-1}\mathbf{S}(\zeta)\overline{z}\right)
   \]
   for a symmetric matrix $\mathbf{S}(\zeta)$ of holomorphic functions of $\zeta_1,\dots,\zeta_{n-s}$ that vanish at $0$ with linearly independent partial derivatives $\mathbf{S}_{\zeta_1}(0),\dots, \mathbf{S}_{\zeta_{n-s}}(0)$.
\end{remark}

An immediate consequence of this classification of 2-nondegenerate models is that in general they lack two of the properties characterizing models of Levi nondegenerate hypersurfaces in \cite{beloshapka2004universal}: The polynomiality of the model (which was already observed for the $2$-nondegenerate hypersurfaces in $\mathbb{C}^3$, see \cite{beloshapka2005symmetries,ebenfelt2001uniformly,fels2007cr,kaup2006local}) and the homogeneity of the model with respect to CR symmetries (there do not have to be symmetries in the Levi kernel directions). Broadening this result, the theorem's real analyticity assumption can be replaced by an assumption about the structure's infinitesimal symmetries due to the following theorem.
\begin{theorem}\label{smooth to analytic}
If a uniformly $2$-nondegenerate CR hypersurface $M$ in $\mathbb{C}^{n+1}$ given by \eqref{gen def fun}
\begin{itemize}
\item is smooth
\item and has a subset of its holomorphic infinitesimal symmetries acting locally transitively at $0$ on the space of leaves of the Levi kernel,
\end{itemize}
then it is locally equivalent to a $2$-nondegenerate model given by \eqref{gen def fun}, \eqref{H from S2 formula}, and \eqref{S from S2 formula}.
\end{theorem}

\begin{remark}
    It is an open problem if the additional assumption on the symmetries is necessary in the smooth category. A related problem  about obtaining a $2$-nondegenerate model from the power series expansion of a smooth hypersurface is if the series $P(z,\zeta,\overline{z},\overline{\zeta})$ from Theorem \ref{model perturbations proposition} in smooth case is convergent.
\end{remark}

We prove these theorems concurrently, presenting the proof in the following structure. We show in Section \ref{Symmetries transversal to the Levi kernel} that real analytic uniformly $2$-nondegenerate CR hypersurfaces in $\mathbb{C}^{n+1}$ given by \eqref{gen def fun} also satisfy the itemized properties of Theorem \ref{smooth to analytic}. In particular, in Proposition \ref{transkernel symmetry formula}, we provide explicit formulas for the holomorphic infinitesimal symmetries referred to in Theorem \ref{smooth to analytic}. We establish the second statement in Theorem \ref{general formula for Cn} and prove Theorem \ref{smooth to analytic} in Section  \ref{Theorem 1.1 proof}, which implies the first statement in Theorem \ref{general formula for Cn}.  Lastly an alternative proof of Theorem \ref{general formula for Cn} using an analysis of the defining equation's series expansion is outlined in Section  \ref{The equivalence problem for models}.

The pair $\mathbf{H}$ and $\mathbf{S}(\zeta)$ defining a $2$-nondegenerate model is not unique and we show the equivalence problem for these pairs reduces to the following.

\begin{proposition}\label{equivalent symbol data}
    Let $M_0$ and $\tilde{M_0}$ be the $2$-nondegenerate models associated with the pairs $\mathbf{H},\mathbf{S}(\zeta)$ and $\tilde{\mathbf{H}},\tilde{\mathbf{S}}(\zeta)$ according to Theorem \ref{general formula for Cn}. Then $M_0$ and $\tilde{M_0}$ are equivalent around $0$ if and only if there is $U\in \mathrm{GL}(s,\mathbb{C})$ and a biholomorphism $g:\mathbb{C}^{n-s}\to \mathbb{C}^{n-s}$ fixing $0$ such that
    \begin{align}
        \tilde{\mathbf{H}}&=U^T\mathbf{H}\overline{U},\\
        \tilde{\mathbf{S}}(\zeta)&=U^T\mathbf{S}(g(\zeta))U.
    \end{align}
\end{proposition}

We prove the Proposition \ref{equivalent symbol data} in Section \ref{The equivalence problem for models} (as a corollary of Theorem \ref{normal form for models}). It turns out that one needs to have a good understanding of CR invariants for uniformly $2$-nondegenerate CR hypersurfaces to solve this equivalence problem and provide the normal forms for the $2$-nondegenerate models. 

There are three basic invariants that we consider in this paper:
\begin{itemize}
\item There are \emph{bigraded CR symbols} introduced in \cite{porter2021absolute} encoded in $\mathbf{H}$ and the part of $\mathbf{S}(\zeta)$ that is linear in $\zeta$. In Sections \ref{sec bigraded symbol}, we recall their definition and properties. This section also introduces several preliminary results and notations used throughout the article.

\item There are \emph{modified CR symbols} introduced in \cite{sykes2023geometry} encoded in  $\mathbf{H}$ and the parts of $\mathbf{S}(\zeta)$ that are linear and quadratic in $\zeta$.  In Section \ref{sec2.2}, we describe them using a new formulation that is well suited for an important generalization, and in Section \ref{sectionnonconstantmodifedsymbol} we introduce this generalized definition of modified symbols (Definition \ref{normalized sections}) for the setting with nonconstant bigraded symbol.

\item There are \emph{obstructions to the first order constancy} (of the bigraded symbol) that we derive in Section \ref{sectionnonconstantmodifedsymbol} (Definition \ref{normalized sections}). These are also encoded in  $\mathbf{H}$ and the parts of $\mathbf{S}(\zeta)$ that are linear and quadratic in $\zeta$.
\end{itemize}

In Section \ref{Key examples of modified symbol calculations}, we compute these invariants for the general hypersurfaces of Theorem \ref{general formula for Cn} and in Section \ref{The equivalence problem for models} we obtain their relation to $\mathbf{H}$ and $\mathbf{S}(\zeta)$.

The reader can find examples and applications of the general theory developed in this article for the particular case of $2$-nondegenerate models in $\mathbb{C}^4$ in \cite{c4paper}.

\subsection{Realizability of modified symbols}

 Given bigraded symbols and modified symbols understood from the abstract viewpoint (see Section \ref{Abstract modified symbols}), we would like to know which of these can be realized by uniformly $2$-nondegenerate CR hypersurfaces. Theorem \ref{model perturbations proposition} reduces the problem to realizibility by $2$-nondegenerate models, because we show in Proposition \ref{mod symbol the same} that a uniformly $2$-nondegenerate CR hypersurfaces has the same modified symbol at $0$ as its $2$-nondegenerate model.

It turns out that not all (abstract) modified symbols can be realized. In Section \ref{Key examples of modified symbol calculations}, we derive necessary conditions (Lemma \ref{key examples prop}) for realizability. We then show that these conditions are sufficient in Section \ref{Approaches to Hypersurface Realization} by providing two constructions of $2$-nondegenerate models realizing these \emph{realizable} modified symbols of uniformly $2$-nondegenerate CR hypersurfaces via the particular form in Theorem \ref{general formula for Cn}. The two constructions have  complementary theoretical applications, where one is very useful for studying maximally symmetric structures while the other is vital for deriving the general defining equations of Theorem \ref{general formula for Cn}.

The first construction, presented in Section \ref{generalized Naruki}, generalizes a natural approach to build homogeneous models from CR algebras (referring to the terminology in \cite{fels2008classification}), in a way that generalizes the homogeneous CR hypersurface realizations described in \cite{naruki1970holomorphic}. 
While this approach always yields parameterizations of the embedded hypersurfaces, deriving defining equations from these parameterizations is not always tractable. An important feature of this construction is that it always produces homogeneous structures that are maximally symmetric (among homogeneous structures with the same reduced modified symbols), which we show in Lemma \ref{Z-graded model lemma}.

The second construction, presented in Theorem \ref{hypersurface realization approach 2}, uses roughly the same technique to build the models that was applied in \cite{gregorovic2021equivalence}. This technique can be applied to construct $2$-nondegenerate models with prescribed modified symbols at a point, and the initial data one needs to implement this construction is just a realizable modified symbol rather than the CR algebras needed in the former approach. However, cases with non-constant bigraded symbol will appear, which have obstructions to homogeneity (existence of symmetries in the Levi kernel direction).

\section{Invariants of CR hypersurfaces: bigraded symbols and modified symbols}\label{Symbols of CR hypersurfaces}
In  this section we describe several CR invariants, beginning with the \emph{bigraded symbols} introduced in \cite{porter2021absolute}, and develop a generalization (Definition \ref{normalized sections}) of the \emph{modified symbols} introduced in \cite{sykes2023geometry} for constant bigraded symbol structures. Our generalization is well defined also for CR hypersurfaces with non-constant bigraded symbol, which is the motivation for its development.

\subsection{Bigraded symbols}\label{sec bigraded symbol}

Let $M$ be a CR hypersurface of dimension $2n+1$ and $\mathcal{H}\subset \mathbb{C}TM$ the CR structure on $M$, that is, $\mathcal{H}$ is an $n$-dimensional complex, integrable (i.e., $[\mathcal{H},\mathcal{H}]\subset \mathcal{H}$) subbundle of $\mathbb{C}TM$ with   $\mathcal{H}\cap \overline{\mathcal{H}}=0$. We will assume that $(M,\mathcal{H})$ has a finite dimensional local symmetry algebra, which is equivalent to assuming that $M$ is not locally equivalent to a CR structure of the form $M^\prime\times \mathbb{C}$ for some CR structure $(M^\prime,\mathcal{H}^\prime)$. In the terminology of \cite{freeman1977local}, such $M$ are called non-straightenable and at a generic point their complexified tangent bundle $\mathbb{C}TM$ admits what we call a \emph{Freeman filtration} of the form
\begin{align}\label{freeman filtration}
\mathcal{K}_0=\mathbb{C}TM\supset \mathcal{K}_{1}=\mathcal{H}\supset \mathcal{K}_{2}\supset \mathcal{K}_{3}\supset\cdots\supset \mathcal{K}_{k+1}=0
\end{align}
for some integer $k$ such that $\mathcal{K}_{k}\neq \mathcal{K}_{k+1}$. For each $j>1$, assuming $\mathcal{K}_{j-1}$ has constant rank (which indeed occurs near generic points), the space $\mathcal{K}_{j}$ is defined as the left-kernel of the multi-sesquilinear map
\[
\mathcal{L}^{j-1}:\mathcal{K}_{j-1}\times \left(\bigodot^{j-1}\mathcal{H}\right)\to \mathcal{K}_0/(\mathcal{H}\oplus \overline{\mathcal{H}})
\]
given by
\[
\mathcal{L}^{j-1}\left(X_p,X^{(1)}_p,\ldots, X^{(j-1)}_p\right):=\frac{1}{2i}\left[\ldots\left[X,\overline{X^{(1)}}\right],\ldots,\overline{X^{(j-1)}}\right]_p\pmod{\mathcal{H}\oplus\overline{\mathcal{H}}}
\,\,\forall\,p\in M
\]
where $X$ and each $X^{(j)}$ are sections of $\mathcal{K}_{j-1}$ and $\mathcal{H}$ respectively, that is,
\begin{align}\label{freeman filtration levels}
\mathcal{K}_j:=\{v\in \mathcal{K}_{j-1}\,|\, \mathcal{L}^{j-1}(v,v_1,\ldots, v_{j-1})=0\,\forall\,v_k\in \mathcal{H}\}.
\end{align}
Indeed, $\mathcal{L}^{j-1}$ is a tensor near generic points because the differences in the Lie bracket values in its definition that would appear using different vector field extensions of the $X_p,X^{(1)}_p,\ldots, X^{(j-1)}_p$ arguments are contained in the subspace $\mathcal{H}\oplus\overline{\mathcal{H}}$.
\begin{definition}
The CR structure $(M,\mathcal{H})$ is uniformly $k$-nondegenerate if all $\mathcal{K}_j$ given by \eqref{freeman filtration levels} have constant rank globally and $\mathcal{K}_{k}\neq \mathcal{K}_{k+1}=0$.
\end{definition}

Note, a more general point-wise definition of $k$-nondegeneracy (defined at all points rather than just generic ones) is given in \cite{baouendi1999real}, and this relationship is discussed in \cite[appendix]{kaup2006local}. Notice that $\mathcal{L}^1$ is just the usual Levi form. Consistent with \cite{freeman1977local}, we refer to $\mathcal{L}^k$ as \emph{generalized Levi forms}. Notably, the generalized Levi forms are all invariants of the CR structure. The set $\{\mathcal{L}^1, \mathcal{L}^2\}$ is equivalent to the invariant termed \emph{CR symbol} in \cite{porter2021absolute} (defined for $2$-nondegenerate structures) and $\{\mathcal{L}^1,\ldots,\mathcal{L}^k\}$ is equivalent to the invariant termed \emph{abstract core} in \cite{santi2020homogeneous}.

First, let us define the nondegenerate part (or Heisenberg part) of the bigraded symbol.

\begin{definition}
The nondegenerate part $\mathbb{C}\mathfrak{g}_{-}(p)$ of the bigraded symbol of $M$ at $p$ is the bigraded Lie algebra
\[\mathbb{C}\mathfrak{g}_{-}(p)=\mathfrak{g}_{-2,0}(p)\oplus \mathfrak{g}_{-1,-1}(p)\oplus \mathfrak{g}_{-1,1}(p)\]
together with the antilinear involution $x\mapsto \overline{x},$
where
\[
\mathfrak{g}_{-2,0}(p):=\mathbb{C}T_pM/(\mathcal{H}_p\oplus \overline{\mathcal{H}}_p),
\,\,
\mathfrak{g}_{-1,-1}(p):=\overline{\mathcal{H}}_p/(\overline{\mathcal{K}_2})_p,
\,\,\mbox{ and }\,\,
\mathfrak{g}_{-1,1}(p):={\mathcal{H}}_p/(\mathcal{K}_2)_p,
\]
and we denote
\[
s:=dim_{\mathbb{C}}(\mathfrak{g}_{-1,1}(p)).
\] 
\end{definition}

\begin{remark}[Notation conventions A]
For a real vector space $V$ with $\mathbb{Z}$-graded decomposition $V=\bigoplus_{j\in \mathbb{Z}} V_j$, we denote its complexification (i.e., tensor product with $\mathbb{C}$) by $\mathbb{C}V=\bigoplus_{j\in \mathbb{Z}} \mathbb{C}V_j$, and label $V_{-}=\bigoplus_{j<0}V_{j}$ and $V_{\leq k}=\bigoplus_{j\leq k}V_{j}$. Conversely, for a complex vector space  $W$ with fixed antilinear involution $\sigma$ we label it's real part as $\Re W:=\{w\in W\,|\, w=\sigma(w)\}$.

In this paper, for special spaces $\mathfrak{g}$ we consider bigraded decomposition of $\mathbb{C}\mathfrak{g}$ with components that we denote by $\mathfrak{g}_{j,k}$, that is,
\[
\mathbb{C}\mathfrak{g}=\bigoplus_{j} \mathbb{C}\mathfrak{g}_{j}
\quad\mbox{where}\quad
\mathbb{C}\mathfrak{g}_j=\bigoplus_{k} \mathfrak{g}_{j,k},
\]
 and with the projections $\pi_{j,k}: \mathbb{C}\mathfrak{g}\to \mathfrak{g}_{j,k}$ along the other components.
In our setting complex conjugation always interchanges $\mathfrak{g}_{j,k}$ and $\mathfrak{g}_{j,-k}$. 
\end{remark}

We will fix the matrix representations of  the $(2s+1)$-dimensional Heisenberg algebra $\mathbb{C}\mathfrak{g}_-=\mathfrak{g}_{-2,0}\oplus \mathfrak{g}_{-1,-1}\oplus \mathfrak{g}_{-1,1}$ and its space of the first grading preserving derivations $\mathfrak{csp}(\mathbb{C}\mathfrak{g}_{-1})=\mathfrak{csp}(\mathbb{C}\mathfrak{g}_{-1})_{0,-2}\oplus \mathfrak{csp}(\mathbb{C}\mathfrak{g}_{-1})_{0,0}\oplus \mathfrak{csp}(\mathbb{C}\mathfrak{g}_{-1})_{0,2}$  as follows:

\begin{align}\label{csp representation}
\mathbb{C}\mathfrak{g}_{-}\oplus\mathfrak{csp}(\mathbb{C}\mathfrak{g}_{-1})=\left\{
\left.
\left(
\begin{array}{cccc}
c & 0 & 0 &  0 \\
v_1 & L & S^{0,2} &  0 \\
v_2 & S^{0,-2}& -L^T  & 0 \\
ui &  v_2^T& -v_1^T &  -c
\end{array}
\right)
\,\right|\,
\parbox{3.6cm}{$L,S^{0,\pm2}\in \mathfrak{gl}(s,\mathbb{C})$, $S^{0,\pm2}=(S^{0,\pm2})^T$, $v_1,v_2\in\mathbb{C}^{s}$, and $c,u\in\mathbb{C}$}
\right\}\\
=\left(
\begin{array}{cccc}
\mathfrak{csp}(\mathbb{C}\mathfrak{g}_{-1})_{0,0}& 0 & 0 &  0 \\
\mathfrak{g}_{-1,1}& \mathfrak{csp}(\mathbb{C}\mathfrak{g}_{-1})_{0,0} &\mathfrak{csp}(\mathbb{C}\mathfrak{g}_{-1})_{0,2}&  0 \\
\mathfrak{g}_{-1,-1} & \mathfrak{csp}(\mathbb{C}\mathfrak{g}_{-1})_{0,-2}& \mathfrak{csp}(\mathbb{C}\mathfrak{g}_{-1})_{0,0}  & 0 \\
\mathfrak{g}_{-2,0}&  \mathfrak{g}_{-1,-1}& \mathfrak{g}_{-1,1}&  \mathfrak{csp}(\mathbb{C}\mathfrak{g}_{-1})_{0,0}
\end{array}
\right).
\end{align}

Note that $\mathbb{C}\mathfrak{g}_{-}\oplus \mathfrak{csp}(\mathbb{C}\mathfrak{g}_{-1})$ is a bigraded Lie algebra and $\mathfrak{csp}(\mathbb{C}\mathfrak{g}_{-1})_{0,0}$ are the derivations in $\mathfrak{csp}(\mathbb{C}\mathfrak{g}_{-1})$ that preserve the bigrading. Therefore, the bundle $\mathcal{F}$ with fibers defined by
\begin{align}\label{F fiber}
\mathcal{F}_p&:=\left\{\tilde \phi:\mathbb{C}\mathfrak{g}_{-}(p)\to \mathbb{C}\mathfrak{g}_{-}\,\left|\,\parbox{7.2cm}{$\tilde \phi$ is a Lie algebra isomorphism preserving bigradings.}\right.\right\}\\
&\,\cong \left\{\tilde \phi:\mathbb{C}\mathfrak{g}_{-}(p)\oplus\mathfrak{csp}\left(\mathbb{C}\mathfrak{g}_{-1}(p)\right)\to \mathbb{C}\mathfrak{g}_{-}\oplus\mathfrak{csp}\left(\mathbb{C}\mathfrak{g}_{-1}\right)\,\left|\,\parbox{3.2cm}{$\tilde \phi$ is a Lie algebra\\ isomorphism pres-\\erving bigradings.}\right.\right\}
\end{align}
is a $\mathrm{CSp}(\mathbb{C}\mathfrak{g}_{-1})_{0,0}$-bundle over $M$, where $\mathrm{CSp}(\mathbb{C}\mathfrak{g}_{-1 })_{0,0}$ is the closed subgroup in $\mathrm{CSp}(\mathbb{C}\mathfrak{g}_{-1 })$ generated by $\mathfrak{csp}(\mathbb{C}\mathfrak{g}_{-1 })_{0,0}$. The bijection between the two sets in \eqref{F fiber} is obtained using the natural extension of an isomorphism $\tilde \phi:\mathbb{C}\mathfrak{g}_{-}(p)\to \mathbb{C}\mathfrak{g}_{-}$ to the larger domain $\mathbb{C}\mathfrak{g}_{-}(p)\oplus\mathfrak{csp}\left(\mathbb{C}\mathfrak{g}_{-1}(p)\right)$ given by defining $\tilde \phi(x):=\tilde\phi\circ x\circ\tilde\phi^{-1}\in \mathfrak{csp}\left(\mathbb{C}\mathfrak{g}_{-1}\right)$ for all $x\in \mathfrak{csp}\left(\mathbb{C}\mathfrak{g}_{-1}(p)\right)$.

The right principal action on $\mathcal{F}$ is given by
\[
(a,\tilde\phi)\mapsto \mathrm{Ad}_{a^{-1}}\circ \tilde\phi,
\]
where, in our notation, $\mathrm{CSp}(\mathbb{C}\mathfrak{g}_{-1})_{0,0}$ is represented by matrices 
\begin{align}\label{csp00matrix}
\left(
\begin{array}{cccc}
b& 0 & 0 &  0 \\
0 & B & 0 &  0 \\
0 & 0& (B^T)^{-1} & 0 \\
0 &  0& 0 &  b^{-1}
\end{array}
\right)\in \mathbb{C}^*\times \mathrm{GL}(s,\mathbb{C})/\mathbb{Z}_2\cong\mathrm{CSp}(\mathbb{C}\mathfrak{g}_{-1 })_{0,0}.
\end{align}
The $\mathbb{Z}_2$ quotient appears because the matrix with $b=-1$ and $B=-\mathrm{Id}$ has trivial adjoint action on the matrix representation of $\mathbb{C}\mathfrak{g}_{-}$ in \eqref{csp representation}.

\begin{lemma}\label{involution lemma}
For each $\tilde \phi \in \mathcal{F}_p,$ the complex conjugation on $\mathbb{C}\mathfrak{g}_{-}(p)$ induces a complex conjugation $\sigma_{\tilde \phi}$ of $\mathbb{C}\mathfrak{g}_{-}\oplus \mathfrak{csp}(\mathbb{C}\mathfrak{g}_{-1})$ given by $\sigma_{\tilde\phi}\circ\tilde\phi(x)=\tilde\phi(\overline{x})$ for $x\in \mathbb{C}\mathfrak{g}_{-}(p)$ and $\sigma_{\tilde\phi}(y)=\sigma_{\tilde\phi}\circ y \circ \sigma_{\tilde\phi}$ for $y\in \mathfrak{csp}(\mathbb{C}\mathfrak{g}_{-1})$, where $\circ$ is the composition of endomorphisms of $\mathbb{C}\mathfrak{g}_{-1}$. There is a Hermitian matrix $H\big(\tilde \phi\big)$ and unit complex number  $e^{ih\big(\tilde \phi\big)}$ such that
\begin{align}\label{involution formula}
&\sigma_{\tilde \phi}\left(
\begin{array}{cccc}
c & 0 & 0 &  0 \\
v_1 & L & S^{0,2} &  0 \\
v_2 & S^{0,-2}& -L^T  & 0 \\
ui &  v_2^T& -v_1^T &  -c
\end{array}
\right)=\\
=&\left(
\begin{array}{cccc}
\overline{c}& 0 & 0 &  0 \\
e^{ih\big(\tilde \phi\big)}(H\big(\tilde \phi\big)^T)^{-1}\overline{v_2} & -(H(\tilde \phi\big)^T)^{-1}L^*H(\tilde \phi\big)^T & (H(\tilde \phi\big)^T)^{-1}\overline{S^{0,-2}}H(\tilde \phi\big)^{-1} &  0 \\
e^{ih\big(\tilde \phi\big)}H\big(\tilde \phi\big)\overline{v_1} & H(\tilde \phi\big)\overline{S^{0,2}}H^T(\tilde \phi\big) & H(\tilde \phi\big)\overline{L}H(\tilde \phi\big)^{-1} & 0 \\
e^{2ih\big(\tilde \phi\big)}\overline{u}i &  \overline{v_1}^TH\big(\tilde \phi\big)^Te^{ih\big(\tilde \phi\big)}& -\overline{v_2} ^T H\big(\tilde \phi\big)^{-1}e^{ih\big(\tilde \phi\big)}&  -\overline{c}
\end{array}
\right).
\end{align}
The pair $(H\big(\tilde \phi\big),e^{ih\big(\tilde \phi\big)})$ is defined uniquely up to a sign and is related to the Levi form, as the pullback $(\tilde \phi^{-1})^* \mathcal{L}^1(p) \in\mathrm{Hom}\left(\mathfrak{g}_{-1,1}\times\mathfrak{g}_{-1,1},\mathfrak{g}_{-2,0}\right)$ is the map
\[
(v,v^\prime)\mapsto \left(
\begin{array}{cccc}
0& 0 & 0 &  0 \\
0& 0 & 0 &  0 \\
0 & 0& 0  & 0 \\
e^{ih\big(\tilde\phi\big)}v^TH\big(\tilde\phi\big)\overline{v^{\prime}}i &  0& 0&  0
\end{array} \right)
\]
for $v,v^{\prime}\in \mathfrak{g}_{-1,1}.$
Moreover, the pair $(H\big(\tilde \phi\big),e^{ih\big(\tilde \phi\big)})$ transforms under the action of element \eqref{csp00matrix} to 
\begin{align}\label{change in H}
\left( B^TH\big(\tilde \phi\big)\overline{B},b\overline{b}^{-1}e^{ih\big(\tilde \phi\big)}\right).
\end{align}
\end{lemma}
\begin{remark}\label{sign of H and h}
Note, that if the signature $H\big(\tilde \phi\big)$ is not split, then fixing the signature fixes the pair uniquely. Nevertheless, the reader can check that formulas in the paper usually do not depend on this choice. Also note that the form of $(\tilde \phi^{-1})^* \mathcal{L}^1(p)$ explains the choice to use $ui$ instead of $u$ in \eqref{csp representation} and is related to our convention (using $\Re(w)$ instead of $\Im(w)$) for defining equations \eqref{gen def fun}. 
\end{remark}
\begin{proof}[Proof (of Lemma \ref{involution lemma})]
The involution $x\mapsto \overline{x}$ is a Lie algebra automorphism that reverses the sign of the second weight in the bigraded decomposition of $\mathbb{C}\mathfrak{g}_{-}(p)$, and since $\tilde\phi$ is bigrading preserving, $\sigma_{\tilde\phi}$ will reverse the second grading of $\mathbb{C}\mathfrak{g}_{-}$. Accordingly, it can be represented by the adjoint action of some matrix
\begin{align}\label{CSP mat a}
\left(
\begin{array}{cccc}
\frac{1}{a}& 0 & 0 &  0 \\
0 & 0 & A_1 &  0 \\
0 & A_2& 0 & 0 \\
0 &  0& 0 &  a
\end{array}\right)\in \mathbb{C}^*\times \mathrm{Sp}(2s,\mathbb{C})/\mathbb{Z}_2\cong\mathrm{CSp}(\mathbb{C}\mathfrak{g}_{-1 })
\end{align}
pre-composed with complex conjugation in the matrix coordinates, that is, 
\begin{align}
\bgroup
\arraycolsep=2pt
\sigma_{\tilde \phi}\left(
\begin{array}{cccc}
c & 0 & 0 &  0 \\
v_1 & L & S^{0,2} &  0 \\
v_2 & S^{0,-2}& -L^T  & 0 \\
ui &  v_2^T& -v_1^T &  -c
\end{array}
\right)=\left(
\begin{array}{cccc}
\frac{\overline{c}}{a\overline{a}}& 0 & 0 &  0 \\
aA_1\overline{v_2} & -A_1L^*\overline{A_2} & A_1\overline{S^{0,-2}A_1} &  0 \\
aA_2\overline{v_1} & A_2\overline{S^{0,2}A_2}& A_2\overline{LA_1} & 0 \\
-a^2\overline{u}i &  a\overline{v_1}^TA_2^T& -a\overline{v_2}^TA_1^T&  -\overline{c}a\overline{a}
\end{array}
\right).
\egroup
\end{align}
 and $A_1^TA_2=-\mathrm{Id}$. The matrix in \eqref{CSP mat a}, and therefore the triple $(a,A_1,A_2)$, is defined up to rescaling by $-1$, corresponding to $\mathbb{Z}_2$ action in \eqref{CSP mat a}. Since $\sigma_{\tilde \phi}^2=\mathrm{Id}$, we see that $a^2\overline{a^2}=1$ and $a\overline{a}A_2\overline{A_1}=\mathrm{Id}.$ Thus $a\overline{a}=1$ and $\overline{A_2}=-A_2^T$. Therefore, there exists a Hermitian matrix $H\big(\tilde \phi\big)$ 
\begin{align}\label{H phi def}
A_2=iH\big(\tilde \phi\big)
\quad\mbox{ and }\quad
A_1=i(H\big(\tilde \phi\big)^T)^{-1}.
\end{align}
 Taking $a=- i e^{ih(\phi)}$ provides the claimed formula for $\sigma_{\tilde \phi}.$ 
 Notice that the sign ambiguity for $(H\big(\tilde \phi\big),e^{ih\big(\tilde \phi\big)})$ corresponds to the above $\mathbb{Z}_2$-action.

 Lastly, the transformation rule for pairs $(H\big(\tilde \phi\big),e^{ih\big(\tilde \phi\big)})$  is straightforward to calculate using the above formula for $(\tilde \phi^{-1})^* \mathcal{L}^1(p)$.
\end{proof}

A consequence of formula \eqref{involution formula} is that fixing a second grading reversing antilinear involution $\sigma$ of $\mathbb{C}\mathfrak{g}_-$ is equivalent to fixing $(H\big(\tilde \phi\big),e^{ih\big(\tilde \phi\big)})$ to some particular value $(\mathbf{H},e^{ih})$ up to sign. Then we can form a subbundle of $\mathcal{F}_{e^{ih}\mathbf{H}}$ of $\mathcal{F}$ consisting of $\tilde \phi$ realizing this value, i.e.,
\[
\mathcal{F}_{e^{ih}\mathbf{H}}:=\left\{\left.\tilde \phi\in\mathcal{F}\,\right|\, h\big(\tilde \phi\big)H\big(\tilde \phi\big)=e^{ih}\mathbf{H} \right\}.
\]
Note that elements in $\mathcal{F}_{e^{ih}\mathbf{H}}$ are precisely the bigraded Lie algebra isomorphisms intertwining the corresponding antilinear involutions on each $\mathfrak{g}_{-}(p)$ with fixed $\sigma:=\sigma_{\tilde \phi}$ that is independent of $\tilde \phi \in \mathcal{F}_{e^{ih}\mathbf{H}}$ and that 
$\mathcal{F}_{e^{ih}\mathbf{H}}$ is a principal bundle with structure group
\[
\mathrm{CU}\left(e^{ih}\mathbf{H}\right):=\left\{(b,B)\in \mathrm{CSp}(\mathbb{C}\mathfrak{g}_{-1})_{0,0}\,\left|\,b\overline{b}^{-1}e^{ih}B^T\mathbf{H}\overline{B}=e^{ih}\mathbf{H}\right.\right\}.
\]
\begin{definition}\label{adapted frame}
A basis $\phi$ of $\mathbb{C}T_pM$ is \emph{adapted} if it has the form
\begin{align}\label{adapted basis}
    \phi=\left(g,f_1,\dots,f_s,\overline{f_1},\dots, \overline{f_s},e_1,\dots,e_{n-s},\overline{e_1},\dots,\overline{e_{n-s}}\right),
\end{align}
where $(g)$, $(f_1,\ldots, f_s)$, and $(e_1,\ldots, e_{n-s})$ represent bases of $\mathbb{C}\mathfrak{g}_{-2,0}(p)$, $\mathfrak{g}_{-1,1}(p)$, and $\mathcal{K}_2$ respectively.
For such $\phi$, we label the \emph{corresponding graded basis} of the nondegenerate part $\mathbb{C}\mathfrak{g}_{-}(p)$ of the bigraded symbol as $\phi_{\mathrm{gr}}:=(g,f_1,\dots,f_s,\overline{f_1},\dots, \overline{f_s})$.
\end{definition}

Now, $\tilde \phi \in \mathcal{F}_p$ defines a unique basis $(g,f_1,\dots,f_s,h_1,\dots, h_s)$ of $\mathbb{C}\mathfrak{g}_{-}(p)$ such that $\tilde \phi(ug+\sum_{j=1}^s(v_jf_j+w_jh_j)=ui+v+w\in \mathbb{C}\mathfrak{g}_-$ for $ui\in 
 \mathfrak{g}_{-2,0}$, $v\in \mathfrak{g}_{-1,1}$ and $w\in \mathfrak{g}_{-1,-1}$, and it is clear from the action of $\mathrm{CSp}(\mathbb{C}\mathfrak{g}_{-1 })_{0,0}$ that points of $\mathcal{F}_p$ are in bijective correspondence with all possible bases $(g,f_1,\dots,f_s)$ of $\mathfrak{g}_{-2,0}(p)\oplus\mathfrak{g}_{-1,1}(p)$. Therefore, there is a bijective correspondence between $\tilde \phi \in \mathcal{F}_p$ and the graded bases $\phi_{\mathrm{gr}}:=(g,f_1,\dots,f_s,\overline{f_1},\dots, \overline{f_s})$ of the nondegenerate part of the bigraded symbol of $M$ at $p.$ Moreover, if $\phi$ and $\psi$ are two adapted bases of $\mathbb{C}T_pM$ such that $\phi_{\mathrm{gr}}=\psi_{\mathrm{gr}}$ then the $\tilde \phi=\tilde \psi$ holds for the corresponding elements of $\mathcal{F}_p$.

 \begin{remark}[Notation conventions B]\label{phi notation convention} 
 In the article, we reserve the respective labels $\phi$, $\phi_{\mathrm{gr}}$, and $\tilde \phi$ for adapted bases, their corresponding graded bases, and the corresponding element of $\mathcal{F}$. In sequel, referencing a $\tilde\phi$ without having fixed a $\phi$ to which it corresponds indicates that the presented result is independent on the corresponding adapted basis $\phi$.
\end{remark}
 
 The explicit relation between $\phi_{\mathrm{gr}}$ and the aforementioned $(g,f_1,\dots,f_s,h_1,\dots, h_s)$ corresponding to a given $\phi$ follows directly from the formula for the complex conjugation in Lemma \ref{involution lemma}:
\begin{corollary}[corollary of Lemma \ref{involution lemma}]\label{F bundle dual descriptions}
For $\tilde \phi \in \mathcal{F}_p$, the corresponding bases $(g,f_1,\dots,f_s,\allowbreak h_1,\dots, h_s)$  and $\phi_{\mathrm{gr}}=(g,f_1,\dots,f_s,\overline{f_1},\dots, \overline{f_s})$ are related by
\[
\left(
\begin{array}{c}
     h_1  \\
     \vdots\\
     h_s
\end{array}\right)
=
e^{ih\big(\tilde \phi\big)}H\big(\tilde \phi\big)
\left(
\begin{array}{c}
     \overline{f_1}  \\
     \vdots\\
     \overline{f_s}
\end{array}\right).
\]
If we write $(u,v,v^{\prime})\in \mathbb{C}\oplus \mathbb{C}^s \oplus \mathbb{C}^s$ for the coordinates in the basis $\phi_{\mathrm{gr}}$, then
\begin{align}\label{tilde phi minus formula}
\tilde{\phi}(u,v,v^{\prime})=\left(
\begin{array}{cccc}
0& 0 & 0 &  0 \\
v & 0 & 0 &  0 \\
e^{ih\big(\tilde \phi\big)}H\big(\tilde \phi\big)v^{\prime} & 0& 0  & 0 \\
ui &  (v^{\prime})^TH\big(\tilde \phi\big)^Te^{ih\big(\tilde \phi\big)}& -v^T &  0
\end{array}
\right).
\end{align}
Moreover, $\tilde \phi$ extends to $\mathfrak{csp}(\mathbb{C}\mathfrak{g}_{-1}(p))$ by mapping  matrix $L$ representing the coordinates in the basis $(f_i\otimes f_j^*)_{i,j=1..s}$, matrix $\Xi$ representing the coordinates in $(f_i\otimes \overline{f_j}^*)_{i,j=1..s}$, matrix $\overline{\Xi}$ representing the coordinates in $(\overline{f_i}\otimes f_j^*)_{i,j=1..s}$ and $c$ representing the multiple of $\mathrm{Id}$ to
\begin{align}\label{tilde phi coordinates}
\tilde{\phi}(c,L,\Xi,\overline{\Xi})=\left(
\begin{array}{cccc}
-c& 0 & 0 &  0 \\
0 & L & e^{-ih\big(\tilde \phi\big)} \Xi H\big(\tilde \phi\big)^{-1} &  0 \\
0 & e^{ih\big(\tilde \phi\big)} H\big(\tilde \phi\big)\overline{\Xi} &-L^T & 0 \\
0 &  0& 0 &  c
\end{array}
\right).
\end{align}
\end{corollary}

The Levi kernel fibers $\mathcal{K}_p$ are embedded into $\mathfrak{csp}(\mathbb{C}\mathfrak{g}_{-1}(p))$ via the embedding $\iota:\mathcal{K}_p\to \mathfrak{csp}(\mathbb{C}\mathfrak{g}_{-1}(p))_{0,2}$ given by
\begin{align}\label{L kernel csp embedding}
\iota(v)(w):=
\begin{cases}
0 & \mbox{ if }w\in \mathfrak{g}_{-1,1}(p)\\
[V,W]_p+(\mathcal{K}_2)_p+\overline{\mathcal{H}}_p & \mbox{ if }w\in \mathfrak{g}_{-1,-1}(p)
\end{cases}
\end{align}
where $V$ and $W$ are sections of $\mathcal{K}_2$ and $\overline{\mathcal{H}}$ respectively satisfying $V_p\equiv v\pmod{\mathcal{K}_2}$ and $W_p\equiv w\pmod{\overline{\mathcal{K}_2}}$.  It is shown in \cite{porter2021absolute} that $\iota$ is indeed well defined by \eqref{L kernel csp embedding} independent from the choice of vector fields $V$ and $W$. The value $\iota(v)(w)$, formally defined as an element in $\mathcal{H}_p\oplus \overline{\mathcal{H}}_p/((\mathcal{K}_2)_p\oplus \overline{\mathcal{H}}_p)$, is naturally identified with a vector in $\mathcal{H}_p/(\mathcal{K}_2)_p\subset \mathbb{C}\mathfrak{g}_{-1}(p)$, thereby defining an element in $\mathfrak{gl}(\mathbb{C}\mathfrak{g}_{-1}(p))$, which one can furthermore show belongs to $\mathfrak{csp}(\mathbb{C}\mathfrak{g}_{-1}(p))$. We extend $\iota$ to an embedding $\iota:(\mathcal{K}_2)_p\oplus (\overline{\mathcal{K}_2})_p\to \mathfrak{csp}(\mathbb{C}\mathfrak{g}_{-1}(p))_{0,2}\oplus \mathfrak{csp}(\mathbb{C}\mathfrak{g}_{-1}(p))_{0,-2}$ using the rule
\[
\iota(v):=\overline{\iota(\overline{v})}
\quad\quad\forall\, v\in  (\overline{\mathcal{K}_2})_p,
\]
where the complex conjugation $v\mapsto \overline{v}$ is defined on $\mathfrak{csp}(\mathbb{C}\mathfrak{g}_{-1}(p))$ via the natural extension of conjugation on $\mathbb{C}\mathfrak{g}_{-1}(p)$ to all of $\mathbb{C}\mathfrak{g}_{-}(p)\oplus \mathfrak{csp}(\mathbb{C}\mathfrak{g}_{-1}(p))$ (i.e., the unique conjugation such that $\overline{[v,w]}=[\overline{v},\overline{w}]$ on the larger algebra). 

\begin{definition}[introduced in \cite{porter2021absolute}]\label{CR symbol}
The \emph{bigraded symbol} $\mathbb{C}\mathfrak{g}_{\leq 0}(p)$ of $M$ at $p$ is the bigraded vector space 
\begin{align}\label{symbol decomposition}
\mathbb{C}\mathfrak{g}_{\leq 0}(p)&:=\mathfrak{g}_{-2,0}(p)\oplus\mathfrak{g}_{-1,-1}(p)\oplus\mathfrak{g}_{-1,1}(p)\oplus \mathfrak{g}_{0,-2}(p)\oplus \mathfrak{g}_{0,0}(p)\oplus \mathfrak{g}_{0,2}(p)\\&=\mathbb{C}\mathfrak{g}_{-}(p)\oplus \mathfrak{g}_{0,-2}(p)\oplus \mathfrak{g}_{0,0}(p)\oplus \mathfrak{g}_{0,2}(p)
\end{align}
given by $\mathfrak{g}_{0,2}(p):=\iota(\mathcal{K}_2)_p$, $\mathfrak{g}_{0,-2}(p):=\iota(\overline{\mathcal{K}_2})_p$, and 
\[
\mathfrak{g}_{0,0}(p):=\{v\in \mathfrak{csp}(\mathbb{C}\mathfrak{g}_{-1}(p))\,|\, [v,w]\subset \mathfrak{g}_{j,k}(p)\,\forall\, w\in\mathfrak{g}_{j,k}(p),\,\forall (j,k)\in\mathcal{I} \},
\]
with $\mathcal{I}:=\{(-1,-1),(-1,1),(0,-2),(0,2)\}$, together with the antilinear involution given by restricting $v\mapsto \overline{v}$ from $\mathbb{C}\mathfrak{g}_{-}(p)\oplus \mathfrak{csp}(\mathbb{C}\mathfrak{g}_{-1}(p))$ to $\mathbb{C}\mathfrak{g}_{\leq 0}(p)$. We label 
\[s:=\dim_{\mathbb{C}}(\mathfrak{g}_{-1,-1})\quad {\rm and} \quad r:=\dim_{\mathbb{C}}(\mathfrak{g}_{0,-2}).\]
\end{definition}

\begin{example}\label{maximal kernel dimension}
Consider the bigraded symbol defined as the bigraded Lie algebra $\mathbb{C}\mathfrak{g}_{-}\oplus \mathfrak{csp}(\mathbb{C}\mathfrak{g}_{-1})$, i.e., $\mathfrak{g}_{0,j}:=\mathfrak{csp}(\mathbb{C}\mathfrak{g}_{-1})_{0,j}$  for $j\in\{-2,0,2\}$, equipped with an antilinear involution of the form \eqref{involution formula} with $(\mathbf{H},e^{ih})$ in place of $(H(\tilde \phi),e^{ih(\tilde \phi)})$ for any $\mathbf{H}$ nondegenerate Hermitian and $h\in\mathbb{R}$. Necessarily, any $2n+1=2(r+s)+1$ dimensional uniformly $2$-nondegenerate CR hypersurface for which
\[
r=\binom{s+1}{2}
\]
has one of these bigraded symbols at every point. In \cite[Section 5.7]{gregorovic2021equivalence}, homogeneous maximally symmetric CR hypersurfaces with these symbols are given by defining equations of the form \eqref{gen def fun}.
\end{example}

We can conclude the following:

\begin{proposition}\label{matrix representation prop}
Fix a point $p\in M$ and $\tilde \phi\in \mathcal{F}_p$ for a $(2n+1)$-dimensional uniformly $k$-nondegenerate CR hypersurface $M$.  The space $\tilde\phi(\mathfrak{g}_{0,2}(p))$ is spanned by matrices 
\begin{align}\label{symbol in the frame}
\tilde \phi\big( \iota(e_\alpha)\big)= \left(
\begin{array}{cccc}
0 & 0 & 0 & 0 \\
0 & 0  &e^{-ih\big(\tilde \phi\big)}\Xi_{\alpha}\big(\tilde \phi\big)H\big(\tilde \phi\big)^{-1} & 0 \\
0 &  0 & 0 & 0 \\
0 &  0 & 0 & 0
\end{array}
\right)
\end{align}
for $\alpha=1,\dots,n-s$, where $\Xi_1\big(\tilde \phi\big),\dots,\Xi_{n-s}\big(\tilde \phi\big)$ are the $s\times s$ matrices corresponding to coefficients by $f_i\otimes \overline{f_j}^*$  of  $\iota(e_1),\dots,\iota(e_{n-s})$ in a compatible adapted basis $\phi$ of the form \eqref{adapted basis}, and the pair $(H\big(\tilde \phi\big),e^{ih\big(\tilde \phi\big)})$ is as in Lemma \ref{involution lemma}.

Since $\tilde\phi(\mathfrak{g}_{0,2}(p))$ transforms to $Ad_{a^{-1}}\circ \tilde\phi(\mathfrak{g}_{0,2}(p))$ for the matrix $a$ of the form \eqref{csp00matrix}, the matrices $\Xi_\alpha\big(\tilde \phi\big)$ transform as 
\begin{align}\label{change in phi}
\Xi_\alpha(\mathrm{Ad}_{a^{-1}}\circ\tilde \phi)=b\overline{b}^{-1}B^{-1}\Xi_\alpha\big(\tilde \phi\big)\overline{B}.
\end{align}
\end{proposition}
\begin{proof}
This is immediate from the matrix representations of bigraded symbols given in \cite[Section 5]{sykes2023geometry} and the natural transformation between the representations of $\mathbb{C}\mathfrak{g}_{-}\oplus\mathfrak{csp}(\mathbb{C}\mathfrak{g}_{-1})$ used here and in \cite[Section 5]{sykes2023geometry}. The results are independent of the corresponding adapted basis $\phi$, because choosing different adapted bases compatible with $\tilde \phi$ just changes the generators of $\tilde \phi(\mathfrak{g}_{0,2}(p)).$
\end{proof}
\begin{remark}
Let us emphasize that the index $\alpha$ of $\Xi_\alpha\big(\tilde \phi\big)$ indicates its dependence on the basis vector $e_\alpha$. Suppressing this dependence in our notation does not cause ambiguity in the sequel as it will in every instance suffice to regard the $(e_1,\dots, e_{n-s})$ part of every adapted basis as arbitrary but fixed. Notice that \eqref{change in phi} resembles the natural transformation for matrices representing antilinear operators. Indeed $\Xi_\alpha$  represent antilinear maps that where described in \cite{porter2021absolute}.
\end{remark}
\begin{corollary}
For $\Xi_1,\ldots, \Xi_{n-s}$ as in \ref{matrix representation prop}, the corresponding structure is uniformly $2$-nondegenerate if and only if $n>s$ and the set of all $\Xi_\alpha$ is linearly independent, i.e., $n-s=r$.
\end{corollary}

Consequently, we have an analogue of \emph{soldering form} $\theta$ on the principal bundle $\mathcal{F}$ defined pointwise by 
\[
\theta\big(\tilde\phi\big):=\tilde{\phi}\circ (\mathrm{pr}_{\mathcal{F}})_* : \mathbb{C}T_{\tilde \phi}\mathcal{F}\to \mathbb{C}\mathfrak{g}_{-}\oplus\mathfrak{csp}(\mathbb{C}\mathfrak{g}_{-1}).
\]
In contrast to a usual soldering form, the pullback of $\theta$ along any section of $\mathcal{F}$ is injective if and only if $M$ is everywhere $2$-nondegenerate, $\theta$ is never surjective, and its image does not have to be constant. Nevertheless, one can check using  formulas \eqref{tilde phi minus formula} and \eqref{symbol in the frame} and formulas \eqref{change in H} and \eqref{change in phi} that $\theta$ is $\mathrm{CSp}(\mathbb{C}\mathfrak{g}_{-1})_{0,0}$--equivariant. 

For special cases, the CR geometry can studied effectively via reductions of $\mathcal{F}$, using $\mathcal{F}$ analogously to frame bundles for $G$-structures. For example, in \cite{gregorovic2020fundamental,gregorovic2021equivalence,porter2021absolute,sykes2021maximal,sykes2023geometry} extensive theory is developed for analysis of uniformly $2$-nondegenerate structures with constant bigraded symbols (Definition \ref{constant symbol defintion}), wherein these respective works study bundles that may be described as reductions of $\mathcal{F}$.

Let us fix a bigraded symbol $\mathbb{C}\mathfrak{g}_{\leq 0}=\mathbb{C}\mathfrak{g}_{-}\oplus \mathfrak{g}_{0,-2}\oplus \mathfrak{g}_{0,0}\oplus \mathfrak{g}_{0,2}$ with  antilinear involution $\sigma$ of $\mathbb{C}\mathfrak{g}_{\leq 0}$ corresponding to a pair $(\mathbf{H},e^{ih})$. For a point $p\in M$, we can then define sets
\begin{align}\label{E fiber}
\mathcal{E}_p:=\left\{\tilde \phi\in \mathcal{F}_p\,\left|\,\mbox{such that }\mathfrak{g}_{0,\pm 2}=\tilde \phi(\mathfrak{g}_{0,\pm 2}(p))\right.\right\}.
\end{align}
and $(\mathcal{E}_{e^{ih}\mathbf{H}})_p:=\mathcal{E}_p\cap \mathcal{F}_{e^{ih}\mathbf{H}}.$ In other words, $(\mathcal{E}_{e^{ih}\mathbf{H}})_p$ consists of isomorphisms between the bigraded symbol at $p$ with the fixed bigraded symbol.

\begin{definition}\label{constant symbol defintion}
The CR hypersurface $M$ has constant bigraded symbol $(\mathbb{C}\mathfrak{g}_{\leq 0},\sigma)$ on a connected submanifold $L\subset M$ if and  only if the set $(\mathcal{E}_{e^{ih}\mathbf{H}})_p\subset \mathcal{E}_p$ is not empty for all $p\in L$.   
\end{definition}

Note that $\mathcal{E}_p$ being non-empty does not necessarily imply the constancy of the symbol, because there can be many different antilinear involutions of $\mathbb{C}\mathfrak{g}_{\leq 0}$. By construction $\mathcal{E}$ is a $G_{0,0}$--subbundle of $\mathcal{F}$ whereas $\mathcal{E}_{e^{ih}\mathbf{H}}$ is a $G_{0,0}\cap \mathrm{CU}(e^{ih}\mathbf{H})$-subbundle, where $G_{0,0}$ is the maximal Lie subgroup in $\mathrm{CSp}(\mathbb{C}\mathfrak{g}_{-1})_{0,0}$ with Lie algebra $\mathfrak{g}_{0,0}$.

\subsection{Modified symbols of constant bigraded symbol structures}\label{sec2.2}

Modified symbols of constant bigraded symbol $2$-nondegenerate CR hypersurfaces were introduced in \cite{sykes2023geometry}, and we present in this section a new formulation of these objects (Lemma \ref{constant symbol corollary}) suitable for generalization to the non-constant bigraded symbol setting. This formulation is also well defined in the more general $k$-nondegenerate CR hypersurface setting.

As in Section \ref{sec bigraded symbol}, we consider an arbitrary, but fixed, uniformly $k$-nondegenerate CR hypersurface $M$. The distribution $TM\cap (\mathcal{K}_{2}\oplus\overline{\mathcal{K}_2})$ is integrable, and thus generates a foliation on $M$. Locally (i.e., after perhaps replacing $M$ with a sufficiently small neighborhood in $M$), the leaf space $N$ of this foliation has a naturally induced smooth structure. Letting $\pi:M\to N$ denote the natural projection, the distribution $\pi_*\left(\mathcal{H}\oplus\overline{\mathcal{H}}\right)$ defines a contact structure on $N$. Proceeding, we assume $M$ is replaced with a sufficiently small neighborhood such that $\pi:M\to N$ has the structure of a fiber bundle, and we denote by $M_{\pi(p)}$ its fiber over $\pi(p)\in N$. The contact structure provides the usual $\mathrm{CSp}(\mathbb{C}\mathfrak{g}_{-1})$-principal bundle
\[
\mathcal{G}:=\left\{\tau:\mathbb{C}\mathfrak{g}_{-}(p)\to \mathbb{C}\mathfrak{g}_{-}\,\left|\,\parbox{3.6cm}{$p\in N$ and $\tau$ is a contact-graded Lie algebra isomorphism}\right.\right\}
\]
over the contact manifold $N$, where $\mathbb{C}\mathfrak{g}_{-}(p)$ denotes the standard Heisenberg algebra obtained from the contact structure on $N$ by nilpotent approximation. We can relate the bundles $\pi\circ\mathrm{pr}:\mathcal{F}\to N$ and $\mathcal{G}$ over $N$ in the following way:

\begin{lemma}\label{I map lemma}
The map $\mathscr{I}:\mathcal{F}\to \mathcal{G}$ given by
\begin{align}\label{scri def}
\mathscr{I}\big(\tilde \phi\big):= \left(\pi_*\circ \tilde \phi^{-1}\right)^{-1}
\end{align}
is a $\mathrm{CSp}(\mathbb{C}\mathfrak{g}_{-1})_{0,0}$-equivariant homomorphism of fiber bundles over the local leaf space $N$. It is an immersion if $(M,\mathcal{H})$ is $1$ or $2$-nondegenerate.
\end{lemma}
\begin{proof}
The equivariance follows from $\pi_*$ inducing isomorphisms between the negatively graded part of each bigraded symbol on $M$ and the nilpotent approximations at each point on $N$. Injectivity of $\mathscr{I}_*$ in the $2$-nondegenerate setting is a consequence of \cite[Remark 2.1]{sykes2023geometry}. 
\end{proof}

 If $\tilde \phi: M_{\pi(p)}\to \mathcal{F}$ is a local section, then $\mathscr{I}\circ \tilde \phi(M_{\pi(p)})$ is in the fiber of $\mathcal{G}$ over $\pi(p)$ and thus there is a linear map $\iota_{\tilde \phi}: \mathcal{K}_2(p)\oplus\overline{\mathcal{K}_2}(p)\to \mathfrak{csp}(\mathbb{C}\mathfrak{g}_{-1})$ given by 
\begin{align}\label{iota phi map}
\iota_{\tilde \phi}(v):=\zeta_{\mathscr{I}\circ \tilde \phi(p)}^{-1}(\mathscr{I}\circ \tilde \phi)_*(v),
\end{align}
where $\zeta_{\tau}:\mathfrak{csp}(\mathbb{C}\mathfrak{g}_{-1})\to\Gamma(T_{\tau}\mathcal{G})$ is the map sending $X\in \mathfrak{csp}(\mathbb{C}\mathfrak{g}_{-1})$ to the value of the fundamental vector field generated by $X$ at $\tau$. This provides the following modified version of the bigraded symbol that depends on a local section $\tilde \phi: M_{\pi(p)}\to \mathcal{F}.$

\begin{definition}\label{mod symb in frame} For $p\in M$ and a local section $\tilde \phi: M_{\pi(p)}\to \mathcal{F}$ around $p$, the \emph{modified symbol $\mathfrak{g}_{\leq 0}^{\mathrm{mod}}\big(\tilde\phi;p\big)$ at the point $p$ along $\tilde \phi$} is the graded vector space
\[
\mathfrak{g}_{\leq 0}^{\mathrm{mod}}\big(\tilde\phi;p\big):=\mathbb{C}\mathfrak{g}_{-}\oplus \mathfrak{g}_0^{\mathrm{mod}}\big(\tilde\phi;p\big)
\]
given by
\[
\mathfrak{g}_0^{\mathrm{mod}}\big(\tilde\phi;p\big):=\iota_{\tilde \phi}\left(\mathcal{K}_2(p)\oplus\overline{\mathcal{K}_2}(p)\right)+\mathfrak{g}_{0,0}\big(\tilde \phi;p\big),
\]
where
\begin{align}\label{g00 in a frame}\mathfrak{g}_{0,0}(\tilde\phi;p):=\tilde\phi(p)(\mathfrak{g}_{0,0}(p)).
\end{align}
\end{definition}
\begin{remark}\label{mod symbol pm}
In order to have a direct product in Definition  \ref{mod symb in frame}, we would need to pick a subspace $V\subset \mathcal{K}_2(p)$ such that 
\begin{align}\label{modified symbol decomposition a}\mathfrak{g}_0^{\mathrm{mod}}\big(\tilde\phi;p\big)=\mathfrak{g}_{0,-}^{\mathrm{mod}}(\tilde\phi; p\big)\oplus \mathfrak{g}_{0,0}\big(\tilde \phi;p\big)\oplus\mathfrak{g}_{0,+}^{\mathrm{mod}}(\tilde\phi; p\big),
\end{align}
where $\mathfrak{g}_{0,-}^{\mathrm{mod}}(\tilde\phi; p\big):=\iota_{\tilde \phi}(\overline{V})$ and $\mathfrak{g}_{0,+}^{\mathrm{mod}}(\tilde\phi; p\big):=\iota_{\tilde \phi}(V)$. The decomposition in \eqref{modified symbol decomposition a} is canonical for the case of $2$-nondegenerate structures, wherein $V= \mathcal{K}_2(p)$. It is also canonical for some special cases in higher-nondegeneracy, but is not canonical in general.
\end{remark}

We include the $\mathfrak{g}_{0,0}(\tilde \phi;p)$ term in the definition of $\mathfrak{g}_0^{\mathrm{mod}}\big(\tilde\phi;p\big)$ so that Definition \ref{mod symb in frame} is consistent with the modified symbol defined in \cite{sykes2023geometry} for $2$-nondegenerate structures with constant bigraded symbols on $M$. The latter definition may be formulated using a Maurer--Cartan form \cite[Remark 4.2]{sykes2023geometry}, and to compare that formulation with our present definition more directly, we note the following observations. For a structure with constant bigraded symbol on $M_{\pi(p)}$, we have a bundle  $\mathrm{pr}:\mathcal{E}\to M_{\pi(p)}$  as in \eqref{E fiber}. At a point $\tilde \phi\in \mathcal{E}$, we can define
\begin{align}\label{g0mod original def}
\mathfrak{g}_0^{\mathrm{mod}}\big(\tilde \phi\big):=\zeta_{\mathscr{I}( \tilde \phi)}^{-1}(\mathscr{I}_*(\mathbb{C}T_{\tilde \phi}\mathcal{E})).
\end{align}
Since $\zeta_{\mathscr{I}( \tilde \phi)}^{-1}(\mathscr{I}_*(\mathbb{C}T_{\tilde \phi}\mathcal{E}))=\omega_{\mathrm{CSp}}(\mathbb{C}T_{e}(\mathscr{I}\big(\tilde \phi\big)\circ (\mathscr{I}(\mathcal{E}))^{-1}))$ holds for the Maurer--Cartan form $\omega_{\mathrm{CSp}}$ on $\mathrm{CSp}(\mathbb{C}\mathfrak{g}_{-1})$, we recover the following definition:

\begin{definition}[introduced in \cite{sykes2023geometry} for constant bigraded symbol structures on $M$]\label{modified symbol}
For a structure with constant bigraded symbol on $M_{\pi(p)}\subset M$ of type $(\mathbb{C}\mathfrak{g}_{\leq0},\sigma)$, the \emph{modified symbol}
$\mathfrak{g}^{\mathrm{mod}}_{\leq 0}\big(\tilde \phi\big)$ of $M$ at a point $\tilde \phi\in \mathcal{E}$ is the vector space
\[
\mathfrak{g}^{\mathrm{mod}}_{\leq 0}\big(\tilde \phi\big):=\mathbb{C}\mathfrak{g}_{-}\oplus \mathfrak{g}_0^{\mathrm{mod}}\big(\tilde \phi\big)\subset \mathbb{C}\mathfrak{g}_{-}\oplus \mathfrak{csp}(\mathbb{C}\mathfrak{g}_{-1})
\]
where $\mathfrak{g}_0^{\mathrm{mod}}\big(\tilde \phi\big)\subset \mathfrak{csp}(\mathbb{C}\mathfrak{g}_{-1})$ is  as in \eqref{g0mod original def}.
\end{definition}

Let us now express in more detail how modified symbols along $\tilde \phi$ can be described for a particular local section $\tilde \phi: M_{\pi(p)}\to \mathcal{F}$ around $p\in M.$

\begin{lemma}\label{matrrep}
Let $\mathfrak{g}_{\leq 0}^{\mathrm{mod}}\big(\tilde\phi; p\big)$ be a modified symbol at $p$ along $\tilde \phi$ associated with a $(2n+1)$-dimensional uniformly $k$-nondegenerate CR hypersurface. Then for $e_\alpha\in \mathcal{K}_2(p)$, there are matrices $\Xi_\alpha(\tilde \phi(p))$ and $\Omega_\alpha\big(\tilde \phi\big)(p)$ such that
\begin{align}\label{modified symbol rep}
\iota_{\tilde \phi}(e_\alpha)\equiv\left(
\begin{array}{cccc}
0 & 0 & 0 & 0 \\
0 &  \Omega_\alpha\big(\tilde\phi\big)(p) &e^{-ih\big(\tilde\phi(p)\big)}\Xi_\alpha(\tilde\phi(p))H\big(\tilde\phi(p)\big)^{-1} & 0 \\
0 &  0 & -\big(\Omega_\alpha\big(\tilde\phi\big)(p)\big)^T & 0 \\
0 &  0 & 0 & 0
\end{array}\right),
\end{align}
where $\Xi_\alpha(\tilde \phi(p))$ is as in Proposition \ref{matrix representation prop} with $e_\alpha$ coming from an adapted basis $\phi(p)$ compatible with $\tilde \phi(p)$ of the form \eqref{adapted basis}, and this equivalence is modulo the conformal scaling element in $\mathfrak{g}_{0,0}$, i.e., matrices of the form in \eqref{csp representation} with only the $c$ blocks nonzero.

In particular, if $k=2$, then $\mathfrak{g}_{0}^{\mathrm{mod}}\big(\tilde\phi; p\big)$ is of the form \eqref{modified symbol decomposition a}, where
$\mathfrak{g}_{0,+}^{\mathrm{mod}}\big(\tilde\phi; p\big)$ is spanned by \eqref{modified symbol rep} for $j=1,\dots, r$ for an adapted basis $\phi(p)$ compatible with $\tilde \phi(p)$ of the form \eqref{adapted basis} and $\mathfrak{g}_{0,-}^{\mathrm{mod}}\big(\tilde\phi; p\big)=\sigma_{\tilde \phi}(\mathfrak{g}_{0,+}^{\mathrm{mod}}\big(\tilde\phi; p\big))$.

For a function $a: M_{\pi(p)}\to \mathrm{CSp}(\mathbb{C}\mathfrak{g}_{-1})_{0,0}$ represented pointwise by the matrices \eqref{csp00matrix},
\begin{align}\label{change of modified symbol in a frame}
\Omega_\alpha\left(\mathrm{Ad}_a^{-1}\circ \tilde\phi\right)=B^{-1}\Omega_\alpha\big(\tilde \phi\big)B+\pi_{\mathfrak{gl}(s,\mathbb{C})}\circ a^*\omega_{\mathrm{CSp}(\mathbb{C}\mathfrak{g}_{-1})_{0,0}}(e_\alpha),
\end{align}
where $\pi_{\mathfrak{gl}(s,\mathbb{C})}$ denotes the projection extracting $L$ from the matrix \eqref{csp representation} and $a^*\omega_{\mathrm{CSp}(\mathbb{C}\mathfrak{g}_{-1})_{0,0}}$ denotes the pullback of the Maurer--Cartan form. 
\end{lemma}
\begin{proof}
We will first prove the existence of $\Omega_\alpha$ such that \eqref{modified symbol rep} holds. Let $v$ be a vector field on $M$ equal to either $e_\alpha+\overline{e_\alpha}$ or $ie_\alpha-i\overline{e_\alpha}$. Since $v$ is in $\Gamma(\mathcal{K}_2\oplus\overline{\mathcal{K}_2})$, leaves of the Levi foliation are invariant under its flow, that is,
\begin{align}\label{leaf preserving}
\pi(p)=\pi\circ\mathrm{Fl}_t^v(p)
\quad\quad\forall\,p\in M
\end{align}
for all $t\in \mathbb{R}$ with $\mathrm{Fl}_t^v(p)$ defined. To calculate $\zeta^{-1}_{\mathscr{I}\circ \tilde \tilde \phi(p)}\circ(\mathscr{I}\circ \phi)_*(v)$ at a point $p\in M$, one should consider a curve $t\mapsto \tau_t\in \mathcal{G}_{\pi(p)}$ such that 
\[\tau_0=\mathscr{I}\circ \tilde \phi(p)
\quad\mbox{ and }\quad
\left.\frac{d}{dt}\right|_{t=0}\tau_t=(\mathscr{I}\circ \tilde \phi)_*\left(\left.v\right|_{p}\right),
\]
so let us take
\[
\tau_t:=\mathscr{I}\circ \tilde \phi\circ\mathrm{Fl}_t^v(p):\mathbb{C}\mathfrak{g}(\pi(p))\to \mathbb{C}\mathfrak{g}_{-},
\]
which indeed satisfies these properties. Thus, regarding $\tau_t$ as isomorphisms $\tau_t:\mathbb{C}\mathfrak{g}_{-}\big(\pi(p)\big)\to \mathbb{C}\mathfrak{g}_{-}$, we may use their compositions to compute
\begin{align}\label{LK image under MC form}
\iota_{\tilde \phi}\left(\left.v\right|_p\right)=\omega_{\mathrm{CSp}}\left(\left.\frac{d}{dt}\right|_{t=0}\tau_0\circ \tau_t^{-1}\right)\in\mathfrak{csp}\left(\mathbb{C}\mathfrak{g}_{-1}\right).
\end{align}

With a local adapted frame $\phi$ compatible with $\tilde \phi$, we can use formula \eqref{tilde phi coordinates} to obtain \eqref{modified symbol rep} from expressing \eqref{LK image under MC form} in terms of the Lie brackets between $v$ and $f_1,\ldots, f_s,\overline{f_1},\ldots,\overline{f_s}$, computed as follows. First, let us compute the value of the automorphism $\tau_0\circ \tau_t^{-1}\in \mathrm{CSp}(\mathbb{C}\mathfrak{g_{-1}})$. Applying \eqref{scri def} and \eqref{leaf preserving}, 
\begin{align}
\tau_t^{-1}\left(\tilde\phi\circ  f_k(p)\right)=\pi_*\circ\left(\left.\tilde\phi\right|_{\mathrm{Fl}^v_t(p)}\right)^{-1}\left(\tilde\phi\left( f_k(p)\right)\right)
&=\pi_*\left(f_k\left(\mathrm{Fl}^v_t(p)\right)\right)\\
&=\left(\pi\circ \mathrm{Fl}^v_{-t}(p)\right)_*\left(f_k\left(\mathrm{Fl}^v_t(p)\right)\right)\\
&= \pi_*\left(\left( \mathrm{Fl}^v_{-t}(p)\right)_*f_k\left(\mathrm{Fl}^v_t(p)\right)\right),
\end{align}
so
\begin{align}\label{LB in MCF applied to frame}
&\iota_{\tilde \phi}\left(\left.v\right|_p\right)\left(\tilde\phi\left( f_k(p)\right)\right)=\omega_{\mathrm{CSp}}\left(\left.\frac{d}{dt}\right|_{t=0}\tau_0\circ \tau_t^{-1}\left(\tilde\phi\left( f_k(p)\right)\right)\right)=\\
&\hspace{5cm}=\omega_{\mathrm{CSp}}\left(\left.\frac{d}{dt}\right|_{t=0}\tau_0\circ\pi_*\left(\left( \mathrm{Fl}^v_{-t}(p)\right)_*f_k\left(\mathrm{Fl}^v_t(p)\right)\right)\right)\\
&\hspace{5cm}=\omega_{\mathrm{CSp}}\left(\tau_0\circ\pi_*\left([v,f_k]_p\right)\right)\\
&\hspace{5cm}= \tilde\phi([v,f_k]_p).
\end{align}
Similarly
\begin{align}
\iota_{\tilde \phi}\left(\left.v\right|_p\right)\left(\tilde\phi\left( \overline{f_k}(p)\right)\right)=\tilde\phi\left([v,\overline{f_k}]_p\right).
\end{align}
Since the Lie bracket $[v,w] \pmod{\mathcal{K}_2\oplus\overline{\mathcal{K}_2}}$ is linear in $v$ for $w=f_k$ or $w=\overline{f_k}$ and $v=e_\alpha+\overline{e_\alpha}$ or $v=ie_\alpha-i\overline{e_\alpha}$, we obtain matrices $\Omega_\alpha\big(\tilde \phi\big)(p)$ such that
\begin{align}
[e_\alpha,f_k](p)\equiv\sum_{l=1}^{s}(\Omega_\alpha\big(\tilde \phi\big)(p))_{l,k} f_l\pmod{\mathcal{K}_2\oplus\overline{\mathcal{K}_2}}
\quad\quad\forall 1\leq k\leq s,
\end{align}
Note the left side of the above formula is independent of $\overline{f_k}$ because $\mathcal{H}$ is involutive. Since
\begin{align}
[e_\alpha,\overline{f_k}](p)\equiv \iota([e_\alpha,\overline{f_k}])=\sum_{l=1}^{s}\left(\Xi_\alpha\big(\tilde \phi(p)\big)\right)_{l,k} f_l\pmod{\mathcal{K}_2\oplus\overline{\mathcal{H}}}
\quad\quad\forall 1\leq k\leq s,
\end{align}
we obtain the claimed formula \eqref{modified symbol rep} as $\tilde \phi(0,\Omega_\alpha\big(\tilde \phi\big),\Xi_\alpha\big(\tilde \phi\big),0)$ using \eqref{tilde phi coordinates}. 

The $k=2$ case then follows from Remark \ref{mod symbol pm}.

Lastly, the transformation formula of $\Omega$ in \eqref{change of modified symbol in a frame} follows from replacing the above $\tau_t$ with
\begin{align}
\tau_t=\mathscr{I}\circ (\mathrm{Ad}_{a}^{-1}\circ \tilde \phi)\circ\mathrm{Fl}_t^v(p)=\mathrm{Ad}_{a\left(\mathrm{Fl}_t^v(p)\right)}^{-1}\circ\left(\mathscr{I}\circ \tilde \phi\circ\mathrm{Fl}_t^v(p)\right),
\end{align}
and then observing how this replacement changes \eqref{LB in MCF applied to frame}.
\end{proof}
\begin{remark}
Given how $e^{ih\big(\tilde\phi\big)}H\big(\tilde\phi\big)$ and $\Xi_\alpha\big(\tilde\phi\big)$ transform algebraically as in \eqref{change in H} and \eqref{change in phi}, their values at $p$ depend only on $\tilde\phi(p)$ rather than $\tilde\phi$ in a neighborhood. Hence the above notation $H\big(\tilde\phi(p)\big)$ and $\Xi_\alpha\big(\tilde\phi(p)\big)$ is unambiguous. Contrastingly, one must denote the value of $\Omega_\alpha\big(\tilde\phi\big)$ at $p$ as $\Omega_\alpha\big(\tilde\phi\big)(p)$ because, due to \eqref{change of modified symbol in a frame}, it necessarily depends on values of $\phi$ in a neighborhood of $p$. Also from \eqref{change of modified symbol in a frame}, one sees that the decomposition \eqref{modified symbol decomposition a} depends on values of $\tilde\phi$ near $p$, rather than just its value at $p$.
\end{remark}

Noting that the structure group of $\mathcal{E}$ has Lie algebra $\mathfrak{g}_{0,0}(\tilde\phi;p)$ for any $\tilde\phi\in\Gamma(E)$ and $p\in M$, we get the following corollary of the formula \eqref{change of modified symbol in a frame}, ensuring that the Definitions \ref{mod symb in frame} and \ref{modified symbol} coincide in the case of constant bigraded symbol.

\begin{corollary}\label{constant symbol corollary}
In the constant bigraded symbol setting, if $\tilde \phi$ is local section $M_{\pi(p)}\to\mathcal{E}$ at $p$ for $p\in M$, then $\mathfrak{g}_{\leq 0}^{\mathrm{mod}}\big(\tilde \phi(p)\big)=\mathfrak{g}_{\leq 0}^{\mathrm{mod}}\big(\tilde \phi; p\big)$. 
\end{corollary}

Corollary \ref{constant symbol corollary} implies one can use \emph{symbols along $\tilde \phi$} to define (frame-independent) modified symbols of a CR hypersurface with constant bigraded symbol by specifying a natural family of frames such that the modified symbols along these frames are all the same. This conceptual viewpoint will be a key idea used in the next section to generalize Definition \ref{modified symbol}.

We denote by $\theta_-$ the part of the soldering form $\theta$ on $\mathbb{C}T\mathcal{F}$ with values in $\mathbb{C}\mathfrak{g}_{-}$. Its kernel is $\ker(\theta_-)=\mathrm{pr}_{\mathcal{F}}^{-1}(\mathcal{K}_2\oplus \overline{\mathcal{K}_2})$, and $\omega=\zeta^{-1}\circ \mathscr{I}_*$ defines (by Lemma \ref{I map lemma}) a $\mathrm{CSp}(\mathbb{C}\mathfrak{g}_{-1})_{0,0}$-equivariant $\mathfrak{csp}(\mathbb{C}
\mathfrak{g}_{-1})$-valued form on $\ker(\theta_-)$. Of course, for a local section $\tilde \phi: M_{\pi(p)}\to \mathcal{F}$,
\[
\tilde\phi^*\omega(p)=\iota_{\tilde \phi}\big|_{\mathcal{K}_2(p)\oplus \overline{\mathcal{K}_2}(p)},
\quad\quad\forall\, p\in M.
\]
For each local adapted frame $\phi$ of $\mathbb{C}TM$, we obtain a $\omega_{\phi}\in \Omega^1\left(M; \mathbb{C}\mathfrak{g}_{-}\oplus\mathfrak{csp}(\mathbb{C}\mathfrak{g}_{-1})\right)$ given by $\omega_{\phi}:=\tilde\phi^{*}\theta_{-}+\tilde\phi^{*}\omega$, where we view $\tilde\phi^{*}\omega$ as a $1$-form on $M$ by defining its kernel as the space spanned by first $2s+1$ vector fields in the frame $\phi$. For $2$-nondegenerate structures, the matrices $\Omega_\alpha\big(\tilde \phi\big)$ and $\Xi_\alpha\big(\tilde \phi\big)$ -- and therefore the modified symbols along $\tilde\phi$ -- described above can alternatively be characterized via the Maurer-Cartan formula $d \omega_{\phi}+[\omega_{\phi},\omega_{\phi}]$ as follows.

\begin{proposition}\label{degree zero characterization}
Suppose that $M$ is a $2$-nondegenerate CR hypersurface and choose a local adapted frame $\phi$ of the form \eqref{adapted basis}. Suppose $\Xi_{1},\ldots, \Xi_{r},\Omega_{1},\ldots,\allowbreak \Omega_{r}$ are $s\times s$ matrices of functions on $M$ such that the $\mathbb{C}\mathfrak{g}_{-}\oplus \mathfrak{csp}(\mathbb{C}\mathfrak{g}_{-1})$-valued form $\eta:=\tilde \phi^*\theta_-+\eta_{0}$ given by
\[
\eta_0(g)=\eta_0(f_j)=0,
\quad\eta_0(\overline{x}):=\sigma_{\tilde \phi(p)}\eta_0(x)
\quad\quad\forall\,x\in T_pM,\, 1\leq j\leq s
\]
and
\[
\eta_0(e_\alpha):=\left(
\begin{array}{cccc}
0 & 0 & 0 & 0 \\
0 &  \Omega_\alpha(p) &e^{-ih(\tilde \phi(p))}\Xi_\alpha(p)H(\tilde\phi(p))^{-1} & 0 \\
0 &  0 & -\big(\Omega_\alpha(p)\big)^T & 0 \\
0 &  0 & 0 & 0
\end{array}\right)
\]
satisfies the Maurer--Cartan equation $d\eta+[\eta,\eta]\equiv 0$ up to terms of positive homogeneity with respect to the bigraded algebra's first grading. Then 
\[
\eta_{0}\big|_{\mathcal{K}_2(p)\oplus \overline{\mathcal{K}_2}(p)}=\tilde \phi^*\omega\big|_{\mathcal{K}_2(p)\oplus \overline{\mathcal{K}_2}(p)},
\]
and the modified symbol $\mathfrak{g}_{\leq 0}^{\mathrm{mod}}(\tilde \phi; p)$  along $\tilde \phi$ is determined by $\Xi_\alpha(\tilde \phi(p))=\Xi_\alpha(p)$ and $\Omega_{j}\big(\tilde \phi\big)(p)\allowbreak=\Omega_{j}(p)$ as in Lemma \ref{matrrep}.
\end{proposition}
\begin{proof}
Suppose $\Xi_{1},\ldots, \Xi_{r},\Omega_{1},\ldots, \Omega_{r}$ are matrices of functions on $M$ satisfying the assumptions of the proposition. The form $\eta_0$ satisfies $d\eta_0+[\eta_0,\eta_0]=0$ by assumption and thus, by the fundamental theorem of calculus \cite[Theorem 3.7.14]{sharpe2000differential}, $\eta_0$ is a logarithmic/Darboux derivative of a map $M\to \mathrm{CSp}(\mathbb{C}\mathfrak{g}_{-1})$. Moreover, the assumption $d\eta+[\eta,\eta]\equiv 0$ up to terms of positive homogeneity with respect to the bigraded algebra's first grading in  addition implies $\tilde \phi^*\theta_-$ is equivariant with respect to the induced local group action of $\mathrm{CSp}(\mathbb{C}\mathfrak{g}_{-1})$ on the local trivialization $\mathcal{F}=M\times \mathrm{CSp}(\mathbb{C}\mathfrak{g}_{-1})_{0,0}$ provided by the frame. In this way we recover the map $\mathscr{I}: \mathcal{F}\to \mathcal{G}$, because $\tilde\phi^*\theta_-$ is then the pullback by $\mathscr{I}\circ\tilde\phi$ of the usual soldering form on the bundle $\mathcal{G}.$  Then $\eta_{0}=\tilde \phi^*\omega$ and the final claim follows from the Lemma \ref{matrrep}.
\end{proof}

\subsection{Modified symbols for non-constant bigraded symbol structures}\label{sectionnonconstantmodifedsymbol}

 In this section, we show how to define a frame-independent modified symbol, generalizing Definition \ref{modified symbol} to general $k$-nondegenerate CR hypersurfaces. This will be done by imposing normalization conditions for the change of the bigraded symbol along the leaves of the Levi kernel. This provides new CR invariants that are obstructions to constancy of the bigraded symbol. The vanishing of these invariants -- a property we refer to as  \emph{ having constant bigraded symbol up to first order} (c.f., Definition \ref{normalized sections}) -- allows one to define the modified symbols independently from the normalization condition, but does not guarantee the constancy of the bigraded symbol.
 
Consider a local section $\tilde \phi: M_{\pi(p)}\to \mathcal{F}$ and a local section $\tilde \phi^\prime=\mathrm{Ad}_a^{-1}\circ \tilde \phi$ for a smooth map $a: M_{\pi(p)}\to \mathrm{CSp}(\mathbb{C}\mathfrak{g}_{-1})_{0,0}$ with $a(p)=\mathrm{Id}$ and $X\in (\mathcal{K}_2)_p$. Labeling  as
\[
S^{0,2}_\alpha(\tilde\phi):=e^{-ih\big(\tilde\phi\big)}\Xi_\alpha\big(\tilde\phi\big)H\big(\tilde\phi\big)^{-1}
\quad\mbox{ and }\quad
S^{0,-2}_\alpha(\tilde \phi):=e^{ih\big(\tilde\phi\big)}H\big(\tilde\phi\big)\overline{\Xi_\alpha\big(\tilde\phi\big)}
\]
the $\mathfrak{csp}(\mathbb{C}\mathfrak{g}_{-1})_{0,2}$-valued and $\mathfrak{csp}(\mathbb{C}\mathfrak{g}_{-1})_{0,-2}$-valued functions (under the natural identification from Proposition \ref{matrix representation prop} for the bigraded symbols), we can define bilinear maps
\[O_{0,-2}(\tilde \phi)(p):(\mathcal{K}_2)_p\otimes \tilde \phi(\mathfrak{g}_{0,-2}(p))\to \mathfrak{csp}(\mathbb{C}\mathfrak{g}_{-1})_{0,-2}/\tilde\phi(\mathfrak{g}_{0,-2}(p))\]
and 
\[O_{0,2}(\tilde \phi)(p):(\mathcal{K}_2)_p\otimes \tilde \phi(\mathfrak{g}_{0,2}(p))\to \mathfrak{csp}(\mathbb{C}\mathfrak{g}_{-1})_{0,2}/\tilde\phi(\mathfrak{g}_{0,2}(p))\]
by $O_{0,-2}(\tilde \phi)(p)(X,S^{0,-2}_\alpha):=X\left(S^{0,-2}_\alpha(\tilde \phi)\right)(p)$ and $O_{0,2}(\tilde \phi)(p)(X,S^{0,2}_{\alpha}):=X\left(S^{0,2}_\alpha(\tilde \phi)\right)(p)$. Let us emphasize that the definitions of $O_{0,\pm2}$ are independent of an adapted frame $\phi$ compatible with section $\tilde \phi$, because derivations of functional combinations of elements of $\tilde \phi(\mathfrak{g}_{0,\pm 2}(p))$ remain in $\tilde \phi(\mathfrak{g}_{0,\pm2}(p))$ and therefore $O_{0,\pm2}$ become linear when we quotient them out. On the other hand, we can conclude from \eqref{change in H} and \eqref{change in phi} that
\begin{align}
X\left(S^{0,-2}_\alpha(\tilde \phi^\prime)\right)(p)&=X\left(S^{0,-2}_\alpha(\tilde \phi)\right)(p)-[X(a)(p),S^{0,-2}_\alpha(\tilde \phi(p))], \\
X\left(S^{0,2}_\alpha(\tilde \phi^\prime)\right)(p)&=X\left(S^{0,2}_\alpha(\tilde \phi)\right)(p)-[X(a)(p),S^{0,2}_\alpha(\tilde \phi(p))]
\end{align}
under the change of local sections of $\mathcal{F},$ which characterizes the CR invariant equivalence classes $\big\{\big(O_{0,-2}(\tilde \phi^\prime)(p),O_{0,2}(\tilde \phi^\prime)(p)\big)\,|\, \tilde \phi^\prime: M_{\pi(p)}\to \mathcal{F}, \tilde \phi^\prime(p)=\tilde \phi(p)\big\}$. Observe that these equivalence classes depend only on $\tilde \phi(\mathfrak{g}_{0,\pm 2}(p))$ ( i.e., on the bigraded symbol) and thus one can universally pose a normalization condition on these equivalence classes by choosing a unique pair of bilinear maps 
\[
\big(O_{0,-2}^N(\tilde \phi(p)),O_{0,2}^N(\tilde \phi(p))\big)\in \left\{\left.\left(O_{0,-2}(\tilde \phi^\prime)(p),O_{0,2}(\tilde \phi^\prime)(p)\right)\,\right|\, \tilde \phi^\prime: M_{\pi(p)}\to \mathcal{F}, \tilde \phi^\prime(p)=\tilde \phi(p)\right\}
\]
in each equivalence class. In general, this is a nontrivial task as the existence of the normalization condition follows from the axiom of choice (except that we can, and for convenience will, always choose $(0,0)$ within its equivalence class). Nevertheless, a choice of normalization condition $\big(O_{0,-2}^N(\tilde \phi(p)),O_{0,2}^N(\tilde \phi(p))\big)$ provides a reduction of the frames and allows us to make the following definition.

\begin{definition}\label{normalized sections}
We say that a local section $\tilde \phi: M_{\pi(p)}\to \mathcal{F}$ is \emph{normalized with respect to a choice of normalization condition $\big(O_{0,-2}^N(\tilde \phi(p)),O_{0,2}^N(\tilde \phi(p))\big)$ at $p$} if
$O_{0,\pm 2}(\tilde \phi)(p)=O_{0,\pm2}^N(\tilde \phi(p))$. We say that $\mathfrak{g}_{\leq0}^{\mathrm{mod}}(\tilde \phi(p)):=\mathfrak{g}_{\leq0}^{\mathrm{mod}}({\tilde \phi;p})$ is the \emph{modified symbol at $\tilde \phi(p)$ with respect to normalization condition $\big(O_{0,-2}^N(\tilde \phi(p)),O_{0,2}^N(\tilde \phi(p))\big)$} if $\tilde \phi$ is normalized with respect to $\big(O_{0,-2}^N(\tilde \phi(p)),O_{0,2}^N(\tilde \phi(p))\big)$. 

We say that $(M,\mathcal{H})$ has \emph{constant bigraded symbol up to first order at $p$} if $O_{0,\pm2}^N(\tilde \phi(p))=0$ for any local section $\tilde \phi: M_{\pi(p)}\to \mathcal{F}$ around $p$, and otherwise we say that $O_{0,\pm2}^N(\tilde \phi(p))$ are obstructions to first order constancy at $p$.
\end{definition}

To prove that Definition \ref{normalized sections} is indeed well posed, we show that this definition of $\mathfrak{g}_{\leq0}^{\mathrm{mod}}(\tilde \phi(p))$ is independent of the normalized section $\tilde \phi$ extending $\tilde \phi(p)$ in the following theorem.

\begin{theorem}\label{modified symbols in normalized frames}
Suppose $\phi$ is a local adapted frame of $\mathbb{C}TM$ around $p$ as in Definition \ref{adapted frame} and the modified symbol $\mathfrak{g}_{\leq0}^{\mathrm{mod}}(\tilde \phi;p)$ at the point $p$ along  $\tilde \phi$ is given by matrices $(e^{ih(\tilde \phi)}H(\tilde \phi),\Xi_1(\tilde \phi),\dots,\Xi_{n-s}(\tilde \phi),\Omega_1(\tilde \phi),\dots,\Omega_{n-s}(\tilde \phi))$. Then there are unique classes $B_1,\dots,B_{n-s} \in \mathfrak{csp}(\mathbb{C}\mathfrak{g}_{-1})_{0,0}/\mathfrak{g}_{0,0}(\tilde \phi; p)$ such that
\begin{align}\label{first order constancy criterion b}
  e_\alpha\left(S^{0,-2}_\beta(\tilde\phi)\right)(p)-\left[B_\alpha,S^{0,-2}_\beta(\tilde\phi(p))\right]= O_{0,-2}^N(\tilde \phi(p))\big(e_\alpha,S^{0,-2}_\beta(\tilde\phi(p))\big)
\end{align}
and
\begin{align}\label{first order constancy criterion a}
   e_\alpha\left(S^{0,2}_\beta(\tilde\phi)\right)(p)- \left[B_\alpha,S^{0,2}_\beta(\tilde\phi(p))\right]= O_{0,2}^N(\tilde \phi(p))\big(e_\alpha,S^{0,2}_\beta(\tilde\phi(p))\big)
\end{align}
hold for all $\alpha,\beta$. The modified symbol $\mathfrak{g}_{\leq0}^{\mathrm{mod}}(\tilde \phi(p))$ at $\tilde \phi(p)$ with respect to normalization condition $\big(O_{0,-2}^N(\tilde \phi(p)),O_{0,2}^N(\tilde \phi(p))\big)$  is given by the matrices 
\[
(e^{ih(\tilde \phi)}H(\tilde \phi),\Xi_1(\tilde \phi),\dots,\Xi_{n-s}(\tilde \phi),\Omega_1(\tilde \phi)+B_1,\dots,\Omega_{n-s}(\tilde \phi)\allowbreak+B_{n-s})
\]
and is independent of extension of $\tilde \phi(p)$ to a normalized section.
\end{theorem}
\begin{proof}
Consider $a:M\to \mathrm{CSp}(\mathbb{C}\mathfrak{g}_{-1})_{0,0}$ as in Lemma \ref{matrrep} with $a(p)=\mathrm{Id}$ and $e_\alpha(a)(p)=B_\alpha$. Accordingly, \eqref{first order constancy criterion b} and  \eqref{first order constancy criterion a} are the formulas for the change of the derivatives in frame $\tilde \phi$ to  derivatives in a normalized section $\tilde\phi^\prime=\mathrm{Ad}_{a}^{-1}\circ \tilde \phi$  with respect to $\big(O_{0,-2}^N(\tilde \phi(p)),O_{0,2}^N(\tilde \phi(p))\big)$. The formula for $\mathfrak{g}_{\leq0}^{\mathrm{mod}}(\tilde \phi^\prime(p))$ then follows from \eqref{change of modified symbol in a frame}. 

Now, if $\tilde \phi$ was already a normalized with respect to $\big(O_{0,-2}^N(\tilde \phi(p)),O_{0,2}^N(\tilde \phi(p))\big)$, then we have $[B_\alpha,S^{0,2}_\beta(\tilde\phi(p))]\in \tilde\phi(\mathfrak{g}_{0,2}(p))$ and $[B_\alpha,S^{0,-2}_\beta(\tilde\phi(p))]\in \tilde\phi(\mathfrak{g}_{0,-2}(p)).$ This implies $B_\alpha\in \mathfrak{g}_{0,0}$. Therefore, the modified symbols for $\tilde \phi$ and $\tilde \phi^\prime$ are the same.
\end{proof}

\subsection{Key examples of modified symbol calculations}\label{Key examples of modified symbol calculations}

Let us now consider a $2$-nondegenerate model. We will now determine its modified symbol (under suitable normalization conditions if its bigraded symbol is not constant to first order).

So let us start with the defining equation
\begin{align}
\Re(w)=z^TH(\zeta,\overline{\zeta})\overline{z}+\Re(\overline{z}^TS(\zeta,\overline{\zeta})\overline{z}),
\end{align}
the corresponding Levi form
\begin{align}
\mathcal{L}^1=\left(
\begin{array}{cc}
H(\zeta,\overline{\zeta}) & H_{\overline{\zeta}}(\zeta,\overline{\zeta})\overline{z}+\overline{S_{\zeta}(\zeta,\overline{\zeta})}z \\
z^TH_{\zeta}(\zeta,\overline{\zeta})+\overline{z}^TS_{\zeta}(\zeta,\overline{\zeta})  & z^TH_{\zeta,\overline{\zeta}}(\zeta,\overline{\zeta})\overline{z}+\Re(\overline{z}^TS_{\zeta\overline{\zeta}}(\zeta,\overline{\zeta})\overline{z})  \\
\end{array}
\right).
\end{align}
and consider the following adapted frame $\phi$:
\begin{align}
g:=&\frac{\partial}{\partial \Im(w)},\\
f_a:=&\frac{\partial}{\partial z_a}-i(H_a(\zeta,\overline{\zeta})\overline{z}+\overline{S_a(\zeta,\overline{\zeta})}z)\frac{\partial}{\partial \Im(w)},\\
e_\alpha:=&\frac{\partial}{\partial \zeta_{\alpha}}-i\left(z^TH_{\zeta_{\alpha}}(\zeta,\overline{\zeta})\overline{z}+\Re(\overline{z}^TS_{\zeta_{\alpha}}(\zeta,\overline{\zeta})\overline{z})\right)\frac{\partial}{\partial \Im(w)}\\
&-\sum_{\overline{b},c}\left(z^TH_{\zeta_{\alpha},\overline{b}}(\zeta,\overline{\zeta})+\overline{z}^TS_{\zeta_{\alpha},\overline{b}}(\zeta,\overline{\zeta})\right)(H(\zeta,\overline{\zeta})^{-1})_{c,\overline{b}}f_c,
\end{align}
where the notation $H_a(\zeta,\overline{\zeta})$ denotes the index $a$ row of $H(\zeta,\overline{\zeta})$, $(H(\zeta,\overline{\zeta})^{-1})_{c,\overline{b}}$ denotes the $(c,b)$ entry of $H(\zeta,\overline{\zeta})^{-1}$, and $H_{\zeta_\alpha,\overline{b}}(\zeta,\overline{\zeta})$ (resp. $S_{\zeta_\alpha,\overline{b}}(\zeta,\overline{\zeta})$) denotes the index $b$ column of $H_{\zeta_\alpha}(\zeta,\overline{\zeta})$ (resp. $S_{\zeta_\alpha}(\zeta,\overline{\zeta})$).

The Levi form $\mathcal{L}^1$ can be represented with respect to the frame $\phi$ according to our assumptions by the matrix
\begin{align}\label{key example Lform}
\left(
\begin{array}{cc}
H(\zeta,\overline{\zeta}) & 0 \\
0  & 0  \\
\end{array}
\right),
\end{align}
that is, $e^{ih\big(\tilde\phi\big)}=1$ and $H\big(\tilde \phi\big)=H(\zeta,\overline{\zeta}).$
Further,
\begin{align}\label{key example Xi mat}
\Xi_{\beta}\big(\tilde \phi\big)=(H(\zeta,\overline{\zeta})^T)^{-1}S_{\zeta_{\beta}}(\zeta,\overline{\zeta}),\quad S^{0,2}_{\beta}\big(\tilde \phi\big)=(H(\zeta,\overline{\zeta})^T)^{-1}S_{\zeta_{\beta}}(\zeta,\overline{\zeta})H(\zeta,\overline{\zeta})^{-1}
\end{align}
and
\begin{align}\label{key example Omega mat}
\Omega_\beta\big(\tilde \phi\big)=(H(\zeta,\overline{\zeta})^T)^{-1}H_{\zeta_{\beta}}(\zeta,\overline{\zeta})^T.
\end{align}
So 
\begin{align}\label{first order constancy criterion alt a}
\tilde\phi(\iota(\mathcal{K}_2))\cong\mathrm{span}\left\{(H(\zeta,\overline{\zeta})^T)^{-1}S_{\zeta_{\beta}}(\zeta,\overline{\zeta})H(\zeta,\overline{\zeta})^{-1}\,|\, {\beta=1,\dots,r}\right\}
\end{align}
and
\begin{align}\label{first order constancy criterion alt b}
\tilde\phi(\iota(\overline{\mathcal{K}_2}))\cong\mathrm{span}\left\{\left.\overline{S_{\zeta_{\beta}}(\zeta,\overline{\zeta})}\,\right|\,{\beta=1,\dots,r}\right\},
\end{align}
where the above congruence refers to the usual identifications between $s\times s$ symmetric matrices and $\mathfrak{csp}(\mathbb{C}\mathfrak{g}_{-1})_{0,\pm 2}$.

We have
\begin{align}
e_\alpha(\overline{S_{\zeta_{\beta}}(\zeta,\overline{\zeta})})&=
 \overline{S_{\zeta_{\beta},\overline{\zeta_{\alpha}}}(\zeta,\overline{\zeta})}=H_{\zeta_\alpha}(\zeta,\overline{\zeta})H(\zeta,\overline{\zeta})^{-1}\overline{S_{\zeta_\beta}(\zeta,\overline{\zeta})}+\overline{S_{\zeta_\beta}(\zeta,\overline{\zeta})}(H(\zeta,\overline{\zeta})^T)^{-1}H_{\zeta_\alpha}(\zeta,\overline{\zeta})^T\\
 &=-[\Omega_\alpha(\tilde \phi),\overline{S_{\zeta_\beta}(\zeta,\overline{\zeta})}]
\end{align}
by \eqref{rank condition 2}, and
\begin{align}
e_\alpha\big(\tilde\phi\big(\iota(e_\beta)\big)\big)&=-(H(\zeta,\overline{\zeta})^T)^{-1}H_{\zeta_\alpha}(\zeta,\overline{\zeta})^T(H(\zeta,\overline{\zeta})^T)^{-1}S_{\zeta_{\beta}}(\zeta,\overline{\zeta})H(\zeta,\overline{\zeta})^{-1}+\\
&+(H(\zeta,\overline{\zeta})^T)^{-1}S_{\zeta_{\alpha},\zeta_{\beta}}(\zeta,\overline{\zeta})H(\zeta,\overline{\zeta})^{-1}+\\
&-(H(\zeta,\overline{\zeta})^T)^{-1}S_{\zeta_{\beta}}(\zeta,\overline{\zeta})H(\zeta,\overline{\zeta})^{-1}H_{\zeta_\alpha}(\zeta,\overline{\zeta})H(\zeta,\overline{\zeta})^{-1}\\
&=(H(\zeta,\overline{\zeta})^T)^{-1}S_{\zeta_{\alpha},\zeta_{\beta}}(\zeta,\overline{\zeta})H(\zeta,\overline{\zeta})^{-1}-[\Omega_\alpha(\tilde \phi),\tilde\phi\big(\iota(e_\beta)\big)]
\end{align}
Thus with a chosen normalization condition $\big(O_{0,-2}^N(\tilde \phi(p)),O_{0,2}^N(\tilde \phi(p))\big)$, the conditions \eqref{first order constancy criterion b} and \eqref{first order constancy criterion a} can be expressed as existence of $B_{\alpha}\in \mathfrak{csp}_{0,0}(\mathbb{C}\mathfrak{g}_{-1})$ such that, for all $\alpha, \beta$,
\begin{align}\label{first order constancy criterion alt c}
-[\Omega_\alpha(\tilde \phi)(p)+B_{\alpha},\overline{S_{\zeta_\beta}(\zeta,\overline{\zeta})}]&=O_{0,-2}^N(\tilde \phi(p))(e_{\alpha},\overline{S_{\zeta_{\beta}}(\zeta,\overline{\zeta})})
\end{align}
and
\begin{align}\label{first order constancy criterion alt d}
\quad(H(\zeta,\overline{\zeta})^T)^{-1}S_{\zeta_{\alpha},\zeta_{\beta}}(\zeta,\overline{\zeta})H(\zeta,\overline{\zeta})^{-1}-[\Omega_\alpha(\tilde \phi)(p)+B_{\alpha},\tilde\phi\big(\iota(e_\beta)\big)]&=O_{0,2}^N(\tilde \phi(p))(e_{\alpha},S_{\beta}\big(\tilde \phi\big)).
\end{align}

The corresponding modified symbol at $\tilde\phi(p)$ with respect to normalization condition $\big(O_{0,-2}^N(\tilde \phi(p)),O_{0,2}^N(\tilde \phi(p))\big)$ is determined by the matrices
\[
H\big(\tilde \phi(p)\big)=H(\zeta,\overline{\zeta}),
\quad
\Xi_{\beta}\big(\tilde \phi(p)\big)=(H(\zeta,\overline{\zeta})^T)^{-1}S_{\zeta_{\beta}}(\zeta,\overline{\zeta}),
\]
and
\[
\Omega_{\beta}\big(\tilde \phi(p)\big)=\Omega_{\beta}(\tilde \phi)(p)+B_{\beta}
\]
in accordance with \eqref{key example Lform}, \eqref{key example Xi mat}, \eqref{key example Omega mat}, and Theorem \ref{modified symbols in normalized frames}, which provides a clear necessary condition for modified symbols.

\begin{lemma}\label{key examples prop}
    Suppose that a $2$-nondegenerate model has modified symbol $\mathfrak{g}_{\leq0}^{\rm{mod}}\big(\tilde \phi(p)\big)$ at $\tilde \phi(p)$ given by the above matrices
    \[
(H(\tilde \phi(p)),\Xi_1(\tilde \phi(p)),\dots,\Xi_{n-s}(\tilde \phi(p)),\Omega_1(\tilde \phi(p)),\dots,\Omega_{n-s}(\tilde \phi(p)))
\]
    with respect to any normalization condition $\big(O_{0,-2}^N(\tilde \phi(p)),O_{0,2}^N(\tilde \phi(p))\big)$ of the form 
    \begin{align}\label{normalization simplification a}
    O_{0,-2}^N(\tilde \phi(p))=0
    \end{align}
    and 
    \begin{align}\label{normalization simplification b}
    O_{0,2}^N(\tilde \phi(p))\big(e_{\alpha},S^{0,2}_\beta(\tilde \phi(p))\big)=O_{0,2}^N(\tilde \phi(p))\big(e_{\beta},S^{0,2}_\alpha(\tilde \phi(p))\big)
    \end{align}
    for all $\alpha,\beta$ and $S^{0,2}_\alpha(\tilde \phi(p))=\Xi_\alpha(\tilde \phi(p))H(\tilde \phi(p))^{-1}$. Then each $\Omega_{\alpha}(\tilde \phi(p))$ belongs to the normalizer
    \begin{align}
        N_{\mathfrak{csp}(\mathbb{C}\mathfrak{g}_{-1})_{0,0}}(\tilde \phi(\mathfrak{g}_{0,-2}(p)))
    \end{align}
    and the collection $\big(\Omega_1(\tilde \phi(p)),\dots,\Omega_{n-s}(\tilde \phi(p))\big)$ belongs to
       \begin{align}\label{g0+ is lie algebra}
       \left\{\left.(B_1,\dots,B_{n-s})\in \bigotimes_{1}^{n-s} \mathfrak{csp}(\mathbb{C}\mathfrak{g}_{-1})_{0,0}\,\right|\, \left[B_\alpha,S^{0,2}_\beta(\tilde \phi(p))\right]-\left[B_\beta,S^{0,2}_\alpha(\tilde \phi(p))\right]\in \mathfrak{g}_{0,2}(\tilde \phi(p)) \right\},
       \end{align}
    i.e., $ \pi_{0,2}([\mathfrak{g}_{0,+}^{\mathrm{mod}}(\tilde \phi(p)),\mathfrak{g}_{0,+}^{\mathrm{mod}}(\tilde \phi(p))])\subset \pi_{0,2}(\mathfrak{g}_{0,+}^{\mathrm{mod}}(\tilde \phi(p))).$
\end{lemma}
\begin{proof}
The claim follows for CR hypersurfaces given by defining equation \eqref{gen def fun}  directly from the discussion above. Indeed, from \eqref{first order constancy criterion alt c} we see that $\Omega_{\alpha}(\tilde \phi(p))$ belongs to the claimed normalizer and from \eqref{first order constancy criterion alt d} the second condition follows.
\end{proof}

\subsection{Abstract modified symbols}\label{Abstract modified symbols}

The information about modified symbols of CR hypersurfaces from Sections \ref{sec2.2} and \ref{sectionnonconstantmodifedsymbol} suggests the following definition of abstract modified symbols.

\begin{definition}\label{AMS}
    A subspace $\mathfrak{g}_{\leq0}^{\rm{mod}}\subset \mathbb{C}\mathfrak{g}_{-}\oplus \mathfrak{csp}(\mathbb{C}\mathfrak{g}_{-1})$ together with an antilinear involution $\sigma$ of $\mathbb{C}\mathfrak{g}_{-}\oplus \mathfrak{csp}(\mathbb{C}\mathfrak{g}_{-1})$ of the form in \eqref{involution formula} (for some nondegenerate Hermitian matrix $H(\tilde \phi):=\mathbf{H}$ and real number $h(\tilde \phi):=h$) is an \emph{abstract modified symbol} if it admits a decomposition $\mathfrak{g}_{\leq0}^{\rm{mod}}=\mathbb{C}\mathfrak{g}_{-}\oplus \mathfrak{g}_{0,-}^{\mathrm{mod}}\oplus \mathfrak{g}_{0,0}^{\mathrm{mod}}\oplus \mathfrak{g}_{0,+}^{\mathrm{mod}}$ such that
    \begin{itemize}
    \item $\sigma(\mathfrak{g}_{\leq0}^{\rm{mod}})\subset \mathfrak{g}_{\leq0}^{\rm{mod}},$ and $\sigma(\mathfrak{g}_{0,+}^{\mathrm{mod}})=\mathfrak{g}_{0,-}^{\mathrm{mod}}$,
    \item $\pi_{0,-2}(\mathfrak{g}_{0,+}^{\mathrm{mod}})=0$,
    \item $\mathfrak{g}_{0,0}^{\mathrm{mod}}\subset \mathfrak{csp}(\mathbb{C}\mathfrak{g}_{-1})_{0,0}$ is the maximal $\sigma$-invariant subalgebra of $\mathfrak{csp}(\mathbb{C}\mathfrak{g}_{-1})_{0,0}$ preserving $\pi_{0,2}(\mathfrak{g}_{0,+}^{\mathrm{mod}})$.
    \end{itemize}

    We say that the abstract modified symbol is \emph{realizable}  if it satisfies
    \[
     \pi_{0,-2}([\mathfrak{g}_{0,+}^{\mathrm{mod}},\mathfrak{g}_{0,-}^{\mathrm{mod}}])\subset \pi_{0,-2}(\mathfrak{g}_{0,-}^{\mathrm{mod}})
    \]
    and
     \[
      \pi_{0,2}([\mathfrak{g}_{0,+}^{\mathrm{mod}},\mathfrak{g}_{0,+}^{\mathrm{mod}}])\subset \pi_{0,2}(\mathfrak{g}_{0,+}^{\mathrm{mod}}).
    \]
\end{definition}
The decomposition $\mathfrak{g}_{\leq0}^{\rm{mod}}=\mathbb{C}\mathfrak{g}_{-}\oplus \mathfrak{g}_{0,-}^{\mathrm{mod}}\oplus \mathfrak{g}_{0,0}^{\mathrm{mod}}\oplus \mathfrak{g}_{0,+}^{\mathrm{mod}}$ is not unique. Two abstract modified symbols are equivalent if they are related via conjugation by an element of $\mathrm{CSp}(\mathbb{C}\mathfrak{g}_{-1})_{0,0}$, where, we stress, this conjugation must intertwine their respective involutions. Therefore, we can arrive to the following description of the modified symbol.

\begin{lemma}\label{symbol representation lemma}
There is bijection between abstract modified symbols and pairs of the form
\begin{align}\label{matrices rep mod sym}
\left(e^{ih}\mathbf{H},\mathrm{span}\left\{(S^{0,2}_1,\Omega_1),\dots,(S^{0,2}_{n-s},\Omega_{n-s})\right\}+(0,\mathfrak{g}_{0,0}^{\prime})\right),
\end{align}
where $\mathbf{H}$ is an $s\times s$ Hermitian matrix, $e^{ih}$ is a unit complex number, each $S_j$ is an $s\times s$ symmetric matrix of the form $S^{0,2}_\alpha=e^{-ih}\Xi_\alpha \mathbf{H}^{-1}$, each $\Omega_\alpha$ is an $s\times s$ matrix, and 
\[
\mathfrak{g}_{0,0}^{\prime}:=\left\{B\in \mathfrak{gl}(s,\mathbb{C})\,\left|\,B^T\mathbf{H}\overline{\Xi_\alpha}+\mathbf{H}\overline{\Xi_\alpha}B\in \mathrm{span}\{\mathbf{H}\overline{\Xi_\beta}\},\,BS^{0,2}_\alpha+S^{0,2}_\alpha B^T\in \mathrm{span}\{S^{0,2}_\beta\}\right.\right\}.
\]
The abstract modified symbol is realizable if and only if
\[\Omega_{\beta}\in N_{\mathfrak{csp}(\mathbb{C}\mathfrak{g}_{-1})_{0,0}}(\mathrm{span}\left\{\mathbf{H}\overline{S^{0,2}_\gamma}\mathbf{H}^T \right\})
\quad\quad\big(\iff
\Omega_\beta^T\mathbf{H}\overline{\Xi_\alpha}+\mathbf{H}\overline{\Xi_\alpha}\Omega_\beta\in \mathrm{span}\{\mathbf{H}\overline{\Xi_\gamma}\}\big)
\]
and 
\[
[\Omega_{\alpha},S^{0,2}_\beta]-[\Omega_{\beta},S^{0,2}_\alpha]=\Omega_{\alpha}S^{0,2}_\beta+S^{0,2}_\beta \Omega_{\alpha}^T-\Omega_{\beta}S^{0,2}_\alpha-S^{0,2}_\alpha \Omega_{\beta}^T\in \mathrm{span}\left\{S^{0,2}_\gamma \right\}\] for all $\alpha,\beta$.
\end{lemma}
\begin{proof}
It follows from the definition of abstract modified symbol that we have the claimed matrix $e^{ih}\mathbf{H}$, and that the elements of $\mathfrak{g}_{0,+}^{\mathrm{mod}}$ can be written in the form \eqref{modified symbol rep} for suitable matrices $S^{0,2}_\alpha$ and $\Omega_\alpha$. Conversely, $e^{ih}\mathbf{H}$ defines an antilinear-involution $\sigma$ of $\mathbb{C}\mathfrak{g}_{-}\oplus \mathfrak{csp}(\mathbb{C}\mathfrak{g}_{-1})$ of the form in \eqref{involution formula} and we can use \eqref{modified symbol rep} to define elements of $\mathfrak{g}_{0,+}^{\mathrm{mod}}$. It is then a simple computation to check that this defines the modified symbol completely, i.e., that $\mathfrak{g}_{0,0}^{\rm{mod}}$ is the sum of $\mathfrak{g}_{0,0}^{\prime}$ with multiples of the grading element. The realizibility conditions then directly translate into the claimed form.
\end{proof}
\begin{remark}\label{action on symbol rem}
    In order to obtain a distinguished pair 
    \[(e^{ih}\mathbf{H},\mathrm{span}\{(S^{0,2}_1,\Omega_1),\dots,(S^{0,2}_{n-s},\Omega_{n-s})\}+(0,\mathfrak{g}_{0,0}^{\prime}))\] representing an equivalence class of modified symbol, one has to consider the action of $\mathrm{CSp}(\mathbb{C}\mathfrak{g}_{-1})_{0,0}$ on the pair and pick suitable representatives of the orbits of this action. Clearly, one can fix $e^{ih}\mathbf{H}=\mathbf{H}$ to be a representative of a Hermitian matrix with given signature, while obtaining the other representatives is an open problem, in general.
\end{remark}

\section{Realizations of modified symbols with 2-nondegenerate model constructions}\label{Approaches to Hypersurface Realization}

In this section, we show that the conditions
\begin{align}\label{sufficient normalizer}
\pi_{0,-2}([\mathfrak{g}_{0,+}^{\mathrm{mod}},\mathfrak{g}_{0,-}^{\mathrm{mod}}])\subset \pi_{0,-2}(\mathfrak{g}_{0,-}^{\mathrm{mod}})
\quad\mbox{and}\quad
\pi_{0,2}([\mathfrak{g}_{0,+}^{\mathrm{mod}},\mathfrak{g}_{0,+}^{\mathrm{mod}}])\subset \pi_{0,2}(\mathfrak{g}_{0,+}^{\mathrm{mod}})
\end{align}
are not only necessary (as was shown in Lemma \ref{key examples prop} and Lemma \ref{symbol representation lemma}) but are also sufficient for the realizability of modified symbols. For every abstract modified symbol satisfying \eqref{sufficient normalizer}, we construct its realization as the modified symbol of a $2$-nondegenerate model at $0$ in Theorem \ref{hypersurface realization approach 2} that is in addition constant up to first order  at $0$.

We know from \cite[system (5.12), item iii]{sykes2023geometry} that the modified symbols $\mathfrak{g}_{\leq0}^{\mathrm{mod}}$ that can be attained on homogeneous structures all satisfy the condition \eqref{sufficient normalizer}, and can thus be realized via Theorem \ref{hypersurface realization approach 2}. We show moreover via Lemma \ref{Z-graded model lemma} that these realizations have maximal symmetry algebra dimensions among all structures with the same reduced modified symbol. In other words, we indeed construct \emph{homogeneous models} for the modified symbols that can be attained on homogeneous $2$-nondegenerate structures.

The proof of maximality of the infinitesimal symmetry algebra is based on comparing the Theorem \ref{hypersurface realization approach 2} construction with a generalization (in Section \ref{generalized Naruki}) of the construction (described in  \cite{naruki1970holomorphic}) of homogeneous Levi nondegenerate structures from Levi--Tanaka algebras. A conceptual strength of this generalization is that it produces homogeneous models by construction. A corresponding drawback is that it can only be applied to the modified symbols of homogeneous structures. A further drawback is that the construction yields CR hypersurfaces realizations given by defining functions expressed in non-holomorphic coordinates, and changing to holomorphic coordinates is not always tractable. Contrastingly, the construction in Theorem \ref{hypersurface realization approach 2} remedies these two drawbacks as it can be applied to any realizable modified symbol and yields defining equations that can be expressed in holomorphic coordinates more easily.

\subsection{Construction of homogeneous models}\label{generalized Naruki}

It is shown in \cite{sykes2023geometry} that if \begin{align}\label{red modified symbol decomposition}
\mathfrak{g}^{\mathrm{mod}}_{\leq 0}=\mathfrak{g}^{\mathrm{mod}}_{\leq 0}(\psi)\subset \mathbb{C}\mathfrak{g}_{-}\oplus \mathfrak{csp}(\mathbb{C}\mathfrak{g}_{-1})
\end{align}
is a modified symbol of a homogeneous $2$-nondegenerate CR hypersurface at a point $\psi\in \mathcal{E}$ then there exists a subspace $\mathfrak{g}^{\mathrm{red}}_{\leq 0}=\mathbb{C}\mathfrak{g}_{-}\oplus \mathfrak{g}_0^{\mathrm{red}}\subset\mathfrak{g}^{\mathrm{mod}}_{\leq0}$ with a further (non-canonical) decomposition
\begin{align}\label{red modified symbol decomposition a}
\mathfrak{g}_{0}^{\mathrm{red}}=\mathfrak{g}_{0,-}^{\mathrm{red}}\oplus \mathfrak{g}_{0,0}^{\mathrm{red}}\oplus \mathfrak{g}_{0,+}^{\mathrm{red}}
\end{align}
satisfying
\begin{itemize}
\item $\overline{\mathfrak{g}_{0}^{\mathrm{red}}}=\mathfrak{g}_{0}^{\mathrm{red}}$, $\overline{\mathfrak{g}_{0,-}^{\mathrm{red}}}=\mathfrak{g}_{0,+}^{\mathrm{red}}, \pi_{0,\pm 2}(\mathfrak{g}_{0,\pm}^{\mathrm{red}})=\mathfrak{g}_{0,\pm 2}$,
\item $[v,\mathfrak{g}_{-1,1}]\not\subset \mathfrak{g}_{-1,1} \quad\quad\forall\, v\in \mathfrak{g}_{0,-}^{\mathrm{red}}$,
\item $[v,\mathfrak{g}_{-1,-1}]\subset \mathfrak{g}_{-1,-1} \quad\quad\forall\, v\in \mathfrak{g}_{0,-}^{\mathrm{red}}$,
\item $[v,\mathfrak{g}_{-1,i}] \subset \mathfrak{g}_{-1,i} \quad\quad\forall\, v\in \mathfrak{g}_{0,0}^{\mathrm{red}},\, i\in\{-1,1\}$,
\end{itemize}
such that $\mathfrak{g}^{\mathrm{red}}_{\leq 0}$ is a subalgebra of $\mathbb{C}\mathfrak{g}_{-}\oplus \mathfrak{csp}(\mathbb{C}\mathfrak{g}_{-1})$. This association arises from the \emph{reduction procedure} of \cite{sykes2023geometry}, and we call $\mathfrak{g}_{0}^{\mathrm{red}}$ the reduced modified symbol of the homogeneous $2$-nondegenerate structure. More generally, we refer to subalgebras $\mathfrak{g}^{\mathrm{red}}_{\leq0}$ in $\mathbb{C}\mathfrak{g}_{-}\oplus \mathfrak{csp}(\mathbb{C}\mathfrak{g}_{-1})$ having all of this additional structure as \emph{abstract reduced modified symbols (ARMS)}.

Each ARMS generates a homogeneous CR hypersurface as follows. Let $\mathfrak{g}^{\mathrm{mod}}_{\leq 0}$ be an abstract modified symbol containing the ARMS $\mathfrak{g}^{\mathrm{red}}_{\leq 0}$, and let us fix a decomposition of $\mathfrak{g}^{\mathrm{red}}_{\leq 0}$ as in \eqref{red modified symbol decomposition} and \eqref{red modified symbol decomposition a}. Define
\[
\mathfrak{q}:=\mathfrak{g}_{-1,-1}\oplus \mathfrak{g}_{0,-}^{\mathrm{red}}\oplus \mathfrak{g}_{0,0}^{\mathrm{red}},
\quad\mbox{ and }\quad
\mathfrak{h}:=\mathfrak{g}_{-1,1}\oplus \mathfrak{g}_{0,+}^{\mathrm{red}}\oplus \mathfrak{g}_{0,0}^{\mathrm{red}}
\]
let $G^{\mathbb{C}}$ be the connected simply-connected Lie group of $\mathfrak{g}^{\mathrm{red}}_{\leq 0}$, let $G_{0,0}^{\mathrm{red}}\subset G^{\mathbb{C}}$ be the subgroup generated by $\mathfrak{g}_{0,0}^{\mathrm{red}}$, let $G=\Re G^{\mathbb{C}}$ be the subgroup generated by $\Re(\mathfrak{g}^{\mathrm{red}}_{\leq 0})$, and let $\pi:G^{\mathbb{C}}\to G^{\mathbb{C}}/G_{0,0}^{\mathrm{red}}$ be the canonical projection to the left-coset space. Letting $\mathcal{H}$ denote the invariant distribution on $G^{\mathbb{C}}/G_{0,0}^{\mathrm{red}}$ generated by $\mathfrak{h}/\mathfrak{g}_{0,0}^{\mathrm{red}}$, the distribution $\mathcal{H}$ defines a CR structure on $\pi(G)$, and it is shown in \cite{sykes2023geometry} that the CR structure
\begin{align}\label{Z-graded model}
\left(\pi(G), \mathcal{H}\right)
\end{align}
is a locally homogeneous $2$-nondegenerate hypersurface-type CR structure with $\mathfrak{g}^{\mathrm{mod}}_{\leq 0}$ as one of its modified symbols (c.f., \emph{flat CR structures} in \cite[Section 6]{sykes2023geometry}).

We can locally embed this hypersurface into the space
\[
\mathfrak{w}:=\mathfrak{g}_{-2,0}\oplus \mathfrak{g}_{-1,1}\oplus \mathfrak{g}_{0,+}^{\mathrm{red}}
\]
as follows, where all constructions are well defined on a sufficiently small neighborhood of the identity element in $G^{\mathbb{C}}$.
Let $W:=\mathrm{exp}(\mathfrak{w})$ and $\mathrm{exp}(\mathfrak{q})$ be the local subgroups of $G^{\mathbb{C}}$ generated by $\mathfrak{w}$ and $\mathfrak{q}$ respectively, and let 
\[
\Pi:G^{\mathbb{C}}/G_{0,0}^{\mathrm{red}}\to G^{\mathbb{C}}/\mathrm{exp}(\mathfrak{q})
\]
be the natural projection. Notice that 
\[
\left.\Pi\right|_{\pi(G)}:\pi(G)\to G^{\mathbb{C}}/\mathrm{exp}(\mathfrak{q}),
\]
is an immersion, whereas,
\begin{align}\label{hypersurface ambient space}
\left.\Pi\right|_{\pi(W)}:\pi(W)\to G^{\mathbb{C}}/\mathrm{exp}(\mathfrak{q}),
\end{align} 
is both an immersion and a submersion. After possibly replacing $G^{\mathbb{C}}$ with a sufficiently small neighborhood $U\subset G^{\mathbb{C}}$ such that $\left.\Pi\right|_{U/G_{0,0}^{\mathrm{red}}}$ is holomorphic and $\left.\Pi\right|_{(U/G_{0,0}^{\mathrm{red}})\cap \pi(W)}$ is a biholomorphism onto its image, the hypersurface-type CR structure on $\left(\pi(U\cap G), \mathcal{H}\right)$ will be equivalent to the structure that $\Pi\circ \pi(U\cap G)$ inherits from $(U\cap G^{\mathbb{C}})/\mathrm{exp}(\mathfrak{q})$ because $\left.\Pi\right|_{\pi(U\cap G)}$ is the restriction of a holomorphic map. Hence we have the following lemma.
\begin{lemma}\label{hypersurface realization lemma}
After replacing $G^{\mathbb{C}}$ with a sufficiently small neighborhood of the group's identity element, the submanifold $\Pi\circ \pi(G)$ is a real hypersurface in the complex manifold $G^{\mathbb{C}}/\mathrm{exp}(\mathfrak{q})$ equipped with the same CR structure as $\left(\pi(G), \mathcal{H}\right)$.
\end{lemma}

The map 
\begin{align}\label{hypersurfae realization map}
\left. \Pi\circ \pi\circ \mathrm{exp}\right|_{\mathfrak{w}}:\mathfrak{w}\to  G^{\mathbb{C}}/\mathrm{exp}(\mathfrak{q}),
\end{align}
is a local biholomorphism, so applying its inverse (locally) to $\Pi\circ \pi(G)$ yields a local hypersurface realization of $\left(\pi(G), \mathcal{H}\right)$ in the complex vector space $\mathfrak{w}$.

For the $2$-nondegenerate CR hypersurfaces in $\mathbb{C}^4$, it is proved in \cite{SykesHomogeneous} that this construction produces all of the homogeneous models. A striking property of these $\mathbb{C}^4$ homogeneous models can be generalized as follows: 

\begin{definition}\label{F filtration compatible}
Letting $\mathfrak{aut}(M)$ denote the algebra of infinitesimal symmetries of a connected locally homogeneous $2$-nondegenerate CR hypersurface $M$, $\mathfrak{aut}(M)$ \emph{admits a compatible $\mathbb{Z}$-grading}  if it has a $\mathbb{Z}$-graded decomposition
\begin{align}\label{F filtration z-grading}
\mathfrak{aut}(M)=\bigoplus_{k\geq -2}\mathfrak{m}_k
\end{align}
such that $[\mathfrak{m}_j,\mathfrak{m}_k]\subset \mathfrak{m}_{j+k}$ and 
$\mathbb{C}\mathfrak{m}_{-2}(p)= \mathbb{C}\mathfrak{g}_{-2}(p)$ and $\mathbb{C}\mathfrak{m}_{-1}(p)=\mathbb{C}\mathfrak{g}_{-1}(p).$
\end{definition}

We can also characterize how the component $\mathfrak{m}_{0}$ of $\mathfrak{aut}(M)$ with a compatible $\mathbb{Z}$-grading relates to the modified symbol and establish that homogeneous $2$-nondegenerate CR hypersurfaces with a compatible $\mathbb{Z}$-gradings are homogeneous models.
\begin{lemma}\label{Z-graded model lemma}
Consider a locally homogeneous $2$-nondegenerate CR structure $(M,\mathcal{H})$ with modified symbol $\mathfrak{g}_{\leq0}^{\mathrm{mod}}$ that admits a compatible $\mathbb{Z}$-grading \eqref{F filtration z-grading}. For $\tilde \phi \in \mathcal{E}$ over $p\in M$, $\tilde \phi(\mathbb{C}\mathfrak{m}_{-2}\oplus \mathbb{C}\mathfrak{m}_{-1}\oplus \mathbb{C}\mathfrak{m}_{0})$  has the structure of an ARMS $\mathfrak{g}_{\leq 0}^{\rm{red}}\subset \mathfrak{g}_{\leq0}^{\mathrm{mod}}(\tilde \phi)$ with bracket that is minus the Lie bracket of the vector fields realizing $\mathfrak{aut}(M)$ around $p$. Consequently, $(M,\mathcal{H})$ has maximal infinitesimal symmetry algebra dimension among all homogeneous structures with the same reduced modified symbol, and is moreover the locally unique homogeneous model admitting a compatible $\mathbb{Z}$-grading \eqref{F filtration z-grading}. Concretely, \eqref{Z-graded model} for the ARMS $\mathfrak{g}_{\leq 0}^{\rm{red}}$ is a homogeneous model of type $\mathfrak{g}_{\leq0}^{\mathrm{mod}}$ admitting a compatible $\mathbb{Z}$-grading \eqref{F filtration z-grading}.
\end{lemma}
\begin{proof}

The first claim follows from \cite[Lemma 2.5]{SykesHomogeneous}, because, using compatibility with the $\mathbb{Z}$-grading and following the \emph{reduction procedure} outlined in \cite[Section 2]{SykesHomogeneous}, one identifies $\mathbb{C}\mathfrak{m}_{0}$ with the degree zero part $\mathfrak{g}_{0}^{\rm{red}}$ of some modified symbol $\mathfrak{g}_{\leq0}^{\rm{red}}\subset \mathfrak{g}_{\leq0}^{\rm{mod}}$. 

Or, alternatively, one can observe that by the definition of compatible $\mathbb{Z}$-grading,  $\tilde \phi(\mathbb{C}\mathfrak{m}_{-2}\allowbreak\oplus \mathbb{C}\mathfrak{m}_{-1}\oplus \mathbb{C}\mathfrak{m}_{0})$ is a graded vector subspace of $\mathfrak{g}_{\leq0}^{\mathrm{mod}}(\tilde \phi)$. Thus, for the first claim, it would remain to show that the nontrivial Lie brackets that are $\mathbb{C}\mathfrak{m}_{0}\wedge \mathbb{C}\mathfrak{m}_{0}\to \mathbb{C}\mathfrak{m}_{0}$, $\mathbb{C}\mathfrak{m}_{0}\wedge \mathbb{C}\mathfrak{m}_{-2}\to \mathbb{C}\mathfrak{m}_{-2}$, $\mathbb{C}\mathfrak{m}_{0}\wedge \mathbb{C}\mathfrak{m}_{-1}\to \mathbb{C}\mathfrak{m}_{-1}$, and $\mathbb{C}\mathfrak{m}_{-1}\wedge \mathbb{C}\mathfrak{m}_{-1}\to \mathbb{C}\mathfrak{m}_{-2}$ coincide up to a sign with the brackets on the modified symbol, which for the later two follows from construction of the modified symbols and for the first two from Jacobi identity.

This immediately shows that $(M,\mathcal{H})$ is locally equivalent to the structure in \eqref{Z-graded model}.

To show that $(M,\mathcal{H})$ is a homogeneous model, let $(M^\prime,\mathcal{H}^\prime)$ be an arbitrary locally homogeneous $2$-nondegenerate structure with reduced modified symbol $\mathfrak{g}_{\leq0}^{\mathrm{red}}$. We will show 
\begin{align}\label{graded h-structures are maximal}
    \dim\mathfrak{aut}(M^\prime)\leq \dim\mathfrak{aut}(M).
\end{align}
The CR structure at a point $p\in M^\prime$ induces a filtration
\begin{align}
\mathfrak{aut}(M^\prime)=\mathfrak{a}_{-2}\supset \mathfrak{a}_{-1}\supset\cdots
\end{align}
of $\mathfrak{aut}(M^\prime)$ defined by
\[
\mathfrak{a}_{-1}:=\left\{X\in \mathfrak{aut}(M^\prime)\,\left|\, X(p)\in \mathcal{H}^\prime\oplus\overline{\mathcal{H}^\prime}\right.\right\}
\]
and
\[
\mathfrak{a}_{j}:=\left\{X\in \mathfrak{aut}(M^\prime)\,\left|\, [X,Y](p)\in \mathfrak{a}_{j-1}\quad\forall\,Y\in\mathfrak{a}_{-1}\right.\right\}\quad\quad\forall\,j\geq0.
\]
Lie brackets on $\mathfrak{aut}(M^\prime)$ induce a Lie algebra structure on the associated graded vector space
\[
\mathfrak{a}^{\mathrm{gr}}:=\bigoplus_{k\geq-2}\mathfrak{a}^{\mathrm{gr}}_k.
\]
The first claim implies that $\mathbb{C}\mathfrak{a}^{\mathrm{gr}}_{-2}\oplus\mathbb{C}\mathfrak{a}^{\mathrm{gr}}_{-1}\oplus \mathbb{C}\mathfrak{a}^{\mathrm{gr}}_{0}$ is isomorphic to a subalgebra in $\mathfrak{g}_{\leq0}^{\mathrm{red}}$, and moreover that $\mathbb{C}\mathfrak{a}^{\mathrm{gr}}_{0}$ contains a subspace $\mathfrak{a}^{\mathrm{gr}}_{0,0}$ isomorphic to a subalgebra in $\mathfrak{g}_{\leq0}^{\mathrm{red}}$ such that 
\begin{align}\label{a and g algebras}
\mathbb{C}\mathfrak{a}^{\mathrm{gr}}_{-2}\cong\mathbb{C}\mathfrak{g}_{-2},\quad
\mathbb{C}\mathfrak{a}^{\mathrm{gr}}_{-1}\cong\mathbb{C}\mathfrak{g}_{-1},\quad
\mathbb{C}\mathfrak{a}^{\mathrm{gr}}_{0}\cong \mathfrak{g}_{+}^{\mathrm{red}}\oplus \mathfrak{g}_{-}^{\mathrm{red}}\oplus \mathfrak{a}^{\mathrm{gr}}_{0,0},
\end{align}
for some decomposition of $\mathfrak{g}_{\leq0}^{\mathrm{red}}$. Thus $\mathbb{C}\mathfrak{a}^{\mathrm{gr}}_{-2}\oplus\mathbb{C}\mathfrak{a}^{\mathrm{gr}}_{-1}\oplus \mathbb{C}\mathfrak{a}^{\mathrm{gr}}_{0}$ has the structure of an ARMS, and the homogeneous structure of the form in \eqref{Z-graded model} generated by this ARMS will be equivalent to $(M, \mathcal{H})$ because $\mathbb{C}\mathfrak{a}^{\mathrm{gr}}_{-2}\oplus\mathbb{C}\mathfrak{a}^{\mathrm{gr}}_{-1}\oplus \mathbb{C}\mathfrak{a}^{\mathrm{gr}}_{0}$ and $\mathfrak{g}_{\leq0}^{\mathrm{red}}$ differ only in their isotropy components. The algebra $\mathfrak{a}$ will be in the infinitesimal symmetries of this structure generated by $\mathbb{C}\mathfrak{a}^{\mathrm{gr}}_{-2}\oplus\mathbb{C}\mathfrak{a}^{\mathrm{gr}}_{-1}\oplus \mathbb{C}\mathfrak{a}^{\mathrm{gr}}_{0}$, and hence $\mathfrak{a}$ is isomorphic to a subalgebra in $\mathfrak{aut}(M)$. Now \eqref{graded h-structures are maximal} follows from noting that $\mathfrak{aut}(M^\prime)$ and $\mathfrak{a}$ are isomorphic as vector spaces.
\end{proof}

Let us mention that the uniqueness from  Lemma \ref{Z-graded model lemma} does not exclude the possibility for existence of homogeneous structures without a compatible $\mathbb{Z}$-grading that are maximally symmetric with respect to their modified symbols. The next lemma describes broad cases in which we know this cannot occur, and, for completeness, additionally
notes the analogous classical results for Levi-nondegenerate structures.
\begin{lemma}\label{z grading models lemma}
Let $M$ be a locally homogeneous CR hypersurface that is either
\begin{enumerate}
\item Levi-nondegenerate, 
\item $2$-nondegenerate and recoverable (c.f., \cite{sykes2023geometry}), or
\item $2$-nondegenerate and its modified symbols are constant on the bundle $\mathcal{E}$ given in Definition \ref{constant symbol defintion} (this implies, in particular, that the bigraded symbol is regular in the terminology of \cite{porter2021absolute}).
\end{enumerate}
Then the  algebra $\mathfrak{aut}(M)$ admits a compatible $\mathbb{Z}$-grading if and only if $M$ is the (locally unique) maximally symmetric homogeneous model with given reduced modified symbol.
\end{lemma}
\begin{proof}
For case 1, the lemma is a well-known result, following from \cite{chern1974,tanaka1962pseudo}, and the Levi-nondegenerate homogeneous models are the the standard left-invariant CR structures on Heisenberg groups (described, for example, as quadrics in \cite{chern1974,tanaka1962pseudo}).

For cases 2 and 3, Lemma \ref{Z-graded model lemma} implies that there exists a unique $2$-nondegenerate homogeneous structure with the same reduced modified symbol as $M$ and symmetry algebra admitting a compatible $\mathbb{Z}$-grading. Moreover, in the setting of case 2, \cite[Theorem 6.2, part 3]{sykes2023geometry} implies that homogeneous models are locally unique. In the setting of case 3, \cite[Theorem 5.2]{sykes2023geometry} implies that the reduced modified symbol has the structure of a bigraded symbol and a Lie algebra, so \cite[Theorem 3.2, part 3]{porter2021absolute} applies, establishing locally uniqueness here as well.
\end{proof}

\subsection{Realizations of abstract modified symbols using 2-nondegenerate models}\label{second approach}

Let us start with a homogeneous space $\tilde{G}^{\mathbb{C}}/\tilde{Q},$ where $\tilde{G}^{\mathbb{C}}:=(\exp(\mathbb{C}\mathfrak{g}_-)\rtimes \mathrm{CSp}(\mathbb{C}\mathfrak{g}_{-1}))$ and $\tilde{Q}$ is the maximal closed subgroup of $\tilde{G}^{\mathbb{C}}$ with Lie algebra $\mathfrak{g}_{-1,-1}\oplus \mathfrak{csp}(\mathbb{C}\mathfrak{g}_{-1})_{0,0}\oplus \mathfrak{csp}(\mathbb{C}\mathfrak{g}_{-1})_{0,-2}.$ The homogeneous space $\tilde{G}^{\mathbb{C}}/\tilde{Q}$ plays the role of the complex manifold in which the CR hypersurface with maximal dimension of Levi kernel  (described in Example \ref{maximal kernel dimension}) is realized in \cite{gregorovic2021equivalence}. We start our construction by associating to an abstract modified symbol satisfying condition \eqref{sufficient normalizer} a higher codimensional CR submanifold in $\tilde{G}^{\mathbb{C}}/\tilde{Q}:$

For an abstract modified symbol $\mathfrak{g}^{\mathrm{mod}}_{\leq 0}$ with decomposition $\mathfrak{g}_0^{\mathrm{mod}}=\mathfrak{g}_{0,-}^{\mathrm{mod}}\oplus \mathfrak{g}_{0,0}^{\mathrm{mod}} \oplus \mathfrak{g}_{0,+}^{\mathrm{mod}}$ as in Definition \ref{AMS}, we can consider the local embedding of $\exp(\mathfrak{g}_-)\exp(\mathfrak{g}_{0,+}^{\mathrm{mod}})$ into $\tilde{G}^{\mathbb{C}}/\tilde{Q},$ which is a CR submanifold that has codimension $1$ plus the codimension of $\pi_{0,2}(\mathfrak{g}_{0,+}^{\mathrm{mod}})$ in $\mathfrak{csp}(\mathbb{C}\mathfrak{g}_{-1})_{0,2}$. 

In the second step, we find a complex submanifold $W\subset \tilde{G}^{\mathbb{C}}/\tilde{Q}$ containing $\exp(\mathfrak{g}_-)\exp(\mathfrak{g}_{0,+}^{\mathrm{mod}})$ as a CR hypersurface realizing the modified symbol at the origin. For this, we notice that for each element $X$ of neighborhood of $0$ in $\mathfrak{g}_{0,+}^{\mathrm{mod}}$ there is unique $\psi\in \mathrm{CSp}(\mathbb{C}\mathfrak{g}_{-1})_{0,0}$ depending holomorphically on $X$ such that $\exp(X)\psi\in \tilde{G}^{\mathbb{C}}/\tilde{Q}$ belongs to $\exp( \mathfrak{csp}(\mathbb{C}\mathfrak{g}_{-1})_{0,2})$. Indeed, locally $\psi=\exp(-\pi_{0,0}(X))$ by Baker--Campbell--Hausdorff formula, because $X=\pi_{0,2}(X)+\pi_{0,0}(X)$ and $\pi_{0,0}(X)$ preserves $\mathfrak{csp}(\mathbb{C}\mathfrak{g}_{-1})_{0,2}$. Thus we take 
\begin{align}\label{N submanifold}
W:=\exp(\mathfrak{g}_{-2,0}\oplus \mathfrak{g}_{-1,1})\{\exp(X)\exp(-\pi_{0,0}(X)): X\in \mathfrak{g}_{0,+}^{\mathrm{mod}}\}.
\end{align}
Since $W\subset\exp(\mathfrak{g}_{-2,0}\oplus \mathfrak{g}_{-1,1})\exp(\mathfrak{csp}(\mathbb{C}\mathfrak{g}_{-1})_{0,2})=\exp(\mathfrak{g}_{-2,0}\oplus \mathfrak{g}_{-1,1}\oplus \mathfrak{csp}(\mathbb{C}\mathfrak{g}_{-1})_{0,2})$, we can use the logarithm as holomorphic coordinates and see that $W$ is given as $\mathfrak{g}_{-2,0}\oplus \mathfrak{g}_{-1,1}$ plus a holomorphic image of $\mathfrak{g}_{0,+}^{\mathrm{mod}}$ in $ \mathfrak{csp}(\mathbb{C}\mathfrak{g}_{-1})_{0,2}.$ Thus $W$ is a complex submanifold of $\tilde{G}^{\mathbb{C}}/\tilde{Q}.$

\begin{theorem}\label{hypersurface realization approach 2}
Let $\mathfrak{g}^{\mathrm{mod}}_{\leq 0}$ be a realizable modified symbol (i.e., abstract modified symbol satisfying condition \eqref{sufficient normalizer}) with decomposition $\mathfrak{g}_0^{\mathrm{mod}}=\mathfrak{g}_{0,-}^{\mathrm{mod}}\oplus \mathfrak{g}_{0,0}^{\mathrm{mod}} \oplus \mathfrak{g}_{0,+}^{\mathrm{mod}}$ and consider the above complex manifold $\tilde{G}^{\mathbb{C}}/\tilde{Q}$ with its complex submanifold $W$ of \eqref{N submanifold}. Locally, the submanifold $\exp(\mathfrak{g}_-)\exp(\mathfrak{g}_{0,+}^{\mathrm{mod}})$ of $\tilde{G}^{\mathbb{C}}/\tilde{Q}$ is contained as an $\exp(\mathfrak{g}_-)$-invariant $2$-nondegenerate model in $W$. At $\exp(0)$, this CR hypersurface has modified symbol $\mathfrak{g}_{\leq 0}^{\mathrm{mod}}$ constant up to first order.
\end{theorem}
\begin{proof}
By construction $\exp(\mathfrak{g}_-)\exp(\mathfrak{g}_{0,+}^{\mathrm{mod}})$ is an $\exp(\mathfrak{g}_-)$-invariant CR hypersurface in $W$. The $\exp(\mathfrak{g}_-)$-invariance ensures uniform Levi-degeneracy, because $\exp(\mathfrak{g}_{0,+}^{\mathrm{mod}})$ is a leaf of the Levi kernel through $\exp(0)$. So we need to show that it is $2$-nondegenerate at $\exp(0)$ and compute its modified symbol. 

Let us use the logarithmic coordinates and use the description \eqref{csp representation} of $\mathbb{C}\mathfrak{g}_{-}\oplus\mathfrak{csp}(\mathbb{C}\mathfrak{g}_{-1})$. To describe the real form $\mathfrak{g}_{-}$ via \eqref{csp representation}, we set:
\begin{itemize}
    \item $S^{0,-2}=L=0,c=0$,
    \item  $S^{0,2}=\mathbf{S}(\zeta)=\ln(\exp(X)\exp(-\pi_{0,0}(X)))\in \mathfrak{g}_{0,2}$ for $X=\sum_{\alpha=1}^{r} \zeta_\alpha e_\alpha \in \mathfrak{g}_{0,+}^{\mathrm{mod}}$,
    \item $v_1=\frac12(Y-iI(Y))$ for $Y\in \mathfrak{g}_{-1}$, and $I$ the complex structure on $\mathfrak{g}_{-1}$,
    \item $v_2=\frac12\mathbf{H}(Y+iI(Y))\in \mathfrak{g}_{-1,-1}$,
    \item $u=Z$ for $Z\in \mathfrak{g}_{-2}$.
\end{itemize}
   Thus the local embedding takes the form 
   \begin{align}\label{embedding formula}
   \ln\left(\exp\left(Y+Z\right)\exp(\mathbf{S}(\zeta))\exp(-v_2)\right)\in  \mathfrak{g}_{-2,0}\oplus \mathfrak{g}_{-1,1} \oplus \mathfrak{csp}(\mathbb{C}\mathfrak{g}_{-1})_{0,2},
   \end{align}
   because $\exp(-v_2)\in \tilde{Q}$ is such that $\exp(Y+Z)\exp(\mathbf{S}(\zeta))\exp(-v_2)\in W.$ In particular, $z=v_1-\mathbf{S}(\zeta)\mathbf{H}\overline{v_1}$ is the part of \eqref{embedding formula} in $\mathfrak{g}_{-1,1}$ and $w=iu+\overline{v_1}^T\mathbf{H}^T(v_1-\mathbf{S}(\zeta)\mathbf{H}\overline{v_1})=iu+z\mathbf{H}\overline{v_1}$ is the part of \eqref{embedding formula} in $\mathfrak{g}_{-2,0}.$ Thus 
   \begin{align}
   v_1&=(\mathrm{Id}-\mathbf{S}(\zeta)\mathbf{H}\overline{\mathbf{S}(\zeta)}\mathbf{H}^T)^{-1}(z+\mathbf{S}(\zeta)\mathbf{H}\overline{z}),\\
   \overline{v_1}&=(\mathrm{Id}-\overline{\mathbf{S}(\zeta)}\mathbf{H}^T\mathbf{S}(\zeta)\mathbf{H})^{-1}(\overline{z}+\overline{\mathbf{S}(\zeta)}\mathbf{H}^Tz),
   \end{align}
   and the defining equation of this CR hypersurface is
\begin{align}\label{general def eqn C4}
\Re(w)&=\frac12(z^T\mathbf{H}\overline{v_1}+\overline{z}^T\mathbf{H}^Tv_1)\\
  &=\frac12(z^T\mathbf{H}(\mathrm{Id}-\overline{\mathbf{S}(\zeta)}\mathbf{H}^T\mathbf{S}(\zeta)\mathbf{H})^{-1}(\overline{z}+\overline{\mathbf{S}(\zeta)}\mathbf{H}^Tz)\\
  &\indent+\overline{z}^T\mathbf{H}^T(\mathrm{Id}-\mathbf{S}(\zeta)\mathbf{H}\overline{\mathbf{S}(\zeta)}\mathbf{H}^T)^{-1}(z+\mathbf{S}(\zeta)\mathbf{H}\overline{z})
\end{align}
Therefore, \eqref{general def eqn C4} is of the form \eqref{gen def fun} for
 \eqref{H from S2 formula} and \eqref{S from S2 formula} from Theorem \ref{general formula for Cn}, but with $\mathbf{S}$ as specified above, i.e., the CR hypersurface is a $2$-nondegenerate model.

So in order to determine the modified symbol at $\Im(w)=0,z=0,\zeta=0$, according to Section \ref{Key examples of modified symbol calculations} we need to compute the following values:
\begin{align}
H(0,0)&=\mathbf{H},\
S(0,0)=0,\
H_{\zeta_{\beta}}(0,0)=0,\
S_{\overline{\zeta_{\beta}}}(0,0)=0,\
S_{\zeta_{\alpha}\overline{\zeta_{\beta}}}(0,0)=0,\\
S_{\zeta_{\beta}}(0,0)&=\mathbf{H}^T\mathbf{S}_{\zeta_{\beta}}(0)\mathbf{H}=\mathbf{H}^T\Xi_{\beta},\\\label{computations for MS calculations}
S_{\zeta_{\alpha}\overline{\zeta_{\beta}}}(0,0)&=\mathbf{H}^T\mathbf{S}_{\zeta_{\alpha}\zeta_{\beta}}(0)\mathbf{H}=\frac12\mathbf{H}^T\left(\frac12[\Omega_\alpha,\Xi_{\beta}\mathbf{H}^{-1}]+\frac12[\Omega_\beta,\Xi_{\alpha}\mathbf{H}^{-1}]\right)\mathbf{H},
\end{align}
where we use $\mathbf{S}(0)=0$, $\mathbf{S}_{\zeta_{\alpha}}(0)=\Xi_{\alpha}\mathbf{H}^{-1}$, and $\mathbf{S}_{\zeta_{\alpha}\zeta_{\beta}}(0)=\frac12[\Omega_\alpha,\Xi_{\beta}\mathbf{H}^{-1}]+\frac12[\Omega_\beta,\Xi_{\alpha}\mathbf{H}^{-1}]$, which follows from the expansion 
\begin{align}\label{S2 formula}
\mathbf{S}(\zeta)=\sum_\beta \Xi_\beta \mathbf{H}^{-1}\zeta_\beta+\frac12\sum_{\alpha,\beta}[\Omega_\alpha,\Xi_{\beta}\mathbf{H}^{-1}]\zeta_{\alpha}\zeta_{\beta}+O(\zeta^3)
\end{align}
derived using the Baker-Hausdorff-Campbell formula. The bracket in \eqref{computations for MS calculations} should  as usual be understood as a bracket of elements of $\mathfrak{csp}(\mathbb{C}\mathfrak{g}_{-1})_{0,0}$ with elements of $\mathfrak{csp}(\mathbb{C}\mathfrak{g}_{-1})_{0,2}$) (and not as a commutator of $s\times s$ matrices). In particular, if we want to achieve constancy up to first order, then the conditions \eqref{first order constancy criterion alt c} and \eqref{first order constancy criterion alt d} become
\begin{align}\label{FOC conditions}
[B_{\alpha},\mathbf{H}\overline{\Xi_{\beta}}]\in \tilde\phi(\iota(\overline{\mathcal{K}_2})), \\
\frac12[\Omega_\alpha,\Xi_{\beta}\mathbf{H}^{-1}]+\frac12[\Omega_\beta,\Xi_{\alpha}\mathbf{H}^{-1}]-[B_{\alpha},\Xi_{\beta}\mathbf{H}^{-1}]\in \tilde\phi(\iota(\mathcal{K}_2)),
\end{align}
where here again brackets are taken in $\mathfrak{csp}(\mathbb{C}\mathfrak{g}_{-1})$. Taking $B_{\alpha}=\Omega_\alpha$ indeed satisfies these conditions, because the first condition in \eqref{sufficient normalizer} then directly implies the first condition in \eqref{FOC conditions}, whereas the second condition in \eqref{sufficient normalizer} implies $[\Omega_\alpha,\Xi_{\beta}\mathbf{H}^{-1}]\equiv[\Omega_\beta,\Xi_{\alpha}\mathbf{H}^{-1}]\pmod{\tilde\phi(\iota(\mathcal{K}_2))}$ from which the second condition in \eqref{FOC conditions} follows. Thus we know that the CR hypersurface has the claimed modified symbol at $\exp(0)$ constant up to first order.
\end{proof}

An essential difference between the constructions of Theorem \ref{hypersurface realization approach 2} and Lemma \ref{hypersurface realization lemma} is seen by contrasting the submanifolds $W$ defined differently in sections \ref{generalized Naruki} and \ref{second approach}. To explore this concretely, let us relabel this submanifold from section \ref{generalized Naruki} as
\begin{align}\label{Naruki method ambiant space}
W^\prime:=\mathrm{exp}\left(\mathfrak{g}_{-2,0}\oplus \mathfrak{g}_{-1,1}\oplus \mathfrak{g}_{0,+}^{\mathrm{red}}\right).
\end{align}
Alternatively, we can write $W^\prime=\mathrm{exp}\left(\mathfrak{g}_{-2,0}\oplus \mathfrak{g}_{-1,1}\right)\mathrm{exp}\left(\mathfrak{g}_{0,+}^{\mathrm{red}}\right)$ and, for $Y\in\mathfrak{g}_{-2,0}\oplus \mathfrak{g}_{-1,1}$ and $X\in\mathfrak{g}_{0,+}^{\mathrm{red}}$, \begin{align}\label{ambient space deformation}
\mathrm{exp}(Y)\mathrm{exp}(X)\mapsto \mathrm{exp}(Y)\mathrm{exp}(X)\mathrm{exp}(-\pi_{0,0}(X))
\end{align}
defines a local biholomorphism between $W^\prime\subset G^{\mathbb{C}}$ and $W\subset\tilde G^{\mathbb{C}}$.

\begin{lemma}\label{construction from model data lemma}
If $\mathfrak{g}_{\leq0}^{\mathrm{red}}\subset\mathfrak{g}_{\leq0}^{\mathrm{mod}}$ is a reduced modified symbol of a homogeneous $2$-nondegenerate CR hypersurface structure then the $2$-nondegenerate model constructed from $\mathfrak{g}_{\leq0}^{\mathrm{mod}}$ in Theorem \ref{hypersurface realization approach 2} for decomposition $\mathfrak{g}_{0}^{\mathrm{mod}}=\mathfrak{g}_{0,-}^{\mathrm{red}}\oplus\mathfrak{g}_{0,0}\oplus\mathfrak{g}_{0,+}^{\mathrm{red}}$ is locally equivalent to the homogeneous CR hypersurface constructed in Lemma \ref{hypersurface realization lemma} for the reduced modified symbol $\mathfrak{g}_{0}^{\mathrm{red}}=\mathfrak{g}_{0,-}^{\mathrm{red}}\oplus\mathfrak{g}_{0,0}^{\mathrm{red}}\oplus\mathfrak{g}_{0,+}^{\mathrm{red}}$.
\end{lemma}
\begin{proof}
Let $\mathfrak{q}:=\mathfrak{g}_{-1,-1}\oplus \mathfrak{g}_{0,-}^{\mathrm{red}}\oplus \mathfrak{g}_{0,0}^{\mathrm{red}}$ with $\mathfrak{q}_{0}:=\mathfrak{g}_{0,-}^{\mathrm{red}}\oplus \mathfrak{g}_{0,0}^{\mathrm{red}}$.
We have $\Re(\mathfrak{g}_{0}^{\mathrm{red}})+\mathfrak{q}_{0}=\mathfrak{g}_{0,+}^{\mathrm{red}}+\mathfrak{q}_{0}$, which implies that $\mathrm{exp}\left(\mathfrak{g}_{-}\right)\mathrm{exp}\left(\Re(\mathfrak{g}_{0}^{\mathrm{red}})\right)$ and $\mathrm{exp}\left(\mathfrak{g}_{-}\right)\mathrm{exp}\left(\mathfrak{g}_{0,+}^{\mathrm{red}}\right)$ represent the same submanifold in $G^{\mathbb{C}}/\mathrm{exp}(\mathfrak{q})$ near $\exp(0)$. It follows that the natural embedding of $G^{\mathbb{C}}/\mathrm{exp}(\mathfrak{q})$ into $\tilde G^{\mathbb{C}}/\tilde Q$ maps the CR hypersurface of Lemma \ref{hypersurface realization lemma} onto the CR hypersurface of Theorem \ref{hypersurface realization approach 2}. Since \eqref{ambient space deformation} is a biholomorphism and $W^\prime$ is biholomorphic to $G^{\mathbb{C}}/\mathrm{exp}(\mathfrak{q})$, it follows that this embedding furthermore identifies the ambient space of Lemma \ref{hypersurface realization lemma} with the submanifold $W$ in Theorem \ref{hypersurface realization approach 2}.
\end{proof}

\section{Infinitesimal symmetry algebras}\label{Symmetry Algebras}

\subsection{Symmetries transversal to the Levi kernel}\label{Symmetries transversal to the Levi kernel}

In this section, we will prove that $2$-nondegenerate models always possesses a $(2s+1)$-dimensional Lie subalgebra of holomorphic infinitesimal  symmetries everywhere transversal to the Levi kernel.

In the following proposition, we write $\frac{\partial}{\partial z}$ to denote the length $s$ column vector whose $j$th entry is $\frac{\partial}{\partial z_j}$.
\begin{proposition}\label{transkernel symmetry formula}
    Let $M$ be a $2$-nondegenerate model. For $a\in \mathbb{C}^s$ and $b\in \mathbb{R}$, the real part of 
    \begin{align}       
    X=&2\left(bi+\left(\overline{a}^T\overline{H(0,\overline{\zeta})}+\left(a^T\mathbf{H}-\overline{a}^TS(\zeta,0)+\overline{a}^TS(0,0)\right)H(\zeta,0)^{-1}\overline{S(0,\overline{\zeta})} \right)z\right)\frac{\partial}{\partial w} \\
    &+\left(a^T\mathbf{H}-\overline{a}^TS(\zeta,0)+\overline{a}^TS(0,0)\right)H(\zeta,0)^{-1}\frac{\partial}{\partial z}
    \end{align}
    is an infinitesimal symmetry of $M$, where $\mathbf{H}=H(0,0)$.
\end{proposition}
\begin{proof}
    This is trivial for the $a=0$ case, so let us consider the complimentary case with $b=0$.

    The one-parameter  family of holomorphic transformations corresponding to the flow of $X$ takes form
\begin{align}
    w\mapsto w+2Azt+ABt^2,
    \quad 
    z\mapsto z+Bt,
    \quad\mbox{ and }\quad \zeta\mapsto\zeta
\end{align}
where
\begin{align}
    A:=\left(\overline{a}^T\overline{H(0,\overline{\zeta})}+B^T\overline{S(0,\overline{\zeta})} \right)
    \quad\mbox{ and }\quad 
    B:=(H(\zeta,0)^{-1})^T \left(\mathbf{H}^Ta-S(\zeta,0)\overline{a}+S(0,0)\overline{a}\right).
\end{align}
Therefore, the defining equation \eqref{gen def fun} transforms to the form
\begin{align}
\Re(w)&=z^TH(\zeta,\overline{\zeta})\overline{z}+\Re(\overline{z}^TS(\zeta,\overline{\zeta})\overline{z})\\
    &+\left(B^TH(\zeta,\overline{\zeta})\overline{z}+z^TH(\zeta,\overline{\zeta})\overline{B}+2\Re(\overline{z}^TS(\zeta,\overline{\zeta})\overline{B})-2\Re(Az) \right)t\\
    &+\left(B^TH(\zeta,\overline{\zeta})\overline{B}+\Re(\overline{B}^TS(\zeta,\overline{\zeta})\overline{B})-\Re(AB) \right)t^2.
\end{align}
As this CR hypersurface is still uniformly 2-nondegenerate, the additional entries linear in $z,\overline{z}$ or scalar satisfy additional PDE's that we derive as follows. Firstly, the Levi form is transformed to 
\[
\bgroup
\arraycolsep=2pt
\mathcal{L}^1=
\left(
\begin{array}{c:c}
H(\zeta,\overline{\zeta}) & H_{\overline{\zeta}}(\zeta,\overline{\zeta})\overline{z}+\overline{S_{\zeta}(\zeta,\overline{\zeta})}z+V_{\overline{\zeta}}(\zeta,\overline{\zeta})\vspace{2pt} \\\hdashline
z^TH_{\zeta}(\zeta,\overline{\zeta})+\overline{z}^TS_{\zeta}(\zeta,\overline{\zeta})+\overline{V_{\overline{\zeta}}(\zeta,\overline{\zeta})}^T  & \parbox{6.2cm}{\vspace{4pt}$z^TH_{\zeta\overline{\zeta}}(\zeta,\overline{\zeta})\overline{z}+\Re(\overline{z}^TS_{\zeta\overline{\zeta}}(\zeta,\overline{\zeta})\overline{z})+$\\$z^TV_{\zeta\overline{\zeta}}(\zeta,\overline{\zeta})+
\overline{V_{\zeta\overline{\zeta}}(\zeta,\overline{\zeta})}^T\overline{z}+W_{\zeta\overline{\zeta}}(\zeta,\overline{\zeta})$}  \\
\end{array}
\right),
\egroup
\]
where $V(\zeta,\overline{\zeta})=H(\zeta,\overline{\zeta})\overline{B}+\overline{S(\zeta,\overline{\zeta})}B-A^T$ and $W(\zeta,\overline{\zeta})=B^TH(\zeta,\overline{\zeta})\overline{B}+\Re(\overline{B}^TS(\zeta,\overline{\zeta})\overline{B})-\Re(AB)$. Since $\mathcal{L}^1$ has rank $a$, adding an appropriate multiple of the first block row to the second then provides the PDE system as before (cf. \eqref{rank condition 2}), but now there are additional terms linear in $z$ or scalar in $z$. This provides the equations
\begin{align}
    V_{\zeta\overline{\zeta}}(\zeta,\overline{\zeta})&=H_{\zeta}(\zeta,\overline{\zeta})H(\zeta,\overline{\zeta})^{-1}V_{\overline{\zeta}}(\zeta,\overline{\zeta})+\overline{S_{\zeta}(\zeta,\overline{\zeta})}(H(\zeta,\overline{\zeta})^{-1})^T\overline{V_{\overline{\zeta}}(\zeta,\overline{\zeta})},\\
    W_{\zeta\overline{\zeta}}(\zeta,\overline{\zeta})&=\overline{V_{\overline{\zeta}}(\zeta,\overline{\zeta})}^TH(\zeta,\overline{\zeta})^{-1}V_{\overline{\zeta}}.
\end{align}
Using the real analyticity assumption, let us express $V(\zeta,\overline{\zeta})=\sum_{i,j\geq 0} V_{ij}\zeta^i\overline{\zeta}^j$ and $W(\zeta,\overline{\zeta})=\sum_{i,j\geq 0} W_{ij}\zeta^i\overline{\zeta}^j$ with $W_{ij}=\overline{W_{ji}}.$ Therefore, $V_{i,j},W_{i,j}$ for $i,j>0$ are determined by $V_{k,l},W_{k,l}$ for $k+l<i+j.$ Thus if $V_{k,0}=V_{0,l}=W_{k,0}=0$ for all $k,l$, then $V$ and $W$ will also vanish. This indeed happens, because we can directly verify:
\begin{align}
V(\zeta,0)&=H(\zeta,0)\overline{B(0)}+\overline{S(0,\overline{\zeta})}B-A^T=H(\zeta,0)\overline{a}-\overline{H(0,\overline{\zeta})}^T\overline{a}=0\\
V(0,\overline{\zeta})&=H(0,\overline{\zeta})\overline{B}+\overline{S(\zeta,0)}B(0)-A^T(0)=\left(\mathbf{H}\overline{a}-\overline{S(\zeta,0)}a+\overline{S(0,0)}a\right)\\
&\quad\quad+\overline{S(\zeta,0)}a-\left(\mathbf{H}\overline{a}+\overline{S(0,0)}a\right)\\
&=0\\
W(\zeta,0)&=B^TH(\zeta,0)\overline{B(0)}+\frac12(B^T\overline{S(0,\overline{\zeta})}B+\overline{B(0)}^TS(\zeta,0)\overline{B(0)})-\frac12(AB+\overline{A(0)}\overline{B(0)})\\
&=B^TH(\zeta,0)\overline{a}+\frac12\overline{a}^TS(\zeta,0)\overline{a}-\frac12(\overline{a}^T\overline{H(0,\overline{\zeta})}B+a^T\mathbf{H}\overline{a}+\overline{a}^TS(0,0)\overline{a})\\
&=
\frac12(B^TH(\zeta,0)\overline{a}+\overline{a}^TS(\zeta,0)\overline{a}-a^T\mathbf{H}\overline{a}-\overline{a}^TS(0,0)\overline{a})=0.
\end{align}
\end{proof}
Corresponding to the normalized defining equations of Theorem \ref{general formula for Cn}, substituting \eqref{H from S2 formula} and \eqref{S from S2 formula} into Proposition \ref{transkernel symmetry formula} yields
\begin{corollary}\label{symmetries for main theorem}
For $a\in \mathbb{C}^s$ and $b\in \mathbb{R}$, the real part of 
\begin{align}\label{transkernel}
X=2\left(bi+\overline{a}^T\mathbf{H}^Tz\right)\frac{\partial}{\partial w}+\left(a^T-\overline{a}^T\mathbf{H}^T\mathbf{S}(\zeta)\right)\frac{\partial}{\partial z}
\end{align}
is an infinitesimal symmetry of the hypersurface given by \eqref{gen def fun}, \eqref{H from S2 formula}, and \eqref{S from S2 formula}.
\end{corollary}

\subsection{Symmetries in the degree zero isotropy}\label{Symmetries in the degree zero isotropy}

In the case of the CR hypersurfaces produced by the Theorem \ref{hypersurface realization approach 2}, we can characterize a large set of the infinitesimal symmetries in $\mathfrak{a}_0^{\mathrm{gr}}$ (as labeled in Section \ref{generalized Naruki}) that belong to the isotropy.

\begin{proposition}\label{isotropy prop}
    Let $(\mathfrak{g}^{\mathrm{mod}}_{\leq 0},\sigma)$ be a realizable modified symbol with decomposition $\mathfrak{g}_0^{\mathrm{mod}}=\mathfrak{g}_{0,-}^{\mathrm{mod}}\oplus \mathfrak{g}_{0,0}^{\mathrm{mod}} \oplus \mathfrak{g}_{0,+}^{\mathrm{mod}}$ and let $x\in \mathfrak{g}_{0,0}^{\prime}$ (defined in Lemma \ref{symbol representation lemma}) be an element whose adjoint action preserves this decomposition. The real part of
    \[
    X=\frac12z^T((L-(\mathbf{H}^T)^{-1}L^*\mathbf{H}^T)\frac{\partial}{\partial z}+c_{\alpha,\beta}\zeta_\beta\frac{\partial}{\partial \zeta_\alpha}
    \]
    is an infinitesimal symmetry of the $2$-nondegenerate model constructed from $\mathfrak{g}^{\mathrm{mod}}_{\leq 0}$ by the Theorem \ref{hypersurface realization approach 2}, where $L$ and $c_{\alpha,\beta}$ are defined by
        \[
    x=\left(
\begin{array}{cccc}
0& 0 & 0 &  0 \\
0 & L & 0 &  0 \\
0 & 0 & -L^T & 0 \\
0 & 0 & 0 &  0
\end{array}
\right)\in \mathfrak{g}^{\mathrm{mod}}_{0,0}
    \] 
    and $\sum^r_{\beta=1}c_{\alpha,\beta}e_\beta=\frac12[x+\sigma(x),e_{\alpha}].$
\end{proposition}
\begin{proof}
   As $x$ preserves the decomposition  $\mathfrak{g}_0^{\mathrm{mod}}=\mathfrak{g}_{0,-}^{\mathrm{mod}}\oplus \mathfrak{g}_{0,0}^{\mathrm{mod}} \oplus \mathfrak{g}_{0,+}^{\mathrm{mod}}$, there is linear transformation $\hat x: \zeta_\beta\mapsto \hat{x}_{\beta}(\zeta)$ such that 
    \[\sum^r_{1=\alpha}\hat x_\alpha(\zeta) e_\alpha  =\exp\left(\frac12\mathrm{ad}(x+\sigma(x))\right)\sum^r_{1=\alpha}e_\alpha\zeta_\alpha\] for $\sum^r_{1=\alpha}\zeta_\alpha e_\alpha \in \mathfrak{g}^{\mathrm{mod}}_{0,+}$. Let $\mathbf{H}$ be the Hermitian matrix such that $\sigma$ has the form of the involution in \eqref{involution formula} with $\mathbf{H}$ substituted for $H(\tilde \phi)$ and $h(\tilde \phi)=1$. 
    Then we check that the holomorphic transformation $w\mapsto w$, $z\mapsto A:=\exp\left(\frac12(L-(\mathbf{H}^T)^{-1}L^*\mathbf{H}^T)\right)z$ and $\zeta\mapsto \hat{x}(\zeta)$, which comes from the flow of $X+\overline{X}$, is a symmetry of the defining equation given by \eqref{gen def fun} and \eqref{general def eqn C4}.
    
    As we are acting by grading preserving linear maps, it suffices to check that
    \begin{align}
        A^TH\left(\hat x(\zeta),\overline{\hat x(\zeta)}\right)\overline{A}=H\left(\zeta,\overline{\zeta}\right),\\
        A^TS\left(\hat x(\zeta),\overline{\hat x(\zeta)}\right)A=S\left(\zeta,\overline{\zeta}\right).
    \end{align}
    By definition, $A^T\mathbf{H}\overline{A}=\mathbf{H}$ and 
    \begin{align}
    &A^{-1}\mathbf{S}(\hat x(\zeta))(A^T)^{-1}=\exp\left(-\frac12\mathrm{ad}(x+\sigma(x))\right)\mathbf{S}(\hat x(\zeta))\\
    &=\exp\left(-\frac12\mathrm{ad}(x+\sigma(x))\right) \ln\left(\exp\left( \sum^r_{1=\alpha}\hat x_\alpha(\zeta)e_\alpha \right)\exp\left(-\pi_{0,0}(\sum^r_{1=\alpha}\hat x_\alpha(\zeta)e_\alpha )\right)\right)\\
    &=\mathbf{S}(\zeta)
    \end{align}
    and the claim follows, because $\exp\circ \mathrm{ad}=\mathrm{Ad} \circ \exp$ is compatible with the operations in formulas \eqref{general def eqn C4}.
\end{proof}

\section{Proofs of the main theorems}

\subsection{Proof of Theorem \ref{model perturbations proposition}}\label{model perturbations section}

We start by proving that if the CR hypersurface
\[\Re(w)=z^TH(\zeta,\overline{\zeta})\overline{z}+\Re(\overline{z}^TS(\zeta,\overline{\zeta})\overline{z})+O(3)\]
is uniformly 2-nondegenerate, then the model hypersurface \eqref{gen def fun} is uniformly 2-nondegenerate. The notation $O(k)$ stands in for terms of weighted degree at least $k$ with respect to the weights \eqref{grading convention}. To compute the Levi form, we start with generators of the CR distribution (given by \cite[Proposition 1.6.4]{baouendi1999real}) that we expand according to the weighted degree as follows
\begin{align}
    \tilde f_a:=&\frac{\partial}{\partial z_a}-i(H_{a}(\zeta,\overline{\zeta})\overline{z}+\overline{S_{a}(\zeta,\overline{\zeta})}z)\frac{\partial}{\partial \Im(w)}+O(2)\frac{\partial}{\partial \Im(w)},\\
\tilde f_\alpha:=&\frac{\partial}{\partial \zeta_{\alpha}}-i\left(z^TH_{\zeta_{\alpha}}(\zeta,\overline{\zeta})\overline{z}+\Re(\overline{z}^TS_{\zeta_{\alpha}}(\zeta,\overline{\zeta})\overline{z})\right)\frac{\partial}{\partial \Im(w)}+O(3)\frac{\partial}{\partial \Im(w)}.
\end{align}
The Levi form with respect to this frame and expansions is 
\begin{align}
\mathcal{L}^1=\left(
\begin{array}{cc}
H(\zeta,\overline{\zeta})+O(1) & H_{\overline{\zeta}}(\zeta,\overline{\zeta})\overline{z}+\overline{S_{\zeta}(\zeta,\overline{\zeta})}z+O(2) \\
z^TH_{\zeta}(\zeta,\overline{\zeta})+\overline{z}^TS_{\zeta}(\zeta,\overline{\zeta})+O(2)  & z^TH_{\zeta\overline{\zeta}}(\zeta,\overline{\zeta})\overline{z}+\Re(\overline{z}^TS_{\zeta\overline{\zeta}}(\zeta,\overline{\zeta})\overline{z})+O(3)  \\
\end{array}
\right).
\end{align}
Since $H(0,0)$ is nondegenerate, the matrix $H(\zeta,\overline{\zeta})+O(1)$ is invertible on some neighborhood of $0$. Writing $\mathcal{L}^1$ as block matrix $\left(\begin{array}{cc}
     A&B  \\
     C&D 
\end{array}\right)$ and using that $A$ is invertible, the rank of $\mathcal{L}^1$ is equivalent to the rank of the matrix
\[
\left(\begin{array}{cc}
     \mathrm{Id}&0  \\
     -CA^{-1}&\mathrm{Id} 
\end{array}\right)
\left(\begin{array}{cc}
     A&B  \\
     C&D 
\end{array}\right)
=
\left(\begin{array}{cc}
     A&B  \\
     0&D- CA^{-1}B
\end{array}\right).
\]
Thus rank condition $\mathrm{rank}(\mathcal{L}^{1})=s$ is thus equivalent to $D= CA^{-1}B$, which yields 
\begin{align}
 0=&z^TH_{\zeta\overline{\zeta}}(\zeta,\overline{\zeta})\overline{z}+\Re(\overline{z}^TS_{\zeta\overline{\zeta}}(\zeta,\overline{\zeta})\overline{z})+O(3)\\
 &-\left(z^TH_{\zeta}(\zeta,\overline{\zeta})+\overline{z}^TS_{\zeta}(\zeta,\overline{\zeta})+O(2)\right)(H(\zeta,\overline{\zeta})+O(1))^{-1}(H_{\overline{\zeta}}(\zeta,\overline{\zeta})\overline{z}+\overline{S_{\zeta}(\zeta,\overline{\zeta})}z+O(2))\\
 =&z^TH_{\zeta\overline{\zeta}}(\zeta,\overline{\zeta})\overline{z}+\Re(\overline{z}^TS_{\zeta\overline{\zeta}}(\zeta,\overline{\zeta})\overline{z})\\
 &-(z^TH_{\zeta}(\zeta,\overline{\zeta})+\overline{z}^TS_{\zeta}(\zeta,\overline{\zeta}))(H(\zeta,\overline{\zeta}))^{-1} (H_{\overline{\zeta}}(\zeta,\overline{\zeta})\overline{z}+\overline{S_{\zeta}(\zeta,\overline{\zeta})}z)+O(3).
\end{align}
Therefore, $H(\zeta,\overline{\zeta}),S(\zeta,\overline{\zeta})$ satisfy the system of PDE \eqref{rank condition 2} and the model hypersurface \eqref{gen def fun} has constant rank $s$ Levi form. The claim about $2$-nondegeneracy follows by the following proposition relating invariants of $2$-nondegenerate structures to their models.

\begin{proposition}\label{mod symbol the same}
    The uniformly 2-nondegenerate CR hypersurface $\Re(w)=P(z,\zeta,\overline{z},\overline{\zeta})+Q(z,\zeta,\overline{z},\overline{\zeta},\Im(w))$ and the 2-nondegenerate model $\Re(w)=P(z,\zeta,\overline{z},\overline{\zeta})$ have the same modified symbol at $0$ with respect to any normalization conditions.
\end{proposition}
\begin{proof}
We approach the computation of the modified symbols as in Section \ref{Key examples of modified symbol calculations}, but in the general setting the adapted frame $\tilde f_a,\tilde e_{\alpha}$ (introduced in the beginning of this section) is expressed
\[
\tilde f_a=f_a+O(2)\frac{\partial}{\partial \Im(w)}, 
\quad\quad\tilde e_{\alpha}=e_{\alpha}+O(2)\frac{\partial}{\partial z}+O(3)\frac{\partial}{\partial \Im(w)}
\]
relative to the adapted frame $f_a,e_{\alpha}$ for the 2-nondegenerate model (introduced in Section \ref{Key examples of modified symbol calculations}). Therefore, for computations along the leaf of the Levi kernel $\Im(w)=0,z=0$, the difference between the frames plays no role (as all the additional terms vanish on this leaf) in the computations done in Section \ref{Key examples of modified symbol calculations}. In particular, the vector fields and brackets $\tilde f_a$, $\tilde e_{\alpha}$, $[\tilde e_{\alpha},\tilde f_a]$ and $[\tilde e_{\alpha},\overline{\tilde f_a}]$ coincide $f_a$, $e_{\alpha}$, $[ e_{\alpha}, f_a]$ and $[ e_{\alpha},\overline{ f_a}]$ over this leaf and the claim follows.
\end{proof}

Let us return to the start and consider a general real analytic defining equation $\Re(w)=F(z,\overline{z},\Im(w))$ for some  real analytic hypersurface $M\subset\mathbb{C}^{n+1}$ that has constant rank $s$ Levi form. As usual one removes terms of polynomial order $1$ from $F$ and distinguishes the $z_1,\dots,z_s$ coordinates in which the polynomial order $2$ terms have the form $z^TH(0,0)\overline{z}$ with $H(0,0)$ invertible. Naming the remaining variables $\zeta$, we impose the weights \eqref{grading convention} on the variables. Consider the following expansion of the defining equation of $M$ sorted according to weighted order
\begin{align}
\Re(w)=W(\zeta,\overline{\zeta})+\Re(V(\zeta,\overline{\zeta})z)+Z(\zeta,\overline{\zeta})\Im(w)+z^TH(\zeta,\overline{\zeta})\overline{z}+\Re(\overline{z}^TS(\zeta,\overline{\zeta})\overline{z})+O(3),
\end{align}
where $W,Z$ are real functions of $\zeta,\overline{\zeta}$ variables and $V$ is a covector of $s$ complex functions of $\zeta,\overline{\zeta}$ variables.

To finish the proof of the theorem's first part, we need to show that there is a holomorphic change of coordinates making the $V,W,Z$ terms  vanish. Adding these additional $V,W,Z$ terms augments the previous $\mathcal{L}_1$ formula, producing the Levi form formula
\begin{align}
\mathcal{L}^1_{V,Z,W}:=\mathcal{L}_1+
\left(
\begin{array}{c:c}
\tilde{O}(V,Z) & \tilde{O}(V,Z,V_{\zeta},W_{\zeta},Z_{\zeta},V_{\overline{\zeta}}) \\\hdashline
 \tilde{O}(V,Z,V_{\zeta},W_{\zeta},Z_{\zeta},V_{\overline{\zeta}})  & 
\parbox{5cm}{\flushleft$W_{\zeta,\overline{\zeta}}+\Re(V_{\zeta,\overline{\zeta}}z)+Z_{\zeta,\overline{\zeta}}\Im(w) +\tilde{O}(V,Z,V_{\zeta},W_{\zeta},Z_{\zeta},V_{\overline{\zeta}})$}
\end{array}
\right),
\end{align}
where $\tilde{O}(R_1,\dots,R_i)$ stands in for functions depending on $R_1,\dots,R_i,\overline{R_1},\dots,\overline{R_i}$ that vanish whenever $R_1,\dots, R_i,\overline{R_1},\dots,\overline{R_i}$ are identically zero.

By the same arguments as above we can eliminate the last $n-s$ rows of the Levi form using Gaussian elimination. This provides a new PDE system of the form
\begin{align}
   W_{\zeta,\overline{\zeta}}&= \tilde{O}(V,Z,V_{\zeta},W_{\zeta},Z_{\zeta},V_{\overline{\zeta}})\\
   V_{\zeta,\overline{\zeta}}&= \tilde{O}(V,Z,V_{\zeta},W_{\zeta},Z_{\zeta},V_{\overline{\zeta}})\\
   Z_{\zeta,\overline{\zeta}}&= \tilde{O}(V,Z,V_{\zeta},W_{\zeta},Z_{\zeta},V_{\overline{\zeta}}).
\end{align}
Using the real analyticity assumption, let us express $V(\zeta,\overline{\zeta})=\sum_{i,j\geq 0} V_{ij}\zeta^i\overline{\zeta}^j$, $W(\zeta,\overline{\zeta})=\sum_{i,j\geq 0} W_{ij}\zeta^i\overline{\zeta}^j$ and $Z(\zeta,\overline{\zeta})=\sum_{i,j\geq 0} Z_{ij}\zeta^i\overline{\zeta}^j$ with $W_{ij}=\overline{W_{ji}}$ and $Z_{ij}=\overline{Z_{ji}}.$ Therefore, $V_{i,j}$, $W_{i,j}$, and $Z_{i,j}$ for $i,j>0$ are uniquely determined by $V_{k,l}$, $W_{k,l}$, and $Z_{k,l}$ for $k+l<i+j.$ Thus if $V_{k,0}=V_{0,l}=W_{k,0}=Z_{k,0}=0$ for all $k,l$, then $V$, $W$ and $Z$ will also vanish. Since it is well-known that there are holomorphic changes of coordinates that can achieve $V_{k,0}=V_{0,l}=W_{k,0}=Z_{k,0}=0$ for all $k,l$, we conclude that these transform the defining equation into the claimed form. Since achieving $Z_{k,0}=0$ is the main difficulty in these transformations, we note that it is accomplished via \cite[Theorem 4.2.6]{baouendi1999real} where an application of the implicit function theorem is used to bring the defining equation to so-called normal coordinates.

Now, let us prove the theorem's second statement about a biholomorphism $\phi: \mathbb{C}^{n+1}\to \mathbb{C}^{n+1}$ with $\phi(M)=M^\prime$ inducing the equivalence of models $\phi_0(M_0)=M_0^\prime$. Write
\[
\phi^{-1}(w,z,\zeta)=\big(h_{0}(\zeta)+h_{1}(z,\zeta)+O(2),g_{0}(\zeta)+O(1),O(0)\big),
\]
where $h_j,g_j,f_j$ denote holomorphic functions of weight $j$ with respect to the weighting \eqref{grading convention}. Note that $h_j(0)=g_j(0)=0$ because we assume $\phi(0)=0$. For any defining function $\rho$ of $M$, one has $\rho\circ\phi^{-1}(M^\prime)= \rho(M)=0$, so $\rho\circ\varphi^{-1}$ is a defining function of $M^\prime$. Therefore $\Re(w)=P(z,\zeta,\overline{z},\overline{\zeta})+Q(z,\zeta,\overline{z},\overline{\zeta},\Im(w))$ is transformed under $\phi$ to a defining equation of $M^\prime$ by substituting $\phi^{-1}(w,z,\zeta)$ for $(w,z,\zeta)$. Under this substitution, the left side transforms to $\Re(h_{0}(\zeta)+h_{1}(z,\zeta)+O(2))=\Re(w)+\Re(h_{0}(\zeta)+h_{1}(z,\zeta)-w+O(2))$ and we can move the second summand to right side of the defining equation, leaving just the term $\Re(w)$ on the left. After this rearranging, substituting  $\Re(w)=P^\prime(z,\zeta,\overline{z},\overline{\zeta})+Q^\prime(z,\zeta,\overline{z},\overline{\zeta},\Im(w))$ on the right side recovers the defining equation $\Re(w)=P^\prime(z,\zeta,\overline{z},\overline{\zeta})+Q^\prime(z,\zeta,\overline{z},\overline{\zeta},\Im(w))$, and we see that the terms of weighted degree $0$ and $1$ are determined by $h_0,h_1,g_0$ and sum to equal $0$. In particular, the terms of lowest polynomial degree in $\zeta$ and lowest weighted degree in the transformed equation are exactly the lowest polynomial degree terms in $h_{0}(\zeta)$, $\overline{h_{0}(\zeta)}$, and $\big(g_{0}(\zeta)\big)^TH(0,0)\overline{g_{0}(\zeta)}$, where $H(\zeta,\overline{\zeta})$ is given by expanding $P$ to the form on the right side of \eqref{gen def fun}. A combination of these leading terms must vanish because the transformed equation does not have weight $0$ terms, so by induction with respect to the polynomial degree $h_0=g_0=0$. After substituting $h_0=g_0=0$ into the transformed equation, the lowest weighted order terms are $\Re\big(h_{1}(z,\zeta)\big)$, so $h_1=0$ as well. Now, after substituting $h_0=g_0=h_1=0$, the terms of weighted homogeneity $2$ depend only on the $P$ and the weighted degree preserving component $\phi_0$ of $\phi$, so the weighted degree preserving component $\phi_0$ of $\phi$ transforms $\Re(w)=P(z,\zeta,\overline{z},\overline{\zeta})$ into $\Re(w)=P^\prime(z,\zeta,\overline{z},\overline{\zeta})$. 

Finally, let us prove the Theorem's third claim. Consider the filtration of $\mathfrak{hol}(M,0)$ determined by the lowest degree component of terms in the coordinate expansion of the germs of holomorphic infinitesimal symmetries at $0$ in $(w,z,\zeta)$, where we assign the weights $-2,-1,0$ to the partial derivatives $\partial_{w},\partial_{z},\partial_{\zeta}.$ With respect to these weights, infinitesimal symmetries decompose into a sum of weighted homogeneous vector fields with weights greater than $-3$. To show that $\dim(\mathfrak{hol}(M,0))\leq \dim(\mathfrak{hol}(M_0,0))$, we show that $\mathrm{gr}(\mathfrak{hol}(M,0))\subset \mathfrak{hol}(M_0,0)$, where $\mathrm{gr}(\mathfrak{hol}(M,0))$ denotes the associated grading of the filtration of $\mathfrak{hol}(M,0).$ In other words, we need to show that if $X\in \mathrm{hol}(M,0)$ is of the form $X=X_i+X_+$, where $X_i$ is the component of the lowest weighted degree $i$, then $X_i$ is a germ of holomorphic infinitesimal symmetry at $0$ of the corresponding $2$-nondegenerate model. Writing $\rho=\Re(w)-P(z,\zeta,\overline{z},\overline{\zeta})-Q(z,\zeta,\overline{z},\overline{\zeta},\Im(w))$ and $\rho_0=\Re(w)-P(z,\zeta,\overline{z},\overline{\zeta})$, the tangency condition becomes
\[
\left.d\rho(\Re(X))\right|_{\rho=0}=0
\]
and collecting the terms of lowest weight in 
\begin{align}
0=d\rho(\Re(X))|_{\rho=0}&=(d\rho_0-dQ)\left(\Re(X_i+X_+)\right)|_{\rho_0=-Q} \\
&=d\rho_0\left(\Re(X_i)\right)|_{\Re(w)=P+Q}+O(i+3)\\
&=d\rho_0\left(\Re(X_i)\right)|_{\Re(w)=P}+O(i+3)
\end{align}
with respect to \eqref{grading convention} gives the tangency condition $0=d\rho_0\left(\Re(X_i)\right)|_{\rho_0=0}$ for $M_0$ showing $X_i\in \mathfrak{hol}(M_0,0)$.

This completes the proof of Theorem \ref{model perturbations proposition}. 

\subsection{Proof of Theorem 1.2 and Theorem 1.4}\label{Theorem 1.1 proof}

Observe that the second claim of Theorem \ref{general formula for Cn} is a generalization of Theorem \ref{hypersurface realization approach 2}, wherein the matrix valued function $\mathbf{S}(\zeta)\in \mathfrak{g}_{0,2}$ is simply replaced by a more general function. This produces real analytic CR hypersurfaces, which have, by  construction, a transitive leaf space action, and which are, moreover, Levi degenerate with $\exp(\mathbf{S}(\zeta))$ as the maximal integral submanifold of their Levi kernel through $0$ (for the same reason occurring in the special case of Theorem \ref{hypersurface realization approach 2}). Thus, for the second claim of Theorem \ref{general formula for Cn}, it remains to show that the CR hypersurface is $2$-nondegenerate at $0$. Following the proof of Theorem \ref{hypersurface realization approach 2} in the $\mathbb{C}^{n+1}$ case, we obtain $H(0,0)=\mathbf{H}$ and $S_{\zeta_{\alpha}}(0,0)=\mathbf{H}^T\mathbf{S}_{\zeta_{\alpha}}(0)\mathbf{H}$ and $2$-nondegeneracy follows from \eqref{key example Xi mat},  since $S^{0,2}_{\alpha}(\tilde \phi(0))=\mathbf{S}_{\zeta_{\alpha}}(0)$ holds for the corresponding $\tilde \phi$ and all $\alpha$ and they are linearly independent by the assumptions of Theorem \ref{general formula for Cn}.

Let us now turn to proving the first part of Theorem \ref{general formula for Cn}. As a consequence of Proposition \ref{transkernel symmetry formula}, it suffices to prove it in the case of a smooth uniformly $2$-nondegenerate CR hypersurface $M$ in $\mathbb{C}^{n+1}$ of the form \eqref{gen def fun} having a subset of holomorphic infinitesimal symmetries acting locally transitively at $0$ on the space of leaves of the Levi kernel. In other words, to prove the Theorem \ref{smooth to analytic}.

The weighted homogeneity of the defining equation \eqref{gen def fun} implies that there is the infinitesimal symmetry given by restriction of the real part of the holomorphic vector field 
\begin{align}\label{Euler field}
2w\frac{\partial}{\partial w}+z_1\frac{\partial}{\partial z_1}+\dots +z_s\frac{\partial}{\partial z_s}
\end{align}
to the CR hypersurface. Since any finite-dimensional Lie algebra of vector fields containing this one is $\mathbb{Z}$-graded by decomposition into the eigenspaces of its adjoint action, the algebra of holomorphic infinitesimal symmetries is graded and contains the Heisenberg component $\mathfrak{g}_-$ if it acts transitively on the Levi leaf space. 

The image $\mathscr{I}\circ \tilde \phi(M_{\pi(p)})$ in $\mathcal{G}$ for a local section $\tilde \phi: M_{\pi(p)} \to \mathcal{F}$ can be identified with a submanifold in $\mathrm{CSp}(\mathbb{C}\mathfrak{g}_{-1})$ by identifying $\mathscr{I}(\tilde \phi(p))$ with the identity element in $\mathrm{CSp}(\mathbb{C}\mathfrak{g}_{-1})$. This submanifold moreover descends to a complex submanifold of the homogeneous space $\mathrm{CSp}(\mathbb{C}\mathfrak{g}_{-1})/\tilde{Q}$. The natural action of the structure group of $\mathcal{G}$ on fibers of the complexified contact distribution $\mathbb{C}\pi_*(\mathcal{H}\oplus \overline{\mathcal{H}})$ induces a transitive action of $\mathrm{CSp}(\mathbb{C}\mathfrak{g}_{-1})$ on the complex Lagrangian Grassmannian $\mathrm{LG}\big(\mathbb{C}\pi_*(\mathcal{H}_p\oplus \overline{\mathcal{H}}_p)\big)$, and since $\tilde{Q}$ is an isotropy subgroup of this action, the above identifications furthermore identify $\mathscr{I}\circ \tilde \phi(M_{\pi(p)})$ with a complex submanifold in $\mathrm{LG}\big(\pi_*(\mathcal{H}_p\oplus \overline{\mathcal{H}}_p)\big)$. An alternative way to describe this submanifold in $\mathrm{LG}\big(\mathbb{C}\pi_*(\mathcal{H}_p\oplus \overline{\mathcal{H}}_p)\big)$ is to identify each point $\psi\in\mathscr{I}\circ \tilde \phi(M_{\pi(p)})$ with the Lagrangian subspace $\psi^{-1}(\mathfrak{g}_{_{-1,1}})\subset \mathbb{C}\pi_*(\mathcal{H}\oplus \overline{\mathcal{H}})$. From this latter description, it is clear that the submanifold does not depend on the local section $\tilde\phi$.

\begin{proposition}\label{DLC equiv}
Suppose $M,M^{\prime}$ are two uniformly $2$-nondegenerate CR hypersurfaces in $\mathbb{C}^{n+1}$ of the form \eqref{gen def fun} with transitive leaf space actions by the restriction of symmetries to the hypersurface. Let $M_{\pi(p)}$ and $M^{\prime}_{\pi(p^{\prime})}$ be such that $\mathscr{I}(\tilde \phi(p))=\mathscr{I}(\tilde \phi^\prime(p^\prime))$ and the images of $M_{\pi(p)}$ and $M^{\prime}_{\pi(p^{\prime})}$ in $\mathrm{CSp}(\mathbb{C}\mathfrak{g}_{-1})/\tilde{Q}$ overlap on a neighborhood of $\mathscr{I}(\tilde \phi(p))$ for some local sections $\tilde \phi,\tilde \phi^{\prime}$ satisfying $\sigma_{\tilde \phi}=\sigma_{\tilde \phi^{\prime}}$. Then $M$ and $M^{\prime}$ are locally CR equivalent around $p$ and $p^\prime$.
\end{proposition}
\begin{proof}

In \cite{sykes2023geometry}, it is shown that every CR structure induces a dynamical Legendrian contact (DLC) structure on the Levi leaf space (c.f. \cite[Definition 2.4]{sykes2023geometry}).  The map $\psi\mapsto \big(\psi^{-1}(\mathfrak{g}_{_{-1,1}}),\psi^{-1}(\mathfrak{g}_{_{-1,-1}})\big)$ identifies $\mathscr{I}\circ \tilde \phi(M_{\pi(p)})$ with a fiber of the DLC structure associated with the CR structure $(M,\mathcal{H})$. Thus, since the present considerations are all local, and the local symmetry group's Heisenberg component $\exp(\mathfrak{g}_{-})$ identifies all fibers of the DLC structure via its transitive action on the Levi leaf space, the hypothesized overlap implies that the DLC structures associated with $M$ and $M^\prime$ are equivalent.

The proposition thus follows if one can show that the CR structures are uniquely determined by their DLC structures. It was shown in \cite{sykes2023geometry} that this indeed happens when the bigraded symbols are \emph{recoverable}. We will, however, show now that this also happens in general for uniformly $2$-nondegenerate CR hypersurfaces in $\mathbb{C}^{n+1}$ of the form \eqref{gen def fun} with transitive leaf space actions, even though their symbols are not always recoverable. Let us recall from \cite[Proposition 2.6]{sykes2023geometry}, symbols are \emph{recoverable} if and only if the first prolongation 
\begin{align}\label{first prol for recoverability}
\left(\left.\overline{\iota\left(\overline{\mathcal{K}}_p\right)}\right|_{\mathfrak{g}_{-1,1}(p)}\right)_{(1)}:=
\left\{f:\mathfrak{g}_{-1,1}(p)\to\mathcal{K}_p\,\left|\,\overline{\iota\circ f(v)}(w)=\overline{\iota\circ f(w)}(v),\,\forall\,v,w\in \mathfrak{g}_{-1,1}(p)\right.\right\}
\end{align}
of the subspace $\left.\overline{\iota\left(\overline{\mathcal{K}}_p\right)}\right|_{\mathfrak{g}_{-1,1}(p)}\subset\mathrm{Hom}\big(\mathfrak{g}_{-1,1}(p),\mathfrak{g}_{-1,-1}(p)\big)$ vanishes. The equivalence of the DLC structures provides a local diffeomorphism $M\to M'$ that allows one to pullback $\mathcal{H}'$ to $M$ and compare it with $\mathcal{H}$. Since these two CR hypersurfaces have the same modified symbol we see that in the adapted frame $\phi=(g,f_1,\dots,f_{s},\overline{f_1},\dots,\overline{f_{s}},e_1,\dots,e_{n-s},\overline{e_1},\dots,\overline{e_{n-s}})$ from Section \ref{Key examples of modified symbol calculations}, there are smooth functions $G_{i,\alpha}$ such that 
\[
\mathcal{H}^{\prime}={\rm span}\left\{f_1+\sum_{j=1}^{n-s}G_{1,\alpha}\overline{e_\alpha},\dots,f_s+\sum_{j=1}^{n-s}G_{s,\alpha}\overline{e_\alpha},e_1,\dots,e_{n-s}\right\},
\] 
where $G_{i,\alpha}$ are components of a linear map in $\left(\left.\overline{\iota\left(\overline{\mathcal{K}}\right)}\right|_{\mathfrak{g}_{-1,1}}\right)_{(1)}$ according to \cite{sykes2023geometry}. 

Now consider the one parameter family of symmetries $\rho_t$ generated by the real part of \eqref{Euler field}. Since these are symmetries, $\mathcal{H}=\rho_t^*\mathcal{H}$ and $\mathcal{H}^\prime=\rho_t^*\mathcal{H}^\prime$, but on the other hand $\rho_t^*f_i(q)=e^tf_i(q), \rho_t^*G_{i,\alpha}(q)=G_{i,\alpha}(q)$ and $\rho_t^*e_\alpha(q)=e_\alpha(q)$  hold for every $i,\alpha$ and every $q\in M$ such that $\rho_t(q)=q.$ As every element of $M_{\pi(0)}$ is fixed by  $\rho_t$, we can conclude that all $G_{i,\alpha}$ vanish on $M_{\pi(0)}$. Consequently, we can consider also $\rho_t$ conjugated by the transitive leaf space action to show  that $G_{i,\alpha}$ vanishes at all points and $\mathcal{H}=\mathcal{H}^{\prime}$.
\end{proof}

This concludes the proof of Theorem \ref{smooth to analytic} and Theorem \ref{general formula for Cn} because the exponential map from $\mathfrak{csp}(\mathbb{C}\mathfrak{g}_{-1})_{0,2}$ to $\mathrm{CSp}(\mathbb{C}\mathfrak{g}_{-1})/\tilde{Q}$ defines local holomorphic coordinates around $\mathrm{exp}(0)$, which implies that the images discussed in Proposition \ref{DLC equiv} can all be described via the matrix $\mathbf{S}(\zeta)$ from the Theorem \ref{general formula for Cn}.

Finally, we can combine the Propositions \ref{transkernel symmetry formula} and \ref{DLC equiv} to prove:

\begin{corollary}
 Suppose $M,M^{\prime}$ are two $2$-nondegenerate models. Let $M_{\pi(p)}$ and $M^{\prime}_{\pi(p^{\prime})}$ be such that $\mathscr{I}(\tilde \phi(p))=\mathscr{I}(\tilde \phi^\prime(p^\prime))$ and the images of $M_{\pi(p)}$ and $M^{\prime}_{\pi(p^{\prime})}$ in $\mathrm{CSp}(\mathbb{C}\mathfrak{g}_{-1})/\tilde{Q}$ overlap on a neighborhood of $\mathscr{I}(\tilde \phi(p))$ for some local sections $\tilde \phi,\tilde \phi^{\prime}$ satisfying $\sigma_{\tilde \phi}=\sigma_{\tilde \phi^{\prime}}$. Then $M$ and $M^{\prime}$ are locally CR equivalent around $p$ and $p^\prime$.   
\end{corollary}

\section{The equivalence problem for 2-nondegenerate models}\label{The equivalence problem for models}

Now, we introduce a (partial) normal form construction for the $2$-nondegenerate CR hypersurfaces with real analytic defining equation \eqref{gen def fun}, which distinguishes a finite dimensional space of representatives within each infinite dimensional equivalence class and reduces the equivalence problem for such hypersurfaces to the finite dimensional linear algebra of CR symbol equivalence (a problem of simultaneously normalizing Hermitian and symmetric forms). This will have two uses:

Firstly, we show how to transform the defining equation \eqref{gen def fun} to a form satisfying the condition of Theorem \ref{general formula for Cn}, offering an alternative proof of Theorem \ref{general formula for Cn} and extending it to the case with constant rank $s$ Levi form (i.e., without $2$-nondegeneracy).
Secondly, we identify all the CR invariants discussed in the article in the (partial) normal form.

The first step in this construction is to kill the non-constant holomorphic part of $H(\zeta,\overline{\zeta})$ and the full anti-holomorphic part of $S(\zeta,\overline{\zeta})$.

\begin{lemma}\label{normalizing H terms}
There is a locally biholomorphic (at $0$) change of coordinates bringing \eqref{gen def fun} to an equation of the form
\begin{align}\label{normalized H eqn}
\Re(w)=z^T\mathbf{H}\overline{z}+\Re\left(\overline{z}^T\mathbf{H}^T\mathbf{S}(\zeta)\mathbf{H}\overline{z}\right)+O(|\zeta|^2),
\end{align}
where $\mathbf{H}:={H}(0,0)$, $\mathbf{S}(\zeta)=(\mathbf{H}^T)^{-1}(S(\zeta,0)-S(0,0))\mathbf{H}^{-1}$, and $O(|\zeta|^2)$ stands in for terms divisible by $\zeta_{\alpha}\overline{\zeta_\beta}$ for various $(\alpha,\beta)$. Moreover, for CR hypersurfaces with constant rank $s$ Levi form, the $O(|\zeta|^2)$ terms are completely determined by $\mathbf{H}$ and $\mathbf{S}(\zeta).$
\end{lemma}
\begin{proof}
To label some components in the series expansion of $H(\zeta,\overline{\zeta})$, let $A(\zeta)$, $B(\zeta,\overline{\zeta})$, and $\mathbf{H}$ satisfy $A(0)=0$, $B(\zeta,\overline{\zeta})$ has no purely holomorphic terms, and
\[
H(\zeta,\overline{\zeta})= B(\zeta,\overline{\zeta})+\mathbf{H}+\left(A(\zeta)\right)^T\mathbf{H}+\mathbf{H}\overline{A(\zeta)}.
\]
Accordingly,
\[
z^TH(\zeta,\overline{\zeta})\overline{z}=\left(\Phi(z)\right)^T\mathbf{H}\overline{\Phi(z)}+z^T\left(B(\zeta,\overline{\zeta})-\left(A(\zeta)\right)^T\mathbf{H}\overline{A(\zeta)}\right)\overline{z}.
\]
where $\Phi:\mathbb{C}^s\to\mathbb{C}^s$ is the local biholomorphism $\Phi(z):=\big(\mathrm{Id}+A(\zeta)\big)z$. The local biholomorphism $(w,z,\zeta)\mapsto\left(w,\Phi^{-1}(z),\zeta\right)$ thus transforms \eqref{gen def fun} into \[\Re(w)=z^Th(\zeta,\overline{\zeta})\overline{z}+z^T\mathbf{H}\overline{z}+\Re(\overline{z}^Ts(\zeta,\overline{\zeta})\overline{z})\]
with 
\[
z^Th(\zeta,\overline{\zeta})\overline{z}=\left(\Phi^{-1}(z)\right)^{T}\left(B(\zeta,\overline{\zeta})-\left(A(\zeta)\right)^T\mathbf{H}\overline{A(\zeta)}\right)\overline{\Phi^{-1}(z)}
\]
and $\overline{z}^Ts(\zeta,\overline{\zeta})\overline{z}=\left(\overline{\Phi^{-1}(z)}\right)^T{S}(\zeta,\overline{\zeta})\overline{\Phi^{-1}(z)}$. Notice that $h(\zeta,\overline{\zeta})$ does not contain any purely harmonic terms in $\zeta$, and the purely holomorphic non-constant part of $s(\zeta,\overline{\zeta})$ equals that of $S(\zeta,\overline{\zeta})$, i.e., $s(\zeta,0)=\mathbf{H}^T\mathbf{S}(\zeta)\mathbf{H}$.

With a change of coordinates that fixes $(z,\zeta)$ but transforms $w$, one can eliminate pluriharmonic terms on the right side, i.e., remove terms in $s(\zeta,\overline{\zeta})$ that are anti-holomorphic in $\zeta$. This provides the claimed formula. To prove that the $O(|\zeta|^2)$ part is determined by $\mathbf{H}$ and $\mathbf{S}(\zeta)$, we need to consider the PDE system \eqref{rank condition 2}. Consider iteratively solving \eqref{rank condition 2} for just the terms in a given polynomial order in $\zeta,\overline{\zeta}$ starting from the lowest order. By induction on this order one observes that the solution is indeed completely determined by $\mathbf{H}$ and $\mathbf{S}(\zeta)$. 
\end{proof}

Observe that the defining equations from Theorem \ref{general formula for Cn} are already in the normal form \eqref{normalized H eqn}, which allows us to provide an explicit formula for the $O(|\zeta|^2)$ terms and extend the result to constant rank Levi form case.

\begin{proposition}\label{general formula to main theorem formula}
A real analytic CR hypersurface with constant rank $s$ Levi form given by defining equation \eqref{gen def fun} is locally biholomorphic (at $0$) with the CR hypersurface in the normal form \eqref{normalized H eqn} given by the defining equation
\begin{align}
\Re(w)&=z^T\mathbf{H}\overline{z}+\Re\left(\overline{z}^T\mathbf{H}^T\mathbf{S}(\zeta)\mathbf{H}\overline{z}\right)\\
&+z^T\frac12(\mathbf{H}((\mathrm{Id}-\overline{\mathbf{S}}\mathbf{H}^T\mathbf{S}(\zeta)\mathbf{H})^{-1}-\mathrm{Id})+((\mathrm{Id}-\mathbf{H}\overline{\mathbf{S}(\zeta)}\mathbf{H}^T\mathbf{S}(\zeta))^{-1}-\mathrm{Id})\mathbf{H})\overline{z}\\
&+\Re\left(\overline{z}^T(\mathbf{H}^T(\mathrm{Id}-\mathbf{S}(\zeta)\mathbf{H}\overline{\mathbf{S}(\zeta)}\mathbf{H}^T)^{-1}\mathbf{S}(\zeta)\mathbf{H}-\mathbf{H}^T\mathbf{S}(\zeta)\mathbf{H})\overline{z}\right),
\end{align}
where $\mathbf{H}:={H}(0,0)$ and $\mathbf{S}(\zeta)=(\mathbf{H}^T)^{-1}(S(\zeta,0)-S(0,0))\mathbf{H}^{-1}$.
\end{proposition}
\begin{proof}
Clearly, the defining equations from Theorem \ref{general formula for Cn} can be reordered into the claimed form and this can be generalized to arbitrary $\mathbf{S}(\zeta)$. Further, it is a simple computation that this is in the normal form \eqref{normalized H eqn}. Then invoking the unique determination of the $O(|\zeta|^2)$ part of Lemma \ref{normalizing H terms} and the freedom in choice of $\mathbf{H}$ and $\mathbf{S}(\zeta)$, we can conclude the claim of the proposition.
\end{proof}

Let us provide explicitly the CR invariants for the (partial) normal form and their dependence on $\mathbf{H}$ and $\mathbf{S}(\zeta)$:

Following the computations in proof of Theorem \ref{hypersurface realization approach 2}, we obtain the description of bigraded symbol.

\begin{corollary}
    The matrices $\mathbf{H}$ and $S^{0,2}_\alpha:=(\mathbf{H}^T)^{-1}S_{\zeta_{\alpha}}(0,0)\mathbf{H}^{-1}=\mathbf{S}_{\zeta_\alpha}(0) \in \mathfrak{csp}(\mathbb{C}\mathfrak{g}_{-1})_{0,2}$ represent the bigraded symbol of $2$-nondegenerate model according to Lemma \ref{symbol representation lemma}.
\end{corollary}

Further, to extend the bigraded symbol at hand to a modified symbol, let us fix a normalization condition $\big(O_{0,-2}^N(\tilde \phi(0)),O_{0,2}^N(\tilde \phi(0))\big)$ with the properties from Lemma \ref{key examples prop} for which the the section $\tilde \phi$ from proof of Theorem \ref{hypersurface realization approach 2} is normalized at $0.$ 
\begin{corollary}
 The modified symbol with respect to $\big(O_{0,-2}^N(\tilde \phi(0)),O_{0,2}^N(\tilde \phi(0))\big)$ is represented according to Lemma \ref{symbol representation lemma} by $\Omega_{\alpha}\in \mathfrak{csp}(\mathbb{C}\mathfrak{g}_{-1})_{0,0}$ such that
\begin{align}
\left[\Omega_{\alpha},\mathbf{H}\overline{S^{0,2}_{\beta}}\mathbf{H}^T\right]&=0,\\
(\mathbf{H}^T)^{-1}S_{\zeta_{\alpha}\overline{\zeta_{\beta}}}(0,0)\mathbf{H}^{-1}&-\left[\Omega_{\alpha},(\mathbf{H}^T)^{-1}S_{\zeta_{\beta}}(0,0)\mathbf{H}^{-1}\right]=\label{mod symb comp 6}\\
\mathbf{S}_{\zeta_\alpha\zeta_\beta}(0)-\left[\Omega_{\alpha},S^{0,2}_{\beta}\right]&= O_{0,2}^N(\tilde \phi(0))(e_{\alpha},S^{0,2}_{\beta})
\end{align}
for all $\alpha,\beta$ (note that we are taking brackets between elements of $\mathfrak{csp}(\mathbb{C}\mathfrak{g}_{-1})_{0,0}$ and $\mathfrak{csp}(\mathbb{C}\mathfrak{g}_{-1})_{0,\pm 2}$).
\end{corollary}

In other words, $\mathbf{S}_{\zeta_\alpha\zeta_\beta}(0)$ gets decomposed into an obstruction to first order constancy and the $\Omega_{\alpha}$-part of the modified symbol according to Lemma \ref{symbol representation lemma}.

Following  Remark \ref{action on symbol rem}, we can apply a transformation of the form
\begin{align}\label{linear in z}
(w,z,\zeta)\mapsto(w,z^TU^T,\zeta)\mbox{ for some }U\in \mathrm{GL}(s,\mathbb{C})
\end{align}
and bring the matrices $(\mathbf{H},\mathrm{span}\left\{S^{0,2}_1,\ldots,\allowbreak S^{0,2}_{n-s}\right\})$ representing the the bigraded symbol at $0$ to 
\[
\left(U^T\mathbf{H}\overline{U},\mathrm{span}\left\{U^{-1}S^{0,2}_1 (U^{-1})^T,\ldots,U^{-1}S^{0,2}_{n-s}(U^{-1})^T\right\}\right).
\]
The transformation is unique up to elements $U$ in the intersection of the groups $\mathrm{U}(\mathbf{H})$ preserving $\mathbf{H}$ and $G_{0,0}$ preserving the span of $H$ and $\{\Xi_1,\ldots, \Xi_{n-s}\}$. So we obtain only a partial normal form defined up to a $\mathrm{U}(\mathbf{H})\cap G_{0,0}$ orbit. Normalizing $(\mathbf{H},\mathrm{span}\left\{S^{0,2}_1,\ldots,\allowbreak S^{0,2}_{n-s}\right\})$ remains an open problem, so presently we aim only to normalize the Theorem \ref{general formula for Cn} hypersurfaces relative to its bigraded symbol representation, again with freedom in the action of $\mathrm{U}(\mathbf{H})\cap G_{0,0}$. The normalization furthermore depends on a convention of how one orders entries in symmetric matrices, and we will adopt the  convention ordering such matrix entries as
\begin{align}\label{sym mat ordering}
\begin{cases}
(j,k)\prec (j^\prime,k^\prime)&\mbox{ if }\min\{j,k\}<\min\{j^\prime,k^\prime\}\\
(j,k)\prec (j^\prime,k^\prime)&\mbox{ if }\min\{j,k\}=\min\{j^\prime,k^\prime\}\mbox{ and }\max\{j,k\}<\max\{j^\prime,k^\prime\}.
\end{cases}
\end{align}

\begin{theorem}\label{normal form for models}
    For a uniformly $2$-nondegenerate real analytic CR hypersurface $M$ given by \eqref{gen def fun}, let $(\mathbf{H},\mathrm{span}\left\{S^{0,2}_1,\ldots, S^{0,2}_{n-s}\right\})$ with $\mathbf{H}=H(0,0)$ and $S^{0,2}_\alpha:=(\mathbf{H}^T)^{-1}S_{\zeta_{\alpha}}(0,0)\mathbf{H}^{-1}$ be the matrices representing its normalized bigraded symbol at $0$. Let $\left((j_1,k_1),\ldots,\allowbreak (j_{n-s},k_{n-s})\right)$ be the first lexicographicly ordered $(n-s)$-tuple of such positions with respect to the ordering \eqref{sym mat ordering} such that 
    \[
    (\zeta_1,\ldots,\zeta_{n-s})\mapsto \left(
    \sum_{\alpha=1}^{n-s}(S^{0,2}_\alpha)_{(j_1,k_1)}\zeta_{\alpha},
    \ldots,
    \sum_{\alpha=1}^{n-s}(S^{0,2}_\alpha)_{(j_{n-s},k_{n-s})}\zeta_{\alpha}
    \right)
    \]
    is injective near $0$, which exists by $2$-nondegeneracy.

    There is a change of coordinates that transforms $M$ to a hypersurface given by \eqref{gen def fun}, \eqref{H from S2 formula}, and \eqref{S from S2 formula} with $\mathbf{H}={H}(0,0)$ and $\mathbf{S}(\zeta)$ satisfying $\left(\mathbf{S}(\zeta)\right)_{(j_\alpha,k_\alpha)}=\zeta_\alpha$ 
    for all $\alpha=1,\ldots, n-s$ and $\mathrm{span}\{S^{0,2}_{\alpha}\}=\mathrm{span}\{\mathbf{S}_{\zeta_\alpha}(0)\}$. The transformed defining equation of $M$ is unique up to the action of $\mathrm{U}(\mathbf{H})\cap G_{0,0}$, acting by transformations of the form $(w,z,\zeta)\mapsto(w,z^TU^T,\zeta+g_U(\zeta))$, where the function $g_U$ is uniquely determined by each $U\in \mathrm{U}(\mathbf{H})\cap G_{0,0}$.
\end{theorem}
\begin{proof}
   By \eqref{general formula to main theorem formula}, it will suffice to address the special case where $M$ is given by \eqref{gen def fun}, \eqref{H from S2 formula}, and \eqref{S from S2 formula}, so let us make this assumption. The indices  $(j_{\alpha},k_{\alpha})$ were chosen such that
    \begin{align}\label{zeta transform for normalization}
    \zeta\mapsto \left(\left((\mathbf{H}^T)^{-1}S(\zeta,0)(\mathbf{H})^{-1}\right)_{(j_1,k_1)},\ldots, \left((\mathbf{H}^T)^{-1}S(\zeta,0)(\mathbf{H})^{-1}\right)_{(j_{n-s},k_{n-s})}
    \right)
    \end{align}
    is local biholomorphism at $0$. Consider the holomorphic changes of coordinates fixing $0$ that preserve the general form of \eqref{gen def fun}, \eqref{H from S2 formula}, and \eqref{S from S2 formula}.

Let us first consider the special case with transformations of the form $(w,z,\zeta)\mapsto(w,z^TU^T,\zeta+g(\zeta))$ where  $g$ is a holomorphic function and $U$ is an invertible matrix-valued holomorphic function of $\zeta$.
Applying the transformation with $U=\mathrm{Id}$ and $\zeta\mapsto \zeta+g(\zeta)$ being the inverse of \eqref{zeta transform for normalization} achieves the theorem's normalization, so it remains to consider what subsequent transformations preserve this normalization. If $U\in\mathrm{GL}(s,\mathbb{C})$ is not contained in $\mathrm{U}(\mathbf{H})\cap G_{0,0}$, then either $\mathbf{H}$ or $\mathrm{span}\{S^{0,2}_1,\ldots, S^{0,2}_{n-s}\}$ is not preserved under the induced action of $(w,z,\zeta)\mapsto(w,z^TU^T,\zeta+g(\zeta))$ for any $g$. Hence preserving the normalization requires $U\in \mathrm{U}(\mathbf{H})\cap G_{0,0}$.

This is also sufficient because for each $U\in \mathrm{U}(\mathbf{H})\cap G_{0,0}$ there is a unique $g_U(\zeta)$ such that $(w,z,\zeta)\mapsto(w,z^TU^T,\zeta+g_U(\zeta))$ preserves $M$. Indeed, existence of $g_U$ follows from the observation that the induced action of $(w,z,\zeta)\mapsto(w,z^TU^T,\zeta)$ on $\mathrm{span}\left\{S^{0,2}_1,\ldots, S^{0,2}_{n-s}\right\}$ is a linear isomorphism, so one can take $\zeta\mapsto \zeta+g_U(\zeta)$ to be its inverse. Uniqueness indeed holds because if $U=\mathrm{Id}$ then $g_U(\zeta)=0$ because otherwise the transformed defining equation will fail $\left(\mathbf{S}(\zeta)\right)_{(j_\alpha,k_\alpha)}=\zeta_\alpha$ for some $\alpha$.

Turning to the general case, consider transformations of the form $(w,z,\zeta)\mapsto(w+h(w,z,\zeta),z+f(w,z,\zeta),\zeta+g(w,z,\zeta))$ for holomorphic functions $h$, $f$, and $g$. Suppose we fixed a defining equation for $M$ satisfying all of the theorem's normalization conditions, and let us assume this general transformation produces a (possibly new) defining equation also satisfying all of the theorem's conditions.

If $f(w,z,\zeta)$ has nonzero terms that depend only on $\zeta$ then the transformed \eqref{gen def fun} will too. Therefore terms in $f$ have weight at least $1$ with respect to \eqref{grading convention}, and similarly one finds that terms in $h$ have weight at least $2$. 

With these simplifications applied, the terms in $h$, $f$, and $g$ of weights greater than $2$, $1$, and $0$ respectively appear only with the terms of the transformed \eqref{gen def fun} that have weights strictly greater than $2$. Since the transformation preserves the general form of \eqref{gen def fun}, these higher weight terms must vanish, so the transformed \eqref{gen def fun} is entirely characterized by the components of $h$, $f$, and $g$ of weights equal to $2$, $1$, and $0$ respectively. 
Hence, it suffices to consider weighted homogeneous transformations with $\mathrm{wt}(h)=2$, $\mathrm{wt}(f)=1$, and $\mathrm{wt}(g)=0$. 

We have nearly reduced to the previously considered special case, except with $h$ possibly nonzero and with $U$ now a function of $\zeta$. By composing the general transformation with one from the previously considered special case, we can also assume without loss of generality that $U(0)=\mathrm{Id}$. The induced action of the transformation $(w,z,\zeta)\mapsto\big(w+h,z^TU^T,\zeta\big)$ on the $S(\zeta,\overline{\zeta})$ part of the defining equation (obtained by writing it in the form \eqref{gen def fun}) has trivial action on its holomorphic part $S(\zeta,0)$. It therefore has trivial action on the entire defining equation because \eqref{gen def fun} with \eqref{H from S2 formula} and \eqref{S from S2 formula} is fully determined by $\mathbf{H}$ and $\mathbf{S}(\zeta)=(\mathbf{H}^T)^{-1}(S(\zeta,0)-S(0,0))\mathbf{H}^{-1}$. Hence, if $(w,z,\zeta)\mapsto\big(w+h,z^TU^T,\zeta+g\big)$ preserves all of the theorem's normalization conditions then $g=0$ because otherwise $\left(\mathbf{S}(\zeta)\right)_{(j_\alpha,k_\alpha)}=\zeta_\alpha$ would fail for some $\alpha$. 

This completes the proof because by assuming our general transformation preserves all of the theorem's normalization conditions we found that it can be composed with a transformation from the previously considered special case to obtain a symmetry of the defining equation.
\end{proof}

 \newcommand{\noop}[1]{}

\end{document}